\pdfoutput=1
\RequirePackage{ifpdf}
\ifpdf % We are running pdfTeX in pdf mode
\documentclass[pdftex]{sigma}
\else
\documentclass{sigma}
\fi

\numberwithin{equation}{section}

\newtheorem{Theorem}{Theorem}[section]
\newtheorem{pr}[Theorem]{Proposition}
\newtheorem{cor}[Theorem]{Corollary}

\newtheorem{con}[Theorem]{Conjecture}
\newtheorem{lem}[Theorem]{Lemma}

\newtheorem{prb}[Theorem]{Problem}

{\theoremstyle{definition}
\newtheorem{de}[Theorem]{Definition}
\newtheorem{rem}[Theorem]{Remark}

\newtheorem{ex}[Theorem]{Example}

\newtheorem{exer}[Theorem]{Exercise}

\newtheorem{comm}[Theorem]{Comments}}

\def\s{{\mathfrak S}}

\def\R{\mathbb{R}}
\def\N{\mathbb{N}}
\def\Z{\mathbb{Z}}

\def\Q{\mathbb{Q}}

\begin{document}

\allowdisplaybreaks

\newcommand{\arXivNumber}{1501.07337}

\renewcommand{\PaperNumber}{034}

\FirstPageHeading

\ShortArticleName{Notes on Schubert, Grothendieck and Key Polynomials}

\ArticleName{Notes on Schubert, Grothendieck\\ and Key Polynomials}

\Author{Anatol N.~{KIRILLOV}~$^{\dag\ddag\S}$}

\AuthorNameForHeading{A.N.~Kirillov}

\Address{$^\dag$~Research Institute of Mathematical Sciences (RIMS), Kyoto, Sakyo-ku 606-8502, Japan}
\EmailD{\href{mailto:kirillov@kurims.kyoto-u.ac.jp}{kirillov@kurims.kyoto-u.ac.jp}}
\URLaddressD{\url{http://www.kurims.kyoto-u.ac.jp/~kirillov/}}

\Address{$^\ddag$~The Kavli Institute for the Physics and Mathematics of the Universe (IPMU),\\
\hphantom{$^\ddag$}~5-1-5 Kashiwanoha, Kashiwa, 277-8583, Japan}

\Address{$^\S$~Department of Mathematics, National Research University Higher School of Economics, \\
\hphantom{$^\S$}~7 Vavilova Str., 117312, Moscow, Russia}

\ArticleDates{Received March 26, 2015, in f\/inal form February 28, 2016; Published online March 29, 2016}

\Abstract{We introduce common generalization of (double) Schubert, Grothendieck, Demazure, dual and stable Grothendieck polynomials, and Di~Francesco--Zinn-Justin polynomials. Our approach is based on the study of algebraic and combinatorial properties of the reduced rectangular plactic algebra and associated Cauchy kernels.}


\Classification{05E05; 05E10; 05A19}

\vspace{-3mm}

\begin{flushright}
\it To the memory of Alexander Grothendieck (1928--2014)
\end{flushright}

\vspace{-7mm}

{\small\renewcommand{\baselinestretch}{0.9} \tableofcontents}

\subsection*{Extended abstract} We introduce
certain f\/inite-dimensional algebras denoted by ${\cal{PC}}_n$ and
${\cal{PF}}_{n,m}$ which are certain quotients of the plactic algebra
${\cal{P}}_{n}$, which had been introduced by A.~Lascoux and
\mbox{M.-P.}~Sch\"{u}t\-zen\-berger \cite{LS1}. We show that
$\dim({\cal{PF}}_{n,k})$ is equal to the number of symmetric plane partitions
 f\/itting inside the box $n \times k \times k$,
$\dim({\cal{PC}}_{n})$ is equal to the number of alternating sign matrices of
size $n \times n$, moreover,
\begin{gather*}
\dim({\cal{PF}}_{n,n}) = {\rm TSPP}(n+1) \times {\rm TSSCPP}(n),\\
\dim({\cal{PF}}_{n,n+1})= {\rm TSPP}(n+1) \times {\rm TSSCPP}(n+1),\\
\dim({\cal{PF}}_{n+2,n})= \dim({\cal{PF}}_{n,n+1}),\qquad
\dim({\cal{PF}}_{n+3,n})={\frac{1}{2}} \dim({\cal{PF}}_{n+1,n+1}),
\end{gather*}
and study decomposition of the Cauchy kernels corresponding to the algebras
${\cal{PC}}_n$ and ${\cal{PF}}_{n,m}$;
as well as introduce polynomials which are common generalizations of the
(double) Schubert, $\beta$-Grothendieck, Demazure (known also as {\it key polynomials}), (plactic) key-Grothendieck, (plactic) Stanley and stable $\beta$-Grothendieck polynomials. Using a family of the Hecke type divided dif\/ference operators
we introduce polynomials which are common generalizations of the Schubert,
$\beta$-Grothendieck, dual $\beta$-Grothendieck, $\beta$-Demazure--Grothendieck, and Di Francesco--Zinn-Justin polynomials.
We also introduce and study some properties of the double af\/f\/ine
nilCoxeter algebras and related polynomials,
put forward a $q$-deformed version of the Knuth relations and
plactic algebra.

\section{Introduction}

The Grothendieck polynomials had been introduced by A.~Lascoux and M.-P.~Sch\"{u}t\-zen\-berger in~\cite{LS6} and studied in detail in~\cite{L33}. There are two equivalent versions of the Grothendieck polynomials depending on a choice
of a basis in the Grothendieck ring $K^{\star}({\cal{F}}l_n)$ of the
complete f\/lag variety ${\cal{F}}l_n$. The basis $\{\exp(\xi_{1}), \ldots,\exp(\xi_{n}) \}$ in $K^{*}({\cal{F}}l_n)$ is one choice, and another choice is the
basis $\{1-\exp(- {\xi}_{j}), \,1 \le j \le n\}$, where $\{{\xi}_{j},
\, 1 \le j \le n \}$ denote the Chern classes of the tautological linear
bundles $L_{j}$ over the f\/lag variety ${\cal{F}}l_n$. In the present paper we
use the basis in a deformed Grothendieck ring $K^{*, \beta}({\cal{F}}_n)$ of
the f\/lag variety ${\cal{F}}l_n$ generated by the set of elements $\{x_{i} =
x_{i}^{(\beta)} = 1-\exp(\beta {\xi}_{i}),\, i=1, \ldots, n \}$. This basis has
been introduced and used for construction of the $\beta$-Grothendieck
polynomials in~\cite{FK1,FK}.

A basis in the classical Grothendieck ring of the f\/lag variety in question
corresponds to the choice $\beta = -1$. For arbitrary $\beta$ the ring
generated by the elements $\big\{x_{i}^{(\beta)},\, 1 \le i \le n \big\}$ has been
identif\/ied with the Grothendieck ring corresponding to the generalized
cohomology theory associated with the multiplicative formal group law
$F(x,y)=x+y+\beta x y$, see \cite{Hu}. The Grothendieck polynomials
corresponding to the classical $K$-theory ring $K^{\star}({\cal{F}}l_{n})$,
 i.e., the case $\beta = -1$, had been studied in depth by A.~Lascoux and
M.-P.~Sch\"{u}t\-zen\-berger in \cite{LS7}. The $\beta$-Grothendieck
polynomials has been studied in~\cite{FK1,FK2,Hu}.

The {\it plactic monoid} over a f\/inite totally ordered set $\mathbb{A}=
\{a < b <c < \cdots < d \}$ is the quotient of the free monoid generated by
elements from $\mathbb{A}$ subject to the elementary Knuth transformations~\cite{Kn}
\begin{gather}\label{equation1.1}
 bca = bac\quad \& \quad acb = cab, \qquad \text{and}\qquad bab=bba \quad \& \quad aba=baa,
\end{gather}
for any triple $\{a < b < c \} \subset \mathbb{A}$.

To our knowledge, the concept of ``plactic monoid'' has its origins in a
paper by C.~Schensted~\cite{Sc}, concerning the study of the longest
increasing subsequence of a permutation, and a~paper by D.~Knuth~\cite{Kn},
concerning the study of combinatorial and algebraic properties of the
Robinson--Schensted correspondence\footnote{See, e.g., \url{https://en.wikipedia.org/wiki/Robinson-Schensted_correspondence}.}.

As far as we know, this monoid and the (unital) algebra $\cal{P}(\mathbb{A})$ corresponding to that monoid\footnote{If $\mathbb{A} = \{1 < 2 < \cdots < n \}$, the elements of the
algebra ${\cal{P}}(\mathbb{A})$ can be identif\/ied with semistandard Young
tableaux. It was discovered by D.~Knuth~\cite{Kn} that modulo {\it Knuth
 equivalence} the equivalence classes of semistandard Young tableaux form an
algebra, and he has named this algebra by {\it tableaux algebra}. It is easily
seen that the tableaux algebra introduced by D.~Knuth is isomorphic to the
algebra introduced by M.-P.~Sch\"{u}tzenberger~\cite{Sch}.}, had been introduced, studied and used by M.-P.~Sch\"{u}tzenberger, see
\cite[Section~5]{Sch}, to give the f\/irst complete proof of the famous
\textit{Littlewood--Richardson rule} in the theory of symmetric functions. A~bit
later this monoid, was named the ``mono\"{i}de plaxique'' and studied in depth
by A.~Lascoux and M.-P.~Sch\"{u}tzenberger~\cite{LS1}. The algebra
corresponding to the plactic monoid is commonly known as
{\it plactic algebra}.
 One of the basic properties of the plactic algebra~\cite{Sch} is that it
contains the distinguished commutative subalgebra which is generated by
noncommutative elementary (quasi-symmetric) polynomials\footnote{See, e.g., \cite{KT} for def\/inition of noncommutative quasi-symmetric functions and polynomials.}
\begin{gather*}%\label{equation1.2}
 e_k({\mathbb{A}}_n) = \sum_{i_1 > i_2 > \cdots > i_k} a_{i_{1}} a_{i_{2}}
\cdots a_{i_{k}}, \qquad k= 1, \ldots,n,
\end{gather*}
see, e.g., \cite[Corollary~5.9]{Sch} and~\cite{FG}.

We refer the reader to nice written overview~\cite{LLT} of the basic
properties and applications of the plactic monoid in combinatorics.

It is easy to see that the plactic relations for two letters $a < b$, namely,
\begin{gather*}
 aba=baa,\qquad bab=bba,
 \end{gather*}
imply the commutativity of noncommutative elementary polynomials in two
variables. In other words, the plactic relations for two letters imply that
\begin{gather*}
 ba (a+b)=(a+b) ba, \qquad a < b.
\end{gather*}
 It has been proved in~\cite{FG} that these relations together with the Knuth
relations~\eqref{equation1.1} for three letters $a < b < c$, imply the commutativity
of noncommutative elementary quasi-symmetric polynomials for any number of
variables.

In the present paper we prove that in fact the commutativity of noncommutative
elementary quasi-symmetric polynomials for $n=2$ and $n=3$ implies the
commutativity of that polynomials for all~$n$, see Theorem~\ref{theorem2.23}\footnote{Let \looseness=-1 us stress that conditions necessary and suf\/f\/icient to
assure the commutativity of noncommutative elementary polynomials for the
number of variables equals $n=2$ and $n=3$ turn out to be weaker then that
listed in~\cite{FG}.}.

One of the main objectives of the present paper is to study combinatorial
properties of the generalized plactic Cauchy kernel
\begin{gather*}%\label{equation1.3}
{\cal C}({\mathfrak{P}}_n,U)= \prod_{i=1}^{n-1}
 \left \{{\prod_{j=n-1}^{i}(1+ p_{i,j-i+1} u_{j})} \right\},
\end{gather*}
where ${\mathfrak{P}}_n$ stands for the set of parameters $\{p_{ij},\, 2 \le i+j \le n+1, \,i >1,\,j>1 \}$, and $U := U_{n}$ stands for a certain noncommutative
algebra we are interested in, see Section~\ref{section5}.

We also want to bring to the attention of the reader on some
interesting combinatorial properties of {\it rectangular} Cauchy kernels{\samepage
\begin{gather*}%\label{equation1.4}
{\cal{F}}({\mathfrak{P}}_{n,m}, U) = \prod_{i=1}^{n-1} \left\{\prod_{j=m-1}^{1} (1 + p_{i,\overline{i-j+1}^{(m)}} u_{j}) \right\},
\end{gather*}
where ${\mathfrak{P}}_{n,m} = \{p_{ij} \}_{1 \le i \le n \atop 1\le j \le m}$;
see Def\/inition~\ref{def6.6} for the meaning of symbol~$\overline{a}^{(m)}$.}

We treat these kernels in the (reduced) plactic algebras~${\cal{PC}}_n$ and
${\cal{PF}}_{n,m}$ correspondingly. The algebras~${\cal{PC}}_n$ and~${\cal{PF}}_{n,m}$ are f\/inite-dimensional and have bases parameterized by
certain Young tableaux described in Sections~\ref{section5.1} and~\ref{section6} correspondingly.
Decomposition of the rectangular Cauchy kernel with respect to the basis in
the algebra ${\cal{PF}}_{n,m}$ mentioned above, gives rise to a~set of
polynomials which are common generalizations of the (double) Schubert.
$\beta$-Gro\-thendieck, Demazure and Stanley polynomials. To be more precise,
the polynomials listed above correspond to certain quotients of the plactic
algebra ${\cal{PF}}_{n,m}$ and appropriate specializations of parameters~$\{p_{ij} \}$ involved in our def\/inition of polynomials
$U_{\alpha}(\{p_{ij} \})$, see Section~\ref{section6}.

As it was pointed out in the beginning of Introduction, the Knuth (or plactic)
relations~\eqref{equation1.1} have been discovered in~\cite{Kn} in the course of the
study of algebraic and combinatorial properties of the Robinson--Schensted
correspondence. Motivated by the study of basic properties of a~{\it quantum} version of the tropical/geometric Robinson--Schensted--Knuth correspondence~-- work in progress, but see~\cite{BK,AK,KB,NY,OP} for def\/inition and basic properties of the tropical/geometric RSK,~-- the author of the present paper came to a discovery that certain
deformations of the Knuth relations {\it preserve} the Hilbert series
(resp. the Hilbert polynomials) of the plactic algebras~${\cal{P}}_n$ and~${\cal{F}}_n$ (resp.\ the algebras~${\cal{PC}}_n$ and~${\cal{PF}}_n$).

 More precisely, let $\{q_2,\ldots,q_{n} \}$ be a set of (mutually commuting)
parameters, and ${\mathbb{U}}_n:=\{u_1,\ldots,u_{n} \}$ be a set of
generators of the free associative algebra over~$\Q$ of rank~$n$. Let $Y, Z
\subset [1,n]$ be subsets such that $Y \cup Z=[1,n]$ and $Y \cap Z=\varnothing$. Let us set $p(a)=0$, if $a \in Y$ and $p(a)=1$, if $a \in Z$.
Def\/ine ${q}$-deformed super Knuth relations among the generators
$u_1, \ldots,u_{n}$ as follows:
\begin{align*}
{\rm SPL}_q \colon \ & (-1)^{p(i)p(k)} q_{k} u_j u_i u_k=u_j u_k u_i, \qquad i < j \le k,\\
& (-1)^{p(i)p(k)} q_{k} u_i u_k u_j=u_k u_i u_j, \qquad i \le j <k.
\end{align*}

We def\/ine
\begin{itemize}\itemsep=0pt
\item $q$-deformed superplactic algebra ${\cal{SQP}}_n$ to
 be the quotient of the free associative algebra $\Q \langle u_1,\ldots,u_{n}
\rangle$ by the two-sided ideal generated by the set of {\bf q}-deformed
 Knuth rela\-tions~$({\rm SPL}_{\bf q})$,

\item reduced $q$-deformed superplactic algebras
${\cal{SQPC}}_n$ and ${\cal{SQPF}}_{n,m}$ to be the quotient of the algebra
${\cal{SQP}}_n$ by the two-sided ideals described in Def\/initions~\ref{def5.15} and~\ref{def6.7} cor\-respon\-dingly.
\end{itemize}

 We state

 \begin{con}\label{conjecture1.0}
The algebra ${\cal{SQP}}_n$ and the algebras ${\cal{SQPC}}_n$ and
${\cal{SQPF}}_{n,m}$, are flat deformations of the algebras
${\cal{P}}_n$, ${\cal{PC}}_n$ and ${\cal{PF}}_{n,m}$ correspondingly.
\end{con}

In fact one can consider more general deformation of the Knuth relations, for example take a set of parameters ${\bf Q}:=\{q_{ik},\, 1 \le i < k \le n \}$
and impose on the set of generators $\{u_1,\ldots,u_n \}$ the following
relations
\begin{gather*}%\label{equation1.5}
q_{ik} u_j u_i u_k=u_j u_k u_i, \qquad i < j \le k,\qquad q_{ik} u_i u_k u_j=u_k u_i u_j, \qquad i \le j <k.
\end{gather*}
However we don't know how to describe a set of conditions on parameters~${\bf Q}$ which imply the f\/latness of the corresponding quotient algebra(s),
as well as we don't know an interpretation and dimension of the algebras
${\cal{SQPC}}_n$ and ${\cal{SQPF}}_{n,m}$ for a ``generic'' values of parameters~${\bf Q}$. We {\it expect} the dimension of algebras ${\cal{SQPC}}_n$ and
${\cal{SQPF}}_{n,m}$ each depends piece-wise polynomially on a set of
parameters $\{q_{ij} \in \Z_{\ge 0},\, 1 \le i < j \le n \}$, and pose a~problem to describe its poly\-no\-mia\-lity chambers.

 We also mention and leave for a separate publication(s), the case of algebras
and polynomials associated with {\it superplactic} monoid~\cite{LNS, LT}, which corresponds to the relations ${\rm SPL}_q$ with $q_i =1$, $\forall\, i$. Finally we point out an interesting and important paper~\cite{Li} wherein
the case $Z= \varnothing$, and the all deformation parameters are equal to
each other, has been independently introduced and studied in depth.

Let us repeat that the important property of plactic algebras ${\cal{P}}_n$
is that the noncommutative elementary polynomials
\begin{gather*}
e_{k}(u_{1}, \ldots,n_{n-1}):= \sum_{n-1 \ge a_{1} \ge a_{2} \ge a_{k} \ge1}
 u_{a_{1}} \cdots u_{a_{k}}, \qquad k=1,\ldots,n-1,
\end{gather*}
generate a commutative subalgebra inside of the plactic algebra
${\cal{P}}_n$, see, e.g.,~\cite{FG, LS1}. Therefore all our f\/inite-dimensional algebras introduced in the present paper, have a~distinguished
f\/inite-dimensional commutative subalgebra. We have in mined to describe
these algebras explicitly in a separate publication.

In Section~\ref{section2} we state and prove necessary and suf\/f\/icient conditions in order
the elementary noncommutative polynomials form a mutually commuting family.
Surprisingly enough to check the commutativity of noncommutative elementary
polynomials for any $n$, it's enough to check these
conditions only for $n= 2,3$. However a combinatorial meaning of a
generalization of the Lascoux--Sch\"{u}tzenberger plactic algebra ${\cal{P}}_n$
 obtained in this way, is still missing.

The plactic algebra ${\cal{PF}}_{n,m}$ introduced in Section~\ref{section6}, has a monomial
 basis parametrized by the set of Young tableaux of shape
$ \lambda \subset (n^{m})$ f\/illed by the numbers from the set
$\{1,\ldots,m \}$. In the case $n=m$ it is
well-known~\cite{G,KGV,Ma}, that this number is equal to the
number of symmetric plane partitions f\/itting inside the cube $n \times n \times n$. Surprisingly enough this number admits a~factorization in the product of
the number of totally symmetric plane partitions (${\rm TSPP}$) by the number of
totally symmetric self-complementary plane partitions~(${\rm TSSCPP}$) f\/it inside
 the same cube. A similar phenomenon happens if $|m-n| \le 2$, see Section~\ref{section6}.
More precisely, we add to the well-known equalities
\begin{gather*}
\#|{\cal{B}}_{1,n}| =2^n, \qquad \#|{\cal{B}}_{2,n}| = {2 n +1 \choose n},\qquad
 \#|{\cal{B}}_{3,n}| =2^n \operatorname{Cat}_{n+1} \qquad \cite[A003645]{SL}, \\
\#|{\cal{B}}_{4,n}| = {\frac{1}{2}} \operatorname{Cat}_{n+1} \operatorname{Cat}_{n+2} \qquad \cite[A000356]{SL},\\
\#|{\cal{B}}_{n,5}| =
\frac{{n+5 \choose 5} {n+7 \choose 7} {n+9 \choose 9}}{{n+2
\choose 2}{n+4 \choose 4}} \qquad \cite[A133348]{SL},
\end{gather*}
 the following relations
\begin{gather*}
\#|{\cal{B}}_{n,n}| = {\rm TSPP}(n+1) \times {\rm ASM}(n), \qquad \#|{\cal{B}}_{n,n+1}| = {\rm TSPP}(n+1) \times {\rm ASM}(n+1), \\
\#|{\cal{B}}_{n+2,n}| =\#|{\cal{B}}_{n,n+1}|, \qquad \#|{\cal{B}}_{n+3,n}|= {\frac{1}{2}} \#|{\cal{B}}_{n+1,n+1}|,\\
\# |{\rm PP}(n)| =\#|{\rm TSSCPP}(n)| \times \#|{\rm ASMHT}(2n)|= \#|{\rm CSSCPP}(2n)| \times \#|{\rm CSPP}(n)|, \\
\#|{\rm CSPP}(2n)| = \#|{\rm TSPP}(2n)| \times \# |{\rm CSTCPP}(2n)|, \\
\#|{\rm CSPP}(2n+1)| = 2^{2 n} \#|{\rm TSPP}(2n+1)| \times \#|{\rm TSPP}(2n)|,
\end{gather*}
where ${\rm PP}(n)$ stands for the set of plane partitions f\/it in a cub of size
$n \times n \times n$; ${\rm AMSHT}(2n)$ denotes the set of alternating sign
matrices of size $2n \times 2n$ invariant under a half-turn and ${\rm CSSPP}(2n)$
denotes the set of cyclically symmetric self-complementary plane partitions
f\/itting inside a cub of size $2n \times 2n \times 2n$, see, e.g.,~\cite{B};
${\rm CSTCPP}(n)$ stands for the set of cyclically symmetric transpose complementary plane partitions f\/itting inside a cub of size $2n \times 2n \times 2n$, see, e.g.,~\cite[$A051255$]{SL}.
See Section~\ref{section6} for the def\/inition of the sets ${\cal{B}}_{n,m}$
and examples.
In Exercise~\ref{exer6.3} we state some (new) divisibility properties of
 the numbers $\#|B_{n+4,n}|$.

It is well-known that ${\rm ASMHT}(2n)= {\rm ASM}(n) \times {\rm CSPP}(n)$, where ${\rm CSPP}(n)$
denotes the number of cyclically symmetric plane partitions f\/itting inside
$n$-cube, and ${\rm CSSCPP}(2n)={\rm ASM}(n)^2$, see,~e.g.,~\cite{B,Ku} and \cite[$A006366$]{SL}.

\begin{prb} \quad
\begin{itemize}\itemsep=0pt
\item Construct bijection between the set of plane partitions fit
inside $n$-cube and the set of $($ordered$)$ triples $(\pi_1,\pi_2, \wp)$,
where $ (\pi_1,\pi_2)$ is a pair of ${\rm TSSCPP}(n)$ and $\wp$ is a cyclically
symmetric plane partition fitting inside $n$-cube.

\item Describe the involution $\kappa \colon {\rm PP}(n) \longrightarrow {\rm PP}(n)$
 which is induced by the involution $(\pi_1,\pi_2, \wp)$
$\longrightarrow (\pi_2,\pi_1,\wp)$ on the set ${\rm TSSCPP}(n) \times {\rm TSSCP}(n)
\times {\rm CSPP}(n)$, and its fixed points. Clearly one has $\# | {\rm Fix}(\kappa)| =
{\rm ASMHT}(2n)$.

\item Characterize pairs of plane partitions $(\Pi_1, \Pi_2) \in
{\rm PP}(n) \times {\rm PP}(n)$ such that
\begin{gather*}
(a) \quad \wp(\Pi_1)=\wp(\Pi_2); \qquad (b) \quad (\pi_1(\Pi_1),\pi_2(\Pi_1))=
(\pi_1(\Pi_2),\pi_2(\Pi_2)).
\end{gather*}
\end{itemize}
\end{prb}

These relations have straightforward proofs based on the explicit product
formulas for the numbers
\begin{gather*}
 \#|{\rm SPP}(n)| = \prod_{1 \le i \le j \le k} \frac{n+i+j+k-1}{i+j+k-1},\\
\#|{\rm TSPP}(n)| =\prod_{i=1}^{n} \prod_{j=i}^{n} \prod_{k=j}^{n} \frac{i+j+k-1}{i+j+k-2}, \\
\#|{\rm PP}(n)|= \prod_{i=0}^{n} \prod_{j=1}^{n-1}{\frac{3 n -i -j}{2 n -i -j}}=
\prod_{i=1}^{n}{\frac{{2 n +i \choose n}}{{n +i \choose n}}},
\end{gather*}
but bijective proofs of these identities are an open problem.

It follows from \cite{AL, LS1} that the dimension of the (reduced)
plactic algebra ${\cal{PC}}_n$ is equal to the number of alternating sign
matrices of size $n \times n$ (note that ${\rm ASM}(n)={\rm TSSCPP}(n)$). Therefore the
key-Grothendieck polynomials can be obtained from $U$-polynomials (see
Section~\ref{section6}, Theorem~\ref{theorem6.12}) after the specialization
$p_{ij}=0$, if $i+j > n+1$.

In Section~\ref{section4} following \cite{Ki1} we introduce and study a
family of polynomials
which are a~common generalization of the Schubert, $\beta$-Grothendieck, dual
$\beta$-Grothendieck, $\beta$-Demazure, $\beta$-key-Grothendieck, Bott--Samelson and $q$-Demazure polynomials, Whittaker functions (see~\cite{BBL} and
Lemma~\ref{lem4.20}) and Di~Francesco--Zinn-Justin polynomials (see
Section~\ref{section4}).
 Namely, for any permutation $w \in {\mathbb{S}}_n$ and composition $\zeta
\subset \delta_n := (n-1,n-2,\ldots,2,1)$, we introduce polynomials
\begin{gather*}
{\cal{KN}}_{w}^{(\beta,\alpha,\gamma,h)}(X_n)= h^{\ell(w)} T_{s_{i_{1}}}
\cdots T_{s_{i_{\ell}}} \big(x^{\delta_n}\big), \\
{\rm KD}_{\zeta}^{(\beta,\alpha,\gamma,h)}(X_n)= h^{\ell(v_{\zeta})} T_{s_{i_{1}}}
\cdots T_{s_{i_{\ell}}} \big(x^{\zeta^{+}}\big),
\end{gather*}
where
\begin{gather*}
\begin{split}
& T_{i}:=T_{i}^{(\beta,\alpha,\gamma,h)} =- \alpha + ((\alpha + \beta + \gamma) x_{i} +\gamma x_{i+1} +h \\
& \hphantom{T_{i}:=T_{i}^{(\beta,\alpha,\gamma,h)} =}{} +
h^{-1} (\alpha+\gamma)(\beta + \gamma) x_{i} x_{i+1}) \partial_{i,i+1},\qquad
i=1,\dots,n-1,
\end{split}
\end{gather*}
denote a collection of divided dif\/ference operators which satisfy the Coxeter
 and Hecke relations
\begin{gather*}
T_{i} T_{j} T_{i} = T_{j} T_{i} T_{j}, \qquad \text{if}\quad |i-j|=1;\qquad
T_{i} T_{j}=T_{j} T_{i}, \qquad \text{if} \quad |i-j| \ge 2, \\
T_{i}^2 = (\beta -\alpha) T_{i} +\beta \alpha, \qquad i=1,\ldots,n-1;
\end{gather*}
by def\/inition for any permutation $w \in {\mathbb{S}}_n$ we set
\begin{gather*}
T_{w} := T_{s_{i_{1}}} \cdots T_{s_{i_{\ell}}},
\end{gather*}
for any reduced decomposition $w =s_{i_{1}} \cdots s_{i_{\ell}}$ of a permutation in question; $\zeta^{+}$ denotes a unique partition obtained from $\zeta$
by ordering its parts, and~$v_{\zeta} \in {\mathbb{S}}_n$ denotes the minimal
 length permutation such that $v_{\zeta}(\zeta)=\zeta^{+}$.

Assume that $h=1$.\footnote{Clearly that if $h \not= 0$, then after rescaling parameters
$\alpha$, $\beta$ and $\gamma$
one can assume that $h=1$. However, see, e.g., \cite[Section~5]{Ki1}, the
parameter~$h$ plays important role in the study of dif\/ferent specializations
of the variab\-les~$x_i$, $1 \le i \le n-1$ and parameters~$\alpha$,~$\beta$ and~$\gamma$.} If $\alpha=\gamma = 0$, these polynomials coincide with
the $\beta$-Grothendieck polynomials~\cite{FK1}, if $\beta=\alpha =1$,
$\gamma=0$ these polynomials coincide with the Di Francesco--Zinn-Justin
polynomials~\cite{DZ1}, if $\beta=\gamma=0$, these polynomials coincide
with dual $\alpha$-Grothendieck polynomials ${\cal{H}}_{w}^{(\alpha)}(X_n)$,\footnote{To avoid the reader's confusion, let us explain that in our paper we use the letter~$\alpha$ either as the lower index to denote a composition, or
as the upper index to denote a parameter which appears in certain polynomials
 treated in our paper. For example ${\cal{H}}_{\alpha}^{(\alpha)}(X_n)$ denotes
 the dual $\alpha$-Grothendieck polynomial corresponding to a composition~$\alpha$. Note that the $\alpha$-Grothendieck polynomial ${\mathfrak{G}}_{w}^{(\alpha)}(X_n)$ can be obtained from the polynomial ${\mathfrak{G}}_{w}^{(\beta)}(X_n)$ by replacing~$\beta$ by~$\alpha$.} where by def\/inition we set $X_n:= (x_1,\ldots,x_n)$.

\begin{con} For any permutation $w \in {\mathbb{S}}_n$ and any composition
 $\zeta \subset \delta_n$, polynomials ${\cal{KN}}_{w}^{(\beta,\alpha,\gamma,h)}(X_n)$ and ${\rm KD}_{\zeta}^{(\alpha,\beta,\gamma,h)}(X_n)$ have nonnegative coefficients, i.e.,
\begin{gather*}
{\cal{KN}}_{w}^{(\beta,\alpha,\gamma,h)}(X_n) \in \N[\alpha,\beta,\gamma,h]
[X_n], \qquad {\rm KD}_{\zeta}^{(\beta,\alpha,\gamma,h)}(X_n) \in \N[\alpha,\beta,\gamma,h] [X_n].
\end{gather*}
\end{con}
We {\it expect} that these polynomials have some geometrical meaning
to be discovered.

More generally we study divided dif\/ference type operators of the form
\begin{gather*}
T_{ij}:= T_{ij}^{(a,b,c,h,e)}= a + (b x_i +c x_j + h +e x_i x_j) \partial_{ij},
\end{gather*}
depending on parameters $a$, $b$, $c$, $h$, $e$ and satisfying the $2D$-Coxeter relations
\begin{gather*}
 T_{ij} T_{jk} T_{ij}= T_{jk} T_{ij} T_{jk}, \qquad 1 \le i < j < k \le n,\\
T_{ij} T_{kl} =T_{kl} T_{ij}, \qquad \text{if}\quad \{i,j\} \cap \{k,l \}=\varnothing.
\end{gather*}
We f\/ind that the necessary and suf\/f\/icient condition which ensure the validity
of the $2D$-Coxeter relations is the following relation among the parameters\footnote{In other words, the divided dif\/ference operators $ \{T_{ij}:= T_{ij}^{(a,b,c,h,e)}\}$ which obey the $2D$-Coxeter relations, have the following
form:
\begin{gather*}
 %(1)
 T_{ij}^{(a,b,c,h)}= \left(1+\frac{(a+b)}{h} x_i\right)\left(1+\frac{c-a}{h} x_{j}\right) \partial_{ij} +a \sigma_{ij},
 \qquad \text{if}\quad h \not= 0, \quad \text{$a$, $b$, $c$ arbitrary}.
\end{gather*}
%$(2)$
If $h=0$, then either $T_{ij}= ((c-a) x_j+e x_i x_j) \partial_{ij} +
a \sigma_{ij}$, or $T_{ij} =(a+b) x_i + e x_i x_j) \partial_{ij} +a \sigma_{ij}$,
where $\sigma_{ij}$ stands for the exchange operator: $\sigma_{ij}(F(z_i,z_j)) = F(z_j,z_i)$.}:
\begin{gather*}
 (a+b)(a-c)+h e = 0.
\end{gather*}

Therefore, if the above relation between parameters $a$, $b$, $c$, $h$, $e$ holds, then
for any permutation $w \in {\mathbb{S}}_n$ the operator
\begin{gather*}
T_{w}:= T_{w}^{(a,b,c,h,e)}= T_{i_{1}}^{(a,b,c,h,e)} \cdots T_{i_{\ell}}^{(a,b,c,h,e)},
\end{gather*}
where $w=s_{i_{1}} \cdots s_{i_{\ell}}$ is any reduced decomposition of~$w$, is
 \textit{well-defined}. Hence under the same assumption on parameters,
for any permutation $w \in {\mathbb{S}}_n$ one can attach the well-def\/ined
 polynomial
\begin{gather*}
G_{w}^{(a,b,c,h,e)}(X,Y) := {T_{w}^{(x)}}^{(a,b,c,h,e)}\bigg(\prod_{i \ge 1,
j \ge 1 \atop i+j \le n+1}(x_i+y_j)\bigg),
\end{gather*}
and in much the same fashion to def\/ine polynomials
\begin{gather*}
D_{\alpha}^{(a,b,c,h,e)}(X,Y) := {T_{w_{\alpha}}^{(x)}}^{(a,b,c,h,e)}\big(x^{\alpha^{+}}\big)
\end{gather*}
 for any composition $\alpha$ such that $\alpha_{i} \le n-i$, $\forall\, i$.
We have used the notation $ {T^{(x)}}_{w}^{(a,b,c,h,e)}$ to point out
that this operator acts only on the variables $X =(x_1,\ldots,x_n)$; for any
composition $\alpha \in \Z_{\ge 0}^n$, $\alpha^{+}$~denotes a unique
partition obtained from~$\alpha$ by reordering its parts in (weakly)
 decreasing order, and~$w_{\alpha}$ denotes a unique minimal length
permutation in the symmetric group~$\mathbb{S}_{n}$ such that
$w_{\alpha}(\alpha)=\alpha^{+}$.

 In the present paper we are interested in to list a conditions on
parameters $A:= \{a,b,c,h,e \}$ with the constraint
\begin{gather*}%\label{equation1.6}
(a+b)(a-c)+h e=0,
\end{gather*}
 which ensure that the above polynomials
$G_{w}^{(a,b,c,h,e)}(X)$ and $D_{\alpha}^{(a,b,c,h,e)}(X)$ or their
specialization $x_i=1$, $\forall\, i$, have nonnegative coef\/f\/icients.
 We state the following conjectures:
\begin{itemize}\itemsep=0pt
\item ${\cal{KN}}_{w}^{(\beta,\alpha,\gamma)}(X_n) \in \N[\alpha, \beta, \gamma] [X_n]$,

\item $G_{w}^{(-b,a+b+c,c,1,(b+c)(a+c)}(X_n) \in \N [a,b,c] [X_n]$,

\item $G_{w}^{(-b,a+b+c,c+d,1,(b+c+d)(a+c)}(x_i=1,\, \forall\, i) \in
\N[a,b,c,d]$, where $a$, $b$, $c$, $d$ are free parameters.
\end{itemize}

In the present paper we treat the case
\begin{gather}\label{equation1.7}
A= (- \beta,\beta+\alpha+\gamma, \gamma,1,(\alpha+\gamma)(\beta+\gamma)).
\end{gather}
As it was pointed above, in this case polynomials $G_{w}^{A}(X)$ are common
generalization of Schubert, $\beta$-Grothendieck and dual $\beta$-Grothendieck,
 and Di Francesco--Zinn-Justin polynomials. We expect a certain interpretation
of the polynomials $G_{w}^{A}$ for general $\beta$, $\alpha$ and $\gamma$.

As it was pointed out earlier, one of the basic properties of the plactic
monoid ${\cal{P}}_{n}$ is that the noncommutative elementary symmetric
polynomials $\{e_{k}(u_{1},\ldots,u_{n-1}) \}_{1 \le k \le n-1}$ generate
a~commutative subalgebra in the plactic algebra in question. One can
reformulate this statement as follows. Consider the generating function
\begin{gather*}%\label{equation1.8}
A_{i}(x): = \prod_{a=n-1}^{i}(1+x u_{a}) = \sum_{a=0}^{i}e_{a}(u_{n-1}, \ldots,u_{i}) x^{i-a},
\end{gather*}
where we set $e_{0}(U)=1$. Then the commutativity property of noncommutative
elementary symmetric polynomials is equivalent to the following commutativity
relation in the plactic as well as in the generic plactic, algebras ${\cal{P}}_n$ and ${\mathfrak{P}}_n$~\cite{FG}, and Theorem~\ref{theorem2.23},
\begin{gather*}
A_{i}(x) A_{i}(y)= A_{i}(y) A_{i}(x), \qquad 1 \le i \le n-1.
\end{gather*}
Now let us consider the Cauchy kernel
\begin{gather*}%\label{equation1.9}
{\cal{C}}({\mathfrak{P}}_n,U) = A_{1}(z_1) \cdots A_{n-1}(z_{n-1}),
\end{gather*}
where we assume that the pairwise commuting variables $z_1, \ldots,z_{n-1}$
commute with the all generators of the algebras~${\cal{P}}_n$ and
${\mathfrak{P}}_n$. In what follows we consider the natural completion
${\widehat{\mathfrak{P}}}_n$ of the plactic algebra~${\mathfrak{P}}_n$ to
allow consider elements of the form $(1 +x u_{i})^{-1}$. Elements of this
form exist in any Hecke type quotient of the plactic algebra
${\widetilde{\mathfrak{P}}}_n$. Having in mind
this assumption, let us compute the
action of divided dif\/ference operators $\partial_{i,i+1}^{z}$ on the Cauchy
kernel. In the computation below, the commutativity property of the elements
$A_{i}(x)$ and $A_{i}(y)$ plays the key role. Let us start computation of
 $\partial_{i,i+1}^{z}({\cal{C}}({\mathfrak{P}}_n,U)) =
\partial_{i,i+1}^{z}(A_1(z_1) \cdots A_{n-1}(z_{n-1}))$.
First of all write $A_{i+1}(z_{i+1})= A_{i}(z_{i+1})(1+z_{i+1}
u_i)^{-1}$. According to the basic property of the elements~$A_i(x)$, one sees
 that the expression $A_{i}(z_{i}) A_{i}(z_{i+1})$ is symmetric with respect
to $z_i$ and $z_{i+1}$, and hence is invariant under the action of divided
dif\/ference operator $\partial_{i,i+1}^{z}$. Therefore,
\begin{gather*}
\partial_{i,i+1}^{z}({\cal{C}}({\mathfrak{P}}_n,U)) = A_1(z_1) \cdots
A_{i}(z_i) A_{i}(z_{i+1}) \partial_{i,i+1}^{z}\big((1+z_{i+1} u_{i})^{-1}\big)\\
\hphantom{\partial_{i,i+1}^{z}({\cal{C}}({\mathfrak{P}}_n,U)) =}{}\times A_{i+2}(z_{i+2}) \cdots A_{n-1}(z_{n-1}).
\end{gather*}
It is clearly seen that $\partial_{i,i+1}^{z}((1+z_{i+1} u_{i})^{-1})=
(1+z_{i} u_{i})^{-1} (1+z_{i+1} u_{i})^{-1} u_{i}$. Therefore,
\begin{gather*}
\partial_{i,i+1}^{z}({\cal{C}}({\mathfrak{P}}_n,U)) =
A_1(z_1) \cdots
A_{i}(z_i) A_{i+1}(z_{i+1}) (1+z_{i} u_{i})^{-1} u_{i} A_{i+2}(z_{i+2}) \cdots A_{n-1}(z_{n-1}).
\end{gather*}
It is easy to see that if one adds Hecke's type relations on the generators
\begin{gather*}
 u_{i}^2= (a+b) u_i + a b, \qquad i=1,\ldots,n-1,
\end{gather*}
then
\begin{gather*}%\label{equation1.10}
 (1+ z u_{i})^{-1} u_{i} =\frac{u_{i} - z a b}{(1+b z)(1-a z)}.
\end{gather*}
Therefore in the quotient of the plactic algebra~${\mathfrak{P}}_n$ by the
Hecke type relations listed above and by the ``locality'' relations
\begin{gather*}
u_i u_j = u_j u_i, \qquad \text{if} \quad |i-j| \ge 2,
\end{gather*}
one obtains
\begin{gather*}
 (-b+(1+z_{i} b)) \partial_{i,i+1}^z \left(A_{1}(z_1) \cdots A_{n-1}(z_{n-1}) \right) = \left(A_{1}(z_1) \cdots A_{n-1}(z_{n-1}) \right) \left(
\frac{e_{i} - b}{1-a z_{i}}\right).
\end{gather*}
Finally, if $a=0$, then the above identity takes the following form
\begin{gather*}%\label{equation1.11}
\partial_{i,i+1}^z \left((1+z_{i+1} b) A_{1}(z_1) \cdots A_{n-1}(z_{n-1}) \right) = \left(A_{1}(z_1) \cdots A_{n-1}(z_{n-1}) \right) ( e_{i} - b).
\end{gather*}
In other words the above identity is equivalent to the statement~\cite{FK} that
in the idCoxeter algebra~${\cal{IC}}_n$ the Cauchy kernel
${\cal{C}}({\mathfrak{P}}_n,U)$ is the generating function for the
$b$-Grothendieck polynomials. Moreover, each (generalized) double
$b$-Grothendieck polynomial is a positive linear combination of the key-Grothendieck polynomials.

A proof of this statement is a corollary of the more
general statement which will be frequently used throughout the present paper,
namely, if an equivalence relation~${\approx}_2$ is a ref\/inement of that~${\approx}_1$, that is if assumption $a \, {\approx}_2 \, b \Longrightarrow a \, {\approx}_1 \, b$
holds $\forall\, a,\, b$, then each equivalence class w.r.t.\ relation~${\approx}_1$
 is disjoint union of the equivalence classes w.r.t.\ relation~${\approx}_2$.

 In the special case $b= -1$ and $P_{ij}= x_i+ y_j$
 if $ 2 \le i+j \le n+1$, $p_{ij}=0$, if $i+j > n+1$, this result had been stated
in~\cite{L2}.

As a possible mean to def\/ine {\it affine versions} of polynomials treated in
the present paper, we introduce the {\it double affine nilCoxeter algebra of
type A} and give construction of a generic family of Hecke's type elements\footnote{Remind that by the name {\it a family of
Hecke's type elements} we mean a set of elements $ \{e_1,\dots,e_n \}$ such
that
$e_i^2 = A e_i +B$, $A$, $B$ are parameters (Hecke type relations),
$e_i e_j =e_j e_i$, if $|i-j| \ge 2$, $e_i e_j e_i=e_j e_i e_j$, if $|i-j|=1$ (Coxeter relations).}
 we will be put to use in the present paper.

In Section~\ref{section5.1} we suggest a~common generalization of some combinatorial formulas from~\cite{CV} and~\cite{Gou}. Namely, we give explicit formula
\begin{gather}\label{equation1.3}
{\prod_{1 \le i \le k, \, 1 \le j \le n \atop j-i \le n-k } \frac{N-i-j+1}{i+j-1}} {\prod_{1 \le i \le k,\, 1 \le j \le n \atop j-i > n-k} \frac{N+i+j-1}{i+j-1}}
\end{gather}
 for the number of $k$-tuples of noncrossing Dyck paths connecting the points
$(0,0)$ and $(N,N-n-k)$. Interpretations of the number~\eqref{equation1.3} as
the number of certain $k$-triangulations of a convex $(N+1)$-gon, or that of
certain alternating sign matrices of size $N \times N$, are interesting
tasks.

In the case $N=n+k$ we recover the \cite[r.h.s.\ of formula~(2)]{CV}. In the
case $k=2$ our formula~\eqref{equation1.3} is equivalent to that obtained in~\cite{Gou}. Our proof that the number~\eqref{equation1.3} counts certain $k$-tuples of noncrossing Dyck paths is
based on the study of combinatorial properties of the so-called column
multi-Schur functions $s_{\lambda}^{*} (X_n)$ introduced in Theorem~\ref{theorem5.6}, cf.~\cite{Ma, Wa}. In particular we show that for rectangular partition $\lambda =(n^k)$ the polynomial $s_{\lambda}^{*}(X_{n+k})$ is essentially coincide
with the Schubert polynomial corresponding to the Richardson permutation
$1^{k} \times w_{0}^{(n)}$. We introduce also a multivariable deformation of
the numbers $\operatorname{ASM}(n)$, namely,
\begin{gather*}%\label{equation1.4}
\operatorname{ASM}(X_{n-1};t):= \sum_{\lambda \subset \delta_n} s_{\lambda}^{*}(X) t^{|\lambda|}.
\end{gather*}
Finally, in Section~\ref{section5.1} we give combinatorial interpretations of
rectangular and staircase components of the refined ${\rm TSSCP}$ vector~\cite{DZ1} in terms of $k$-fans of noncrossing Dyck paths in rectangular case and Gandhi--Dumont polynomials and Genocchi numbers in staircase case.

As Appendix we include several examples of polynomials studied in the present
paper to illustrate results obtained in these notes. We also include an
expository text concerning the {\it MacNeille completion} of a poset to draw
attention of the reader to this subject. It is an exami\-nation of the
MacNeille completion of the poset associated with the (strong) Bruhat order on the
symmetric group, that was one of the main streams of the study in the present
paper. Namely, our concern was the challenge how to attach to each edge $e$ of
the MacNeille completion~${\cal{MN}}({\mathbb{S}}_n)$ of the Bruhat order poset on the symmetric group ${\mathbb{S}}_n$ an operator $\partial_{e}$ acting on the ring of polynomials $\Z[X_n]$, such that $\partial_{e}(K_{h(e)})= K_{t(e)}$
together with compatibility conditions among the set of operators
 $\{\partial_{e}\}_{e \in {\cal{MN}}({\mathbb{S}}_{n})}$, that is for any two vertices of~${\cal{MN}}({\mathbb{S}}_{n})$, say~$\alpha$ and~$\beta$, and a~path
$p_{\alpha,\beta}$ in the MacNeille completion which connects these vertices, the naturally def\/ined operator $\partial_{p_{\alpha,\beta}}$ depends only on the vertices~$\alpha$ and~$\beta$ taken, and doesn't depend on a path~$p_{\alpha,\beta}$ selected. As far as I know, this problem is still open.

\section{Plactic, nilplactic and idplactic algebras}\label{section2}

\begin{de}[\cite{LS1}]\label{definition2.1} The {\it plactic algebra ${\cal P}_n$} is an
(unital) associative algebra over $\Z$ generated by elements $\{u_1,
\dots, u_{n-1} \}$ subject to the set of relations
\begin{gather*}
\text{(PL1)} \quad u_j u_i u_k=u_j u_k u_i, \qquad u_i u_k u_j=u_k u_i u_j, \qquad \text{if} \quad i < j < k, \\
\text{(PL2)} \quad u_i u_j u_i=u_j u_i u_i, \qquad u_j u_i u_j=u_j u_j u_i, \qquad \text{if}\quad i < j.
\end{gather*}
\end{de}
\begin{pr}[\cite{LS1}]\label{prop2.2} Tableau words\footnote{For the reader convenience we recall a def\/inition of a
{\it tableau word}. Let $T$ be a (regular shape) semistandard Young
tableau. The tableau word $w(T)$ associated with~$T$ is the {\it reading
word} of~$T$ is the sequence of entries of~$T$ obtained by concatenating the
columns of~$T$ bottom to top consecutively starting from the f\/irst column.
 For example, take
\begin{gather*}
T= \begin{matrix}1 & 2 &3 &3 \cr 2 & 3 & 4\cr 3 & 4 \cr 5
\end{matrix}
\end{gather*}
The corresponding {\it tableau word} is $w(T)= 5321432433$. By def\/inition, a
tableau word is the tableau word corresponding to some (regular shape)
semistandard Young tableau. It is well-known~\cite{LS2} that the number of
tableau subwords contained in the staircase word
 $I_{0}^{(n)}:= \underbrace{u_{n-1}u_{n-2}\cdots u_{2}u_{1}}\underbrace{u_{n-1}u_{n-2}\cdots u_{2}}\cdots \underbrace{u_{n-1}u_{u-2}}\underbrace{u_{n-1}}$
 is equal to the number of alternating
sign matrices ${\rm ASM}(n)$.\label{footnote9}}
 in the alphabet $U=\{u_1,\dots,u_{n-1} \}$ form a~basis in the plactic algebra~${\cal P}_n$.
\end{pr}

In other words, each plactic class contain a unique tableau word.
In particular,
\begin{gather*}
\operatorname{Hilb}({\cal P}_{n+1},t)=(1-t)^{-n}\big(1-t^2\big)^{-{n \choose 2}}.
\end{gather*}

\begin{rem}\label{rem2.3} There exists another algebra over $\Z$ which
has the same Hilbert series as that of the plactic algebra ${\cal P}_{n}$.
 Namely, def\/ine
algebra ${\cal L}_n$ to be an associative algebra over $\Z$ generated by the
elements $ \{e_1,e_2,\ldots,e_{n-1}\}$, subject to the set of relations
\begin{gather*}
(e_i,(e_j,e_k)):= e_i e_j e_k -e_j e_i e_k -e_j e_k e_i +e_k e_j e_i = 0,
\end{gather*}
for all\looseness=-1 $1 \le i,j, k \le n-1$, $j < k$. Observe that the number of def\/ining relations in the algebra~${\cal L}_n$
is equal to $2 {n \choose 3}$. Note that elements $e_1+e_2$ and
$e_2 e_1$ do not commute in the algebra~${\cal{L}}_3$, but do commute if are considered as elements in the plactic algebra~${\cal{P}}_3$. See Example~\ref{exam5.25}(C) for some details.
 \end{rem}

\begin{exer}\label{exer2.4} \samepage\quad
\begin{itemize}\itemsep=0pt
\item Show that the dimension of the degree $k$ homogeneous
 component ${{\cal L}}^{(k)}_n$ of the algebra~${\cal L}_n$ is equal to the
number semistandard Young tableaux of the size~$k$ f\/illed by the numbers from
the set $\{1,2,\ldots,n-1 \}$.

\item Let us set $e_{ij}:= (e_i,e_j):= e_i e_j-e_j e_i$, $i < j$.
Show that the elements $\{e_{ij} \}_{1 \le i < j \le n-1}$ generate
 the {\it center} of the algebra~${\cal{L}}_n$.
 \end{itemize}
\end{exer}

\begin{de}\label{def2.5} \quad
\begin{enumerate}\itemsep=0pt
\item[$(a)$] The {\it local plactic algebra} ${\cal LP}_n$, see, e.g.,~\cite{FG},
is an associative algebra over $\Z$ generated by elements $\{u_1,\ldots,u_{n-1} \}$ subject to the set of relations
\begin{gather*}
u_i u_j=u_j u_i, \qquad \text{if}\quad |i-j| \ge 2, \\ u_{j} u_i^2=u_i u_{j} u_i, \qquad u_{j}^2 u_i
=u_j u_i u_j, \qquad \text{if}\quad |i-j|=1.
\end{gather*}
One can show (A.K.) that
\begin{gather*}
\operatorname{Hilb}({\cal LP}_n,t)= \prod_{j=1}^{n} \left(
{1 \over 1-t^j} \right)^{n+1-j}.
\end{gather*}

\item[$(b)$] The {\it affine local plactic algebra} $\widehat{{\cal{PL}}}_n$, see~\cite{KoS}, is an associative algebra over~$\mathbb{Q}$ generated by the elements
$\{e_{0},\ldots,e_{n-1} \}$ subject to the set of relations listed in item $(a)$, where all indices are understood modulo~$n$.

\item[$(c)$] Let $q$ be a parameter, the {\it affine quantum local plactic algebra} $\widehat{{\cal{PL}}}_n^{(q)}$ is an associative algebra over $\mathbb{Q}[q]$ generated by the elements $\{e_{0},\ldots,e_{n-1} \}$ subject to the set of relations
\begin{enumerate}\itemsep=0pt
\item[$(i)$] $e_i e_j = e_j e_i$, if $|i-j| > 1$,

\item[$(ii)$] $(1+q)e_i e_j e_i = q e_i^2 e_j + e_j e_i^2$, $(1+q)e_j e_i e_j = q e_i e_j^2 + e_j^2 e_i$, if $|i-j|=1$,
 where all indices appearing in the relation~$(ii)$ are understood modulo~$n$.
 \end{enumerate}
\end{enumerate}
\end{de}

It was observed in~\cite{Kor} that the relations listed in~$(ii)$ are the Serre
relations of $U_{q}^{\ge}\big(\hat{\mathfrak{gl}}(n)\big)$ rewritten\footnote{For the reader convenient, we recall, see, e.g., \cite{Kor} the
def\/inition of algebra $U_{q}^{\ge}(\hat{\mathfrak{gl}}(n))$. Namely, this
algebra is an unital associative algebra over~$\Q(q^{\pm 1})$ generated by
the elements
$\{k_{i}^{\pm 1},E_i \}_{i=0,\ldots,n-1}$, subject to the set of relations
\begin{enumerate}\itemsep=0pt
\item[$(R1)$] $k_i k_j = k_j k_i$, $k_i k_i^{-1}= k_i^{-1}=1$, $\forall\, i,\, j$,
\item[$(R2)$] $k_i E_j = q^{\delta_{i,j}-\delta_{i,j+1}} E_j k_i$,
\item[$(R3)$] (Serre relations)
\begin{gather*}
E_i^2 E_{i+1} -\big(q+q^{-1}\big) E_i E_{i+1} E_{i}+ E_{i+1} E_i^2 = 0,\qquad
E_{i+1}^2 E_{i} -\big(q+q^{-1}\big) E_{i+1} E_{i} E_{i+1}+ E_{i} E_{i+1}^2 = 0,
\end{gather*}
where all indices are understood modulo $n$, e.g., $e_{n}=e_{0}$.
\end{enumerate}
It is clearly seen that after rewriting the Serre relations $(R3)$ in terms of
the elements $\{e_i:=k_i E_i\}_{0 \le i \le n-1}$, one comes to the relations~$(ii)$ which have been listed in~$(c)$.}
 in generators~$k_i E_i$. Some interesting properties and applications of the
local and af\/f\/ine local plactic algebras one can f\/ind in~\cite{Kor,KoS,KT}. It seems an interesting problem to investigate properties and
applications of the af\/f\/ine {\it pseudoplactic} algebra which is the quotient
of the af\/f\/ine local plactic algebra ${\widehat{\cal{LP}}}_n$ by the two-sided ideal generated by the set of elements
\begin{gather*}
 \{(e_j,(e_i,e_k)), \, 0 \le i < j < k \le n-1 \}.
\end{gather*}

\begin{de}[nil Temperley--Lieb algebra]\label{def2.6}
 Denote by ${\cal TL}_n^{(0)}$ the quotient of the local plactic algebra
 ${\cal LP}_n$ by the two-sided ideal generated by the elements
$\big\{u_1^2,\ldots,u_{n-1}^2 \big\}$.
\end{de}

It is well-known that $\dim {\cal{TL}}_n = C_n$,
the $n$-th Catalan number. One also has
\begin{gather*}
 \operatorname{Hilb}\big({\cal{TL}}_{4}^{(0)},t\big)= (1,3,5,4,1), \qquad \operatorname{Hilb}\big({\cal{TL}}_{5}^{(0)},t\big)=(1,4,9,12,10,4,2), \\ \operatorname{Hilb}({\cal{TL}}_{6},t)=(1,5,14,25,31,26,16,9,4,1).
\end{gather*}

\begin{pr}\label{prop2.7}
The Hilbert polynomial $\operatorname{Hilb}({\cal{TL}}_{n},t)$ is equal to the generating
function for the number or $321$-avoiding permutations of the set
$\{1,2,\ldots,n\}$
having the inversion number equals to $k$, see~{\rm \cite[$A140717$]{SL}}, for other
combinatorial interpretations of the polynomials $\operatorname{Hilb}({\cal{TL}}_{n},t)$.
\end{pr}

\begin{exer}\label{exer2.8} \quad
\begin{itemize}\itemsep=0pt
\item Show that $\deg_{t} \operatorname{Hilb}({\cal{TL}}_{n},t) = \big\lbrack {\frac{n^2}{4}} \big\rbrack$.

\item Show that
\begin{gather*}
\lim_{k \rightarrow \infty} t^{k^{2}} \operatorname{Hilb}\big({\cal{TL}}_{2k},t^{-1}\big) =
\prod_{n \ge 1} {\frac{\big(1-t^{4 n-2}\big)}{\big(1-t^{2 n -1}\big)^{4} \big(1-t^{4 n}\big)}} = \sum_{n \ge 0} a_{n} t^{n},
\end{gather*}
where $a_{n}$ is equal to the number of $2$-colored generalized Frobenius
partitions, see, e.g., \cite[$A051136$]{SL} and the literature quoted therein\footnote{The second equality in the above formula is due to G.~Andrews~\cite{An}. The second formula for the generating function of the numbers $\{a_n\}_{n \ge 0}$ displayed in~\cite[$A051136$]{SL}, either contains misprints or counts something else.}.

\item Show that
\begin{gather*}
 \lim_{k \rightarrow \infty} t^{k(k+1)} \operatorname{Hilb}\big({\cal{TL}}_{2k+1},t^{-1}\big) =
2 \prod_{n \ge 1} {\frac{(1+t^{n}) \big(1+t^{2n}\big)^2}{1-t^n}} = 2 \sum_{n \ge 0}
b_{n} t^{n}.
\end{gather*}
See \cite[$A201078$]{SL} for more details concerning relations of this exercise
 with the Ramanujan theta functions\footnote{See, e.g., \url{http://en.wikipedia.org/wiki/Ramanujan_theta_function}.}.
 \end{itemize}
\end{exer}

We denote by ${\cal{TH}}_{n}^{(\beta)}$ the quotient of the local plactic algebra ${\cal{LP}}_n$ by the two-sided ideal generated by the elements
$\big\{u_i^2- \beta u_i,\, i=1,\ldots,n-1 \big\}$.
\begin{de} \label{def2.9} The modif\/ied \textit{plactic algebra} ${\cal{MP}}_n$ is an
associative algebra over $\Z$ generated by $\{u_1,\ldots,u_{n-1} \}$ subject
to the set of relations (PL1) and that
\begin{gather*}
 u_j u_j u_i = u_j u_i u_i, \qquad \text{and} \qquad u_i u_j u_i=u_j u_i u_j, \qquad \text{if} \quad 1 \le i < j \le n-1.
\end{gather*}
\end{de}

\begin{de}\label{def2.10} The \textit{nilplactic algebra} ${\cal{NP}}_n$ is an associative algebra over $\Q$ generated by $\{u_1, \ldots,u_{n-1} \}$,
subject to set of relations
\begin{gather*}%\label{equation2.1}
u_i^2=0, \qquad u_{i} u_{i+1} u_{i}= u_{i+1} u_{i} u_{i+1},
\end{gather*}
the set of relations (PL1), and that $u_i u_j u_i= u_j u_i u_j$, if $|i-j| \ge 2$.
\end{de}

\begin{pr}[\cite{LS3}]\label{prop2.11}
Each nilplactic class not containing zero, contains one and only one tableau word.
\end{pr}

\begin{pr}\label{prop2.12} The nilplactic algebra ${\cal{NP}}_n$ has finite dimension, its
Hilbert polynomial $\operatorname{Hilb}({\cal{NP}}_n,t)$ has degree ${n \choose 2}$, and
$\dim_{\Q}({\cal{NP}}_n)_{{n \choose 2}} =1$.
\end{pr}

\begin{ex}\label{exer2.13}
\begin{gather*}
\operatorname{Hilb}({\cal{NP}}_3,t)= (1,2,2,1), \qquad \operatorname{Hilb}({\cal{NP}}_4,t)=(1,3,6,6,5,3,1), \\
\operatorname{Hilb}({\cal{NP}}_5,t)=(1,4,12,19,26,26,22,15,9,4,1), \qquad \dim_{\Q}({\cal{NP}}_5)=139,\\
\operatorname{Hilb}({\cal{NP}}_6,t)=(1,5,20,44,84,119,147,152,140,224,81,52,29,14,5,1),\\
\dim_{\Q}({\cal{NP}}_6) = 1008.
\end{gather*}
\end{ex}

\begin{de}\label{def2.14} The \textit{idplactic algebra} ${\cal{IP}}_{n}:= {\cal{IP}}_{n}^{(\beta)}$ is an
associative algebra over $\Q[\beta]$ ge\-ne\-rated by elements $\{u_1,\ldots,u_{n-1} \}$ subject to the set of relations
\begin{gather}\label{equation2.2}
u_i^2= \beta u_i, \qquad u_{i} u_{j} u_{i}= u_{j} u_{i} u_{j}, \qquad i < j,
\end{gather}
and the set of relations~(PL1).
\end{de}

In other words,the idplactic algebra ${\cal{IP}}_n^{(\beta)}$ is the quotient
of the modif\/ied plactic algebra~${\cal{MP}}_n$ by the two-sided ideal generated by the elements $\big\{u_i^2- \beta u_i, \, 1 \le i \le n\big\}$.
\begin{pr}\label{prop2.15} Each idplactic class not containing zero, contains a unique tableau word associated with a row strict semistandard Young tableau\footnote{Recall that a row strict semistandard Young tableau, say~$T$, is a tableau
such that the numbers in each row and each column of $T$ are strictly
increasing. For example,
\begin{gather*}
\begin{matrix}
1 & 2 & 3 & 5 & 6\\
2 & 3 & 6 & 8\\
4 & 5 & 7
\end{matrix}
\end{gather*}}.
\end{pr}

 For each word $w$ denote by $\operatorname{rl}(w)$ the length of a unique tableau word of
minimal length which is idplactic equivalent to $w$.
\begin{ex}\label{exam2.16} Consider words in the alphabet $\{a < b < c < d \}$. Then
\begin{gather*}
\operatorname{rl}(dbadc)=4=\operatorname{rl}(cadbd), \qquad \operatorname{rl}(dbadbc)=5=\operatorname{rl}(cbadbd).
\end{gather*}
 Indeed,
\begin{gather*}
dbadc \sim dbdac \sim dbdca \sim ddbca \sim dbac,\\
 dbadbc \sim dbabdc \sim dabadc \sim adbdac \sim abdbca \sim abbdca
\sim dbabc.
\end{gather*}
\end{ex}

Note that according to our def\/inition, tableau words $w=31$, $w=13$ and
 $w=313$ belong to dif\/ferent idplactic classes.

\begin{pr}\label{2.17} The idplactic algebra ${\cal {IP}}_n^{(\beta)}$ has finite
dimension, and its Hilbert polynomial has degree~${n \choose 2}$.
\end{pr}

\begin{ex}\label{exam2.18}
\begin{gather*}
\operatorname{Hilb}({\cal {IP}}_3,t)=(1,2,2,1), \qquad \operatorname{Hilb}({\cal {IP}}_4,t)=
(1,3,6,7,5,3,1), \qquad \dim({\cal {IP}}_4)=26, \\
\operatorname{Hilb}({\cal {IP}}_5,t)=(1,4,12,22,30,32,24,15,9,4,1), \qquad \dim({\cal {IP}}_5) = 154, \\
\operatorname{Hilb}({\cal {IP}}_6,t)=(1,5,20,50,100,156,188,193,173,126,84,52,29,14,5,1), \\
 \dim({\cal {IP}}_6)=1197, \qquad
\dim({\cal {IP}}_7)= 9401.
\end{gather*}
\end{ex}

Note that for a given $n$ some words corresponding to strict semistandard
Young tableaux of size between~5 and ${n-1 \choose 2}$ are idplactic equivalent to zero. For example,
\begin{gather*}
\begin{matrix}
1 & 2 & 4\\
3 & 4 \end{matrix} \sim 31 42 4 \sim 31242 \sim 13242 \sim 13422 \sim 0,
\\
\begin{matrix}
1 & 2 & 3 \\
2 & 4 \\
4
\end{matrix} \sim 421 42 3 \sim 421243 \sim 412143 \sim 142413 \sim 124213 \sim
122431 \sim 0.
\end{gather*}
It seems an interesting {\it problem} to count the number of all row strict
semistandard Young tableaux contained in the staircase $(n,n-1, \ldots,2,1)$
and bounded by~$n$, as well as count the number the number of such tableaux
which are idplactic equivalent to zero. For $n=2$ these numbers are $(6,6)$, for
$n=3$ these numbers are $(26,26)$, for $n=4$ these numbers are $(160,154)$ and
for $n=5$ they are $(1427,1197)$.

\begin{de}\label{def2.19} The \textit{idplactic Temperly--Lieb algebra} ${\cal PTL}_n^{(\beta)}$
is def\/ine to be the quotient of the idplactic algebra ${\cal IP}_n^{(\beta)}$
by the two-sided ideal generated by the elements
\begin{gather*}
\{u_i u_j u_i, \, \forall \, i \not= j \}.
\end{gather*}
\end{de}

For example,
\begin{gather*}
\operatorname{Hilb}\big({\cal{PTL}}_4^{(0)},t\big) =(1,3,6,4,1)_{t},\qquad
\operatorname{Hilb}\big({\cal{PTL}}_5^{(0)},t\big) = (1,4,12,16,14,4,2)_t,\\
\operatorname{Hilb}\big({\cal{PTL}}_6^{(0)},t\big) = (1,5,20,40,60,46,32,10,4,1)_{t},\\
\operatorname{Hilb}\big({\cal{PTL}}_7^{(0)},t\big)=
(1,6,30,80,170,216,238,152,96,44,14,4,2)_{t}.
\end{gather*}
One can show that
\begin{gather*}
\deg_{t}
\operatorname{Hilb}\big({\cal{PTL}}_n^{(0)},t\big) = \left[\frac{n^2}{4} \right],
\end{gather*}
 and
\begin{gather*}
{\rm Coef\/f}_{t^{\max}}
\operatorname{Hilb}({\cal{PTL}}_n,t) =\begin{cases} 1, & \text{if $n$ is even},\\
2, & \text{if $n$ is odd}.
\end{cases}
\end{gather*}

\begin{de}\label{def2.20} The \textit{nilCoxeter algebra} ${\cal {NC}}_n$ is def\/ined to be the
quotient of the nilplactic algebra ${\cal {NP}}_n$ by the two-sided ideal
generated by elements $\{u_i u_j-u_j u_i,\, |i-j| \ge 2 \}$.
\end{de}

Clearly the nilCoxeter algebra ${\cal NC}_n$ is a quotient of the modif\/ied
plactic algebra ${\cal MP}_n$ by the two-sided ideal generated by the elements
 $\{u_i u_j -u_j u_i, \, |i-j| \ge 2 \}$.

\begin{de}\label{def2.21} The {\it idCoxeter algebra} ${\cal {IC}}_n^{(\beta)}$ is
def\/ined to be the quotient of the idplactic algebra
${\cal {IP}}_n^{(\beta)}$ by the two-sided ideal generated by the
elements $\{u_i u_j-u_j u_i,\, |i-j| \ge 2 \}$.
\end{de}

It is well-known, see, e.g., \cite{AL}, that the algebra
${\cal NC}_n$ and ${\cal IC}_n^{(\beta)}$
has dimension $n !$, and the elements $\{u_w := u_{i_{1}} \cdots u_{i_{\ell}} \}$, where
$w =s_{i_{1}} \cdots s_{i_{\ell}}$ is any reduced decomposition of
$w \in {\mathbb S}_n$, form a {\it basis} in the nilCoxeter and idCoxeter
algebras ${\cal NC}_n$ and ${\cal IC}_n^{(\beta)}$.

\begin{rem}\label{2.22} There is a common generalization of the algebras def\/ined
above which is due to S.~Fomin and C.~Greene~\cite{FG}.
 Namely, def\/ine generalized plactic algebra
${\widetilde {\cal P}}_n$ to be an associative algebra generated by elements
$u_1,\dots,u_{n-1}$, subject to the relations (PL2) and relations
\begin{gather}\label{equation2.3}
u_ju_i(u_i+u_j)=(u_i+u_j)u_ju_i,\qquad i < j.
\end{gather}
The relation \eqref{equation2.3} can be written also in the form
\begin{gather*}
u_j(u_iu_j-u_ju_i)=(u_iu_j-u_ju_i)u_i, \qquad i < j.
\end{gather*}
\end{rem}

\begin{Theorem}[\cite{FG}]\label{theorem2.23} For each pair of numbers
$1 \le i < j \le n$ define
\begin{gather*}
 A_{i,j}(x)= \prod_{k=j}^{i}(1+x u_k).
\end{gather*}
Then the elements $A_{i,j}(x)$ and $A_{i,j}(y)$ commute in the generalized
plactic algebra ${\widetilde {\cal P}}_n$.
\end{Theorem}

\begin{cor} \label{cor2.24} Let $1 \le i < j \le n$ be a pair of numbers.
Noncommutative elementary polynomials
$e_a^{ij}:= \sum\limits_{j \ge i_1 \ge \cdots \ge i_k \ge i}u_{i_1} \cdots u_{i_a}$,
 $i \le a \le j$, generate a commutative subalgebra ${\cal C}_{i,j}$ of
rank $j-i+1$ in the plactic algebra ${\cal P}_n$.

Moreover, the algebra ${\cal C}_{1,n}$ is a maximal commutative subalgebra of
${\cal P}_n$.
\end{cor}

To establish Theorem~\ref{theorem2.23}, we are going to prove more general result. To start
 with, let us def\/ine {\it generic plactic algebra} ${\mathfrak{P}}_n$.

\begin{de}\label{def2.25} The {\it generic plactic algebra ${\mathfrak{P}}_n$} is an
associative algebra over $\Z$ generated by
$\{e_1,\dots, e_{n-1} \}$ subject to the set of relations
\begin{alignat}{3}\label{equation2.4}
 & e_j (e_i,e_j) = (e_i,e_j) e_i, \qquad && \text{if}\quad i < j,&
\\
\label{equation2.5}
 &(e_j,(e_i,e_k)) = 0, \qquad && \text{if} \quad i < j < k,&
\\
\label{equation2.6}
& (e_j,e_k)(e_i,e_k)=0,\qquad && \text{if}\quad i < j < k.&
\end{alignat}
Hereinafter we shell use the notation $(a,b):= [a,b]:=a b-b a$.
\end{de}

{\sloppy Clearly seen that relations \eqref{equation2.4}--\eqref{equation2.6} are consequence of the plactic
relations~(PL1) and~(PL2), but not vice versa.

}

\begin{Theorem}\label{theorem2.26} Define
\begin{gather*}
A_{n}(x)= \prod_{k=j}^{1}(1+x e_k).
\end{gather*}
Then the elements $A_{n}(x)$ and $A_{n}(y)$ commute in the generic
plactic algebra ${\mathfrak{P}}_n$.
Moreover the elements $A_n(x)$ and $A_n(y)$ commute if and only if the
generators $\{e_1,\ldots,e_{n-1} \}$ satisfy the relations
\eqref{equation2.4}--\eqref{equation2.6}.
\end{Theorem}

\begin{proof} For $n=2,3$ the statement of Theorem \ref{theorem2.26} is obvious. Now assume
that the statement of Theorem \ref{theorem2.26} is true in the algebra $\mathfrak{P}_n$. We
have to prove that the commutator $[A_{n+1}(x),A_{n+1}(y)]$ is equal to zero.
First of all, $A_{n+1}(x)=(1+ x e_n) A_{n}(x)$. Therefore
\begin{gather*}
 [A_{n+1}(x),A_{n+1}(y)] = (1+x e_n) [A_{n}(x),1+y e_n] A_n(y) -
[A_n(y), 1+x e_n] A_n(x).
\end{gather*}
Using the standard identity $[a b,c] =a [b,c]+[a,c] b$, one f\/inds that
\begin{gather*}
\frac{1}{x y} [A_{n}(x), 1+y e_n] = \sum_{i=}^{n-1} \prod_{a=n-1}^{i+1}
 (1+x e_a) (e_i,e_n) \prod_{a=i-1}^{1} (1+x e_a).
\end{gather*}
Using relations \eqref{equation2.4} we can move the commutator $(e_i,e_n)$ to
the left, since $i < a< n$, till we meet the term $(1+ x e_n)$. Using
relations \eqref{equation2.5} we see that $(1+x e_n) (e_i,n)= (e_i,n)(1+ x e_i)$. Therefore we come to the
following relation
\begin{gather*}
\frac{1}{x y} [A_{n}(x), 1+y e_n] \\
\qquad{} =\sum_{i=n-1}^{1} (e_i,e_n) \left(
(1+ x e_i) \prod_{a=n-1 \atop a \not= i}^{1} (1+x e_a) A_n(y) -
(1+ y e_i) \prod_{a=n-1 \atop a \not= i}^{1} (1+y e_a) A_n(x) \right).
\end{gather*}
Finally let us observe that according to the relation \eqref{equation2.6},
\begin{gather*}
 (e_i,e_n)\bigl( (1+x e_i)(1+x e_{n-1}) -(1+ x e_{n-1})(1+x e_i)\bigr) =
x^2 (e_i,e_n) (e_i, e_{n-1}) = 0.
\end{gather*}
Indeed,
\begin{gather*}
(e_i,e_n)(e_i,e_{n-1})= (e_i,e_{n}) e_i e_{n-1} - (e_i,e_{n}) e_{n-1}
e_{i} = e_n e_{n-1} (e_i,e_n) - e_{n-1} e_{n} (e_i,e_n) =0.
\end{gather*} Therefore
\begin{gather*}
\frac{1}{x y} [A_{n}(x), 1+y e_n] = \left(\sum_{i={n-1}}^{1} (e_i,e_n)
\right) [A_n(x),A_n(y)] =0
\end{gather*}
 according to the induction assumption.

Finally, if $i < j$, then $(e_i+e_j,e_j e_i) =0 \Longleftrightarrow \eqref{equation2.4}$;
if $i < j < k$, and the relations~\ref{equation2.4} hold, then $(e_i+e_j+e_k,
e_j e_i+e_k e_j+e_k e_i)=0 \Longleftrightarrow \eqref{equation2.5}$;
if $i < j < k$, and relations~\eqref{equation2.4} and~\eqref{equation2.5} hold, then $(e_i+e_j+e_k,
e_k e_j e_i) =0 \Longleftrightarrow \eqref{equation2.6}$;
the relations $(e_j e_i+e_k e_j+e_k e_i,e_k e_j e_i)=0$ are consequences of
 the above ones.
\end{proof}

Let $T$ be a semistandard tableau and $w(T)$ be the column reading word
corresponding to the tableau~$T$. Denote by $R(T)$ (resp. $IR(T)$) the set of
 words which are plactic (resp.\ idplactic) equivalent to $w(T)$.
Let ${\bf a}=(a_1,\dots,a_n) \in R(T)$, where $n:=|T|$
 (resp. ${\bf a}=(a_1,\dots,a_m) \in IR(T)$, where $m \ge |T|$).

\begin{de}[compatible sequences ${\bf b}$]\label{def2.27} Given a word ${\bf a} \in R(T)$
 (resp.~${\bf a} \in IR(T)$), denote by
$C(\bf a)$ (resp.~$IC(\bf a)$) the set of sequences of positive integers,
called compatible sequences,
${\bf b}:=(b_1 \le b_2 \le \cdots \le b_m)$ such that
\begin{gather*}%\label{equation2.7}
 b_i \le a_i, \qquad \text{and if} \quad a_i \le a_{i+1}, \qquad \text{then}\quad b_i < b_{i+1}.
\end{gather*}
\end{de}

Finally, def\/ine the set $C(T)$ (resp.~$IC(T)$) to be the union
$\bigcup C({\bf a})$ (resp. the union $\bigcup IC({\bf a})$),
 where ${\bf a}$ runs over all words which are plactic (resp. idplactic)
 equivalent to the word~$w(T)$.
\begin{ex}\label{exam2.28} Take $ T= \begin{matrix}2 &3 \cr 3 & \cr
\end{matrix}$.
The corresponding tableau
word is $w(T)=323$. We have $R(T)=\{232, 323 \}$ and $IR(T)= R(T) \bigcup
 \{2323, 3223, 3232, 3233, 3323, 32323,\dots \}$. Moreover,
\begin{gather*}
 C(T)= \left \{\begin{matrix} {\bf a}\colon & 232 & 323 & 323 & 323 & 323 \cr
{\bf b}\colon & 122 & 112 & 113 & 123 & 223
\end{matrix} \right \}, \\
IC(T)= C(T) \bigcup \left
\{\begin{matrix} {\bf a}\colon & 2323 & 3223 & 3232 & 3233 & 3323 & 32323 \cr
 {\bf b}\colon & 1223 & 1123 & 1122 & 1123 & 1223 & 11223
\end{matrix}\right\}.
\end{gather*}
\end{ex}

Let $ {\mathfrak P}:= {\mathfrak{P}}_n := \{p_{i,j}, \, i \ge 1,\,j \ge 1,\, 2 \le i+j \le n+1 \}$
be the set of (mutually commuting) variables.

\begin{de}\label{def2.29}\quad
\begin{enumerate}\itemsep=0pt
\item[(1)] Let $T$ be a semistandard tableau, and
$n:=|T|$. Def\/ine the double key polynomial ${\cal K}_{T}({\mathfrak P})$
corresponding to the tableau~$T$ to be
\begin{gather*}%\label{equation2.8}
{\cal K}_{T}(\mathfrak{P})= \sum_{{\bf b} \in C(T)} \prod_{i=1}^{n}
 p_{b_{i},a_i-b_i+1}.
\end{gather*}

\item[(2)] Let $T$ be a semistandard tableau, and $n:=|T|$. Def\/ine the
double key Grothendieck polynomial ${\cal {GK}}_{T}(\mathfrak{P})$
corresponding to the tableau $T$ to be
\begin{gather*}%\label{equation2.9}
{\cal {GK}}_{T}({\mathfrak P})= \sum_{{\bf b} \in IC(T)} \prod_{i=1}^{m}
 p_{b_{i},a_i-b_i+1}.
\end{gather*}
\end{enumerate}
\end{de}

In the case when $p_{i,j}= x_i+y_j$, $\forall\, i,j$, where $X= \{x_1,\ldots,x_n \}$ and $Y = \{y_1, \ldots,y_n \}$ denote two sets of variables, we will
 write ${\cal K }_{T}(X,Y)$, ${\cal GK}_{T}(X,Y),\dots$, instead of
 ${\cal K}_{T}({\mathfrak P})$,
${\cal GK}_{T}({\mathfrak P}), \ldots $.

\begin{de}\label{def2.30} Let $T$ be a semistandard tableau, denote by $\alpha(T)=
(\alpha_1,\dots,\alpha_{n})$ the exponent of the {\it smallest} monomial
in the set $\big\{x^{\bf b}:= \prod\limits_{i=1}^{m} x_{i}^{b_{i}}, \,{\bf b} \in C(T) \big\}$
 with respect to the lexicographic order.
\end{de}
We will call the composition $\alpha(T)$ to be the bottom code of tableau $T$.

\section{Divided dif\/ference operators}\label{section3}

In this subsection we remind some basic properties of divided dif\/ference
operators will be put to use in subsequent sections. For more details,
see \cite{Ma1}.

Let $f$ be a function of the variables $x$ and $y$ (and possibly other
variables), and $\eta \not= 0$ be a~parameter. Def\/ine the {\it divided
difference operator $\partial_{xy}(\eta)$} as follows
\begin{gather*}%\label{equation3.1}
\partial_{xy}(\eta) f(x,y)={f(x,y)-f(\eta^{-1} y,\eta x)
\over x-\eta^{-1} y}.
\end{gather*}
Equivalently, $(x-\eta^{-1} y) \partial_{xy}(\eta)= 1-s_{xy}^{\eta}$,
where the operator $s_{xy}^{\eta}$ acts on the variables $(x,y, \ldots)$
according to the rule: $s_{xy}^{\eta}$ transforms the pair $(x,y)$ to
$(\eta^{-1} y,\eta x)$, and f\/ixes all other variables. We set by def\/inition,
 $s_{yx}^{\eta}:=s_{xy}^{{\eta^{-1}}}$.

The operator $\partial_{xy}(\eta)$ takes polynomials to polynomials and has
degree~$-1$. The case $\eta =1$ corresponds to the {\it Newton divided
difference operator} $\partial_{xy}:=\partial_{xy}(1)$.
\begin{lem}\label{lem3.1} \quad
\begin{enumerate}\itemsep=0pt
\item[$(0)$] $s_{xy}^{\eta} s_{xz}^{\eta \xi}=s_{yz}^{\xi} s_{xy}^{\eta}$,
 $s_{xy}^{\eta} s_{xz}^{\eta \xi} s_{yz}^{\xi}=
s_{yz}^{\xi} s_{xz}^{\eta \xi} s_{xy}^{\eta}$,

\item[$(1)$] $\partial_{yx}(\eta)= -\eta \partial_{xy}(\eta^{-1})$,
$s_{xy}^{\eta} \partial_{yz}(\xi)=
\eta^{-1} \partial_{xz}(\eta \xi) s_{xy}^{\eta}$,

\item[$(2)$] $\partial_{xy}(\eta)^2= 0$,

\item[$(3)$] $($three term relation$)$
$\partial_{xy}(\eta) \partial_{yz}(\xi)=
\eta^{-1} \partial_{xz}(\eta \xi) \partial_{xy}(\eta)+
\partial_{yz}(\xi) \partial_{xz}(\eta \xi)$.

\item[$(4)$] $($twisted Leibniz rule$)$
$ \partial_{xy}(\eta) (fg)=
 \partial_{xy}(\eta) (f) g+s_{xy}^{\eta}(f) \partial_{xy}(\eta) (g)$,

\item[$(5)$] $($crossing relations, cf. {\rm \cite[formula~(4.6)]{FK})}
\begin{itemize}\itemsep=0pt
\item $x \partial_{xy}(\eta)=\eta^{-1} \partial_{xy}(\eta) y+1,
y \partial_{xy}(\eta)= \eta \partial_{xy}(\eta) x-\eta$,

\item $\partial_{xy}(\eta) y \partial_{yz}(\xi)=
\partial_{xz}(\eta \xi) x \partial_{xy}(\eta)+
\xi^{-1}\partial_{yz}(\xi) z \partial_{xz}(\eta \xi)$,
\end{itemize}
\item[$(6)$] $\partial_{xy} \partial_{xz}
\partial_{yz} \partial_{xz}=0$.
\end{enumerate}
\end{lem}

Let $x_1,\ldots,x_n$ be independent variables, and let
$P_n:= \Q[x_1,\ldots,x_n]$. For each $i < j $ put
$\partial_{ij}:=\partial_{x_i x_j}(1)$ and $\partial_{ji}=-\partial_{ij}$.
 From Lemma~\ref{lem3.1} we have
\begin{gather*}
\partial_{ij}^2=0, \qquad
\partial_{ij} \partial_{jk}+\partial_{ki} \partial_{ij}+\partial_{jk}
\partial_{ki}=0, \\
\partial_{ij} x_j \partial_{jk}+\partial_{ki} x_i \partial_{ij}+
\partial_{jk} x_k \partial_{ki},
\qquad \text{if} \quad \text{$i$, $j$, $k$ are distinct}.
\end{gather*}

It is interesting to consider also an {\it additive or affine} analog
$\partial_{xy}[k]$
of the divided dif\/ference operators $\partial_{xy}(\eta)$, namely,
\begin{gather*}%\label{equation3.2}
\partial_{xy}[k](f(x,y))={f(x,y)-f(y-k,x+k) \over x-y+k}.
\end{gather*}
One has $\partial_{yx}[k] = -\partial_{xy}[-k]$, and
$\partial_{xy}[p] \partial_{yz}[q]=
\partial_{xz}[p+q] \partial_{xy}[p]+
\partial_{yz}[q] \partial_{xz}[p+q]$.

{\it A short historical comments in order.}
 As far as I know, the divided dif\/ference operator had been invented
by I.~Newton in/around the year~$1687$, see, e.g., \cite[Ref.~\protect{[142]}]{AL}. Since that time the literature concerning divided dif\/ferences, a plethora of its
generalizations and applications in dif\/ferent f\/ields of mathematics and
physics, grows exponentially and essentially is immense. To the best of our
knowledge, the f\/irst systematic use of the isobaric divided dif\/ference operators, namely $ \pi_i(f)= \partial_i(x_{i} f)$, and $\overline{\pi}_{i}:=\pi_{i}-1$, goes back to papers by M.~Demazure concerning the study of desingularization of Schubert varieties and computation of characters of certain modules which
is nowadays called {\it Demazure modules}, see, e.g.,~\cite{De} and the literature quoted therein;
the f\/irst systematic use and applications of divided dif\/ference operators to
the study of cohomology rings of partial f\/lag varieties $G/P$ and
description of Schubert classes inside the former, goes back to a paper by
I.N.~Berstein, I.M.~Gelfand and S.I.~Gelfand~\cite{BGG};
it is A.~Lascoux and M.-P.~Sch\"{u}tzenberger who had discovered
{\it a polynomial representative} of a Schubert class in the cohomology ring
of the type~$A$ complete f\/lag varieties, and developed a rich and beautiful
combinatorial theory of these polynomials, see~\cite{LS9,LS7,LSZ,LS4,LS3} for acquainting the reader with basic ideas and
results concerning the Lascoux--Sch\"{u}tzenber Schubert polynomials.

\section{Schubert, Grothendieck and key polynomials}\label{section4}

Let $w \in {\mathbb S}_n$ be a permutation, $X=(x_1,\dots,x_n)$ and
$Y=(y_1,\dots,y_n)$ be two sets of variables. Denote by
$w_0 \in {\mathbb S}_n$ the longest
permutation, and by $\delta_n= (n-1,n-2,\dots,1)$ the staircase
partition. For each partition $\lambda$ def\/ine
\begin{gather}\label{equation4.1}
R_{\lambda}(X,Y):=
\prod_{(i,j) \in \lambda} (x_i+y_j).
\end{gather}
 For $i=1,\dots,n-1$, let $s_i=(i,i+1) \in {\mathbb S}_n$ denote the simple
transposition that interchanges $i$ and $i+1$ and f\/ixes all other elements of
the set $\{1,\dots,n \}$. If $\alpha=
(\alpha_1,\dots,\alpha_i,\alpha_{i+1},\dots,\alpha_n)$ is a~composition, we
will write
\begin{gather*}%\label{equation4.2}
s_i\alpha = (\alpha_1,\dots,\alpha_{i+1},\alpha_{i},\dots,\alpha_n).
\end{gather*}

\begin{de}[cf.~\cite{AL} and the literature quoted therein]\label{def4.1}\quad
\begin{itemize}\itemsep=0pt
\item For each permutation $w \in {\mathbb S}_n$ {\it the double
Schubert polynomial} ${\s}_w(X,Y)$ is def\/ined to be
\begin{gather*}%\label{eqiuarion4.3}
 \partial_{w^{-1} w_{0}}^{(x)}(R_{\delta_{n}}(X,Y)).
\end{gather*}
\item
Let $\alpha$ be a composition.
 \textit{The key polynomials} $K[\alpha](X)$ are def\/ined recursively
 as follows:
if $\alpha$ is a partition, then $K[\alpha](X)= x^{\alpha}$;
otherwise, if $\alpha$ and $i$ are such that $\alpha_i < \alpha_{i+1}$, then
\begin{gather*}%\label{equation4.4}
 K[s_i(\alpha)](X)= \partial_{i} \big( x_i K[\alpha](X) \big).
\end{gather*}
\item \textit{The reduced key polynomials}
${\widehat K}[\alpha](X)$ are def\/ined
recursively as follows:
if $\alpha$ is a partition, then ${\widehat K}[\alpha](X)= K[\alpha](X)=
x^{\alpha}$;
otherwise, if $\alpha$ and $i$ are such that $\alpha_i < \alpha_{i+1}$, then
\begin{gather*}%\label{equation4.5}
 {\widehat K}[s_i(\alpha)](X)=
 x_{i+1} \partial_{i} \big( {\widehat K}[\alpha](X) \big).
\end{gather*}
\item For each permutation $w \in {\mathbb S}_n$ the {\it double
$\beta$-Grothendieck polynomial} ${\cal G}_w^{\beta}(X,Y)$ is def\/ined
recursively as follows:
if $w=w_0$ is the longest element, then ${\cal G}_{w_0}(X,Y)=
R_{\delta_{n}}(X,Y)$;
if $w$ and $i$ are such that $w_i > w_{i+1}$, i.e., $l(ws_i)=l(w)-1$, then
\begin{gather*}%\label{equation4.6}
{\cal G}_{ws_i}^{\beta}(X,Y)=
\partial_i^{(x)} \big( {(1+ \beta x_{i+1}) {\cal G}_w^{\beta} (X,Y) } \big).
\end{gather*}
\item For each permutation $w \in {\mathbb S}_n$ the {\it double dual
$\beta$-Grothendieck polynomial} ${\cal H}_w^{\beta}(X,Y)$ is def\/ined
recursively as follows:
if $w=w_0$ is the longest element, then ${\cal H}_{w_0}(X,Y) =
R_{\delta_{n}}(X,Y)$;
if $w$ and $i$ are such that $w_i > w_{i+1}$, i.e., $l(ws_i)=l(w)-1$, then
\begin{gather*}%\label{equation4.7}
{\cal H}_{ws_i}^{\beta}(X,Y)=
 (1+ \beta x_{i}) \partial_i^{(x)} \big( {\cal H}_w^{\beta}(X,Y) \big).
 \end{gather*}
\item \textit{The key $\beta$-Grothendieck polynomials} $\operatorname{KG}[\alpha](X;\beta)$ are def\/ined recursively as follows\footnote{In the case $\beta= -1$ divided dif\/ference operators
$D_i:=\partial_{i}(x_i-x_i x_{i+1})$ \cite[formula~(6)]{L3} had been
used by A.~Lascoux to describe the transition on Grothendieck polynomials, i.e.,
stable decomposition of any Grothendieck polynomial corresponding to a~permutation $w \in {\mathbb{S}}_n$. into a sum of Grasmannian ones
corresponding to a collection of {\it Grasmannin} permutations
$v_{\lambda} \in {\mathbb{S}}_{\infty}$, see~\cite{L3} for details. The above
mentioned operators~$D_i$ had been used in~\cite{L3} to construct a basis
$\Omega_{\alpha} \vert \alpha \in \Z_{\ge 0}$ that deforms the basis which
is built up from the Demazure (known also as key) polynomials. Therefore
polynomials $\operatorname{KG}[\alpha](X;\beta= -1)$ coincide with those introduced by A.~Lascoux in~\cite{L3}.
 In \cite{RY} the authors give a conjectural construction for polynomials
$\Omega_{\alpha}$ based on the use of {\it extended Kohnert moves}, see, e.g., \cite[Appendix by N.~Bergeron]{Ma} for def\/inition of the Kohnert moves. We
state {\it conjecture} that
\begin{gather*}%\label{equation4.8}
J_{\alpha}^{(\beta)} = \operatorname{KG}[\alpha](X;\beta),
\end{gather*}
where polynomials $J_{\alpha}^{\beta)}$ are def\/ined in \cite{RY} using the
$K$-theoretic versions of the Kohnert moves. For $\beta = -1$ this Conjecture
has been stated by C.~Ross and A.~Yong in \cite{RY}. It seems an interesting
problem to relate the $K$-theoretic Kohnert moves with certain moves of f\/irst
introduced in~\cite{FK1}.}:
if $\alpha$ is a partition, then $\operatorname{KG}[\alpha](X;\beta)= x^{\alpha}$;
otherwise, if $\alpha$ and $i$ are such that $\alpha_i < \alpha_{i+1}$, then
\begin{gather*}%\label{equation4.9}
 \operatorname{KG}[s_i(\alpha)](X;\beta)=
\partial_i \big( (x_i+\beta x_i x_{i+1}) \operatorname{KG}[\alpha](X;\beta) \big).
\end{gather*}
\item \textit{The reduced key $\beta$-Grothendieck polynomials}
${\widehat {\operatorname{KG}}}[\alpha](X;\beta)$ are def\/ined recursively as follows:
if $\alpha$ is a partition, then ${\widehat {\operatorname{KG}}}[\alpha](X;\beta)= x^{\alpha}$;
otherwise, if~$\alpha$ and~$i$ are such that $\alpha_i < \alpha_{i+1}$, then
\begin{gather*}%\label{equation4.10}
{\widehat {\operatorname{KG}}}[s_i(\alpha)](X;\beta)=
 (x_{i+1}+\beta x_i x_{i+1}) \partial_i
\big( {\widehat {\operatorname{KG}}}[\alpha](X;\beta) \big).
\end{gather*}
\end{itemize}
\end{de}

For brevity, we will write $\operatorname{KG}[\alpha](X)$ and ${\widehat {\operatorname{KG}}}[\alpha](X)$ instead of $\operatorname{KG}[\alpha](X;\beta)$ and ${\widehat {\operatorname{KG}}}[\alpha](X;\beta)$.

\begin{rem}\label{rem4.2} We can also introduce polynomials ${\cal Z}_w$, which are
def\/ined recursively as follows:
if $w=w_0$ is the longest element, then ${\cal Z}_{w_0}(X)=
x^{\delta_{n}}$;
if~$w$ and~$i$ are such that $w_i > w_{i+1}$, i.e., $l(ws_i)=l(w)-1$, then
\begin{gather*}%\label{equation4.11}
 {\cal Z}_{ws_i}(X)=
\partial_i \big( (x_{i+1}+x_ix_{i+1}) {\cal Z}_w(X) \big).
\end{gather*}
However, one can show that
\begin{gather*}
{\cal Z}_w(x_1,\dots,x_n)=
(x_1 \cdots x_n)^{n-1}{\cal G}_{w_{0}w w_{0}}\big(x_n^{-1},\dots,x_1^{-1}\big).
\end{gather*}
\end{rem}

\begin{Theorem}\label{theorem4.3}\sloppy The polynomials ${\s}_w(X,Y)$, $K[\alpha](X)$,
$\widehat {K}[\alpha](X)$, ${\cal G}_w(X,Y)$, ${\cal H}_w(X,Y)$,
$\operatorname{KG}[\alpha](X)$ and $\widehat {\operatorname{KG}}[\alpha](X)$
 have nonnegative integer coefficients.
\end{Theorem}

The key step in a proof of Theorem~\ref{theorem4.3} is an observation that for a~given~$n$
the all algebras involved in the def\/inition of the polynomials listed in that theorem, happen to be a suitable quotients of the reduced plactic algebra~${\cal{P}}_n$, and can be extracted from the Cauchy kernel associated with the
algebra ${\cal{P}}_n$ (or that~${\cal{P}}_{n,n}$).

We will use notation ${\s}_w(X)$, ${\cal G}_w(X)$, \dots, for
polynomials ${\s}_w(X,0)$, ${\cal G}_w(X,0)$, \dots.

% $\bullet$ \textit{Di Francesco--Zinn-Justin polynomials})

\begin{de}\label{def4.4} For each permutation $w \in {\mathbb S}_n$ the
\textit{Di~Francesco--Zinn-Justin polynomials} ${\cal {DZ}}_w(X)$ are def\/ined
recursively as follows:
if $w$ is the longest element in ${\mathbb S}_n$, then
${\cal {DZ}}_w(X)= R_{\delta_{n}}(X,0)$;
otherwise, if $w$ and $i$ are such that $w_i > w_{i+1}$, i.e.,
$l(ws_i)=l(w)-1$, then
\begin{gather*}%\label{equation4.12}
 {\cal {DZ}}_{ws_i}(X)= \big( (1+x_i) \partial_{i}^{(x)} +
\partial_i^{(x)}(x_{i+1}+x_ix_{i+1}) \big) {\cal {DZ}}_w(X).
\end{gather*}
\end{de}

\begin{con}\label{conj4.5} \quad
\begin{enumerate}\itemsep=0pt
\item[$(1)$] Polynomials ${\cal {DZ}}_w(X)$ have nonnegative integer
coefficients.
\item[$(2)$] For each permutation $w \in {\mathbb S}_n$ the polynomial
${\cal {DZ}}_w(X)$ is a linear combination of key polynomials
$K[\alpha](X)$ with nonnegative integer coefficients.
\end{enumerate}
\end{con}
As for def\/i
nition of the double Di Francesco--Zinn-Justin polynomials
${\cal {DZ}}_w(X,Y)$ they are well def\/ined, but may have negative coef\/f\/icients.

%$\bullet$ (\textit{Hecke--Grothendieck polynomials})

 Let $\beta$ and $\alpha$ be two parameters, consider divided
dif\/ference operator
\begin{gather*}%\label{equation4.13}
T_i :=T_{i}^{\beta,\alpha} = -\beta +((\beta+\alpha) x_i +1 +\beta \alpha x_{i} x_{i+1}) \partial_{i,i+1}.
\end{gather*}
\begin{de}\label{def4.6}
Let $w \in {\mathbb{S}}_n$, def\/ine {\it Hecke--Grothendieck} polynomials
${\cal{KN}}_{w}^{\beta,\alpha}(X_n)$ to be
\begin{gather*}
 {\cal{KN}}_{w}^{(\beta,\alpha)}(X_n) := T_{w}^{\beta,\alpha}
\big(x^{\delta_{n}}\big),
\end{gather*}
where as before $x^{\delta_{n}}:= x_{1}^{n-1} x_{2}^{n-2} \cdots x_{n-1}$;
if $u \in {\mathbb{S}}_n$, then set
\begin{gather*}%\label{equation4.14}
 T_{u}^{\beta,\alpha} := T_{i_{1}}^{\beta,\alpha} \cdots T_{i_{\ell}}^{\beta,\alpha},
\end{gather*}
where $w =s_{i_{1}} \cdots s_{i_{\ell}}$ is any reduced decomposition of a
permutation taken.

%$\bullet$
More generally, let $\beta$, $\alpha$ and $\gamma$ be parameters,
consider divided dif\/ference operators
\begin{gather*}
T_{i}:=T_{i}^{\beta,\alpha,\gamma} = - \beta +((\alpha+\beta+\gamma)
x_i +\gamma x_{i+1} + 1 \\
\hphantom{T_{i}:=T_{i}^{\beta,\alpha,\gamma} =}{} +(\beta+\gamma)(\alpha+ \gamma) x_{i} x_{i+1})
 \partial_{i,i+1},\qquad i=1,\ldots,n-1.
\end{gather*}
For a permutation $w \in {\mathbb{S}}_n$ def\/ine polynomials
\begin{gather*}%\label{equation4.15}
{\cal{KN}}_{w}^{(\beta,\alpha,\gamma)}(X_n):= T_{i_{1}}^{\beta,\alpha,\gamma} \cdots T_{i_{\ell}}^{\beta,\alpha,\gamma} \big(x^{\delta_{n}}\big),
\end{gather*}
where $w=s_{i_{1}} \cdots s_{i_{\ell}}$ is any reduced decomposition of~$w$.
\end{de}

\begin{rem}\label{rem4.7} A few comments in order.
\begin{enumerate}\itemsep=0pt
\item[$(a)$] The divided dif\/ference operators $\big\{T_{i}:=T_{i}^{(\beta,\alpha,\gamma)}, \,i=1,
\dots,n-1 \big\}$ satisfy the following relations
\begin{itemize}\itemsep=0pt
\item (Hecke relations)
\begin{gather*}
T_i^2 =(\alpha- \beta) T_{i} + \alpha \beta,
\end{gather*}
\item (Coxeter relations)
\begin{gather*}
T_{i} T_{i+1}T_{i}= T_{i+1} T_{i} T_{i+1},
T_{i}T_{j}=T_{j} T_{i},\qquad \text{if}\quad
|i-j| \ge 2.
\end{gather*}
\end{itemize}
Therefore the elements $T_{w}^{\beta,\alpha, \gamma}$ are well def\/ined for any
$w \in {\mathbb{S}}_n$.
\begin{itemize}\itemsep=0pt
\item (Inversion)
\begin{gather*}
(1+x T_i)^{-1} = \frac{1+(\alpha - \beta) x - x T_{i}}{(1-\beta x)(1 + \alpha x)}.
\end{gather*}
\end{itemize}
\item[$(b)$] Polynomials ${\cal{KN}}_{w}^{(\beta,\alpha,\gamma)}$ constitute a
common generalization of
\begin{itemize}\itemsep=0pt
\item the $\beta$-Grothendieck polynomials, namely, ${\cal{G}}_{w}^{(\beta)} = {\cal{KN}}_{w_{0} w^{-1}}^{(\beta, \alpha =0,\gamma=0)}$,

\item the Di Francesco--Zinn-Justin polynomials,
namely, ${\cal{DZ}}_{w} ={\cal{KN}}_{w}^{(\beta=\alpha=1,\gamma=0)}$,

\item the dual $\alpha$-Grothendieck polynomials,
 namely, ${\cal{KN}}_{w_{0} w^{-1}}^{(\beta=0, \alpha,\gamma=0)} ={\cal{H}}_{w}^{\alpha}(X)$.
 \end{itemize}
 \end{enumerate}
\end{rem}

\begin{pr}\label{prop4.8} \quad
\begin{itemize}\itemsep=0pt
\item $($Duality$)$ Let $w \in {\mathbb{S}}_n$, $\ell = \ell(w)$ denotes
its length, then $(\alpha \beta \not= 0)$
\begin{gather*}%\label{equation4.16}
{\cal{KN}}_{w}^{(\beta,\alpha)}(1) = (\beta \alpha)^{\ell}
{\cal{KN}}_{{w^{-1}}}^{(\alpha^{-1},\beta^{-1})}(1).
\end{gather*}

\item $($Stability$)$ Let $w \in \mathbb{S}_n$ be a permutation and
$w=s_{i_{1}} s_{i_{2}} \cdots s_{i_{\ell}}$ be any its reduced decomposition.
Assume that $ i_{a} \le n-3$, $\forall\, 1\le a \le \ell$, and define permutation $\widetilde{w} := s_{i_{1}+1} s_{i_{2}+1} \cdots s_{i_{\ell}+1} \in \mathbb{S}_n$. Then
\begin{gather*}
{\cal{KN}}_{w}^{(\beta,\alpha)}(1) = {\cal{KN}}_{\widetilde{w}}^{(\beta, \alpha)}(1).
\end{gather*}
\end{itemize}
\end{pr}

It is well-known that
\begin{itemize}\itemsep=0pt
\item the number ${\cal{KN}}_{w_{0}}^{(\beta=1,\alpha=1)}(1)$
 is equal to the degree of the variety of pairs commuting matrices of size $n \times n$ \cite{DZ2, Knu}.

\item the bidegree of the af\/f\/ine homogeneous variety $V_w$, $w \in {\mathbb{S}}_n$ \cite{DZ1} is equal to
\begin{gather*}
 A^{{n \choose 2}-\ell(w)} B^{{n \choose 2} +\ell(w)} {\cal{KN}}_{w}^{(\beta =\alpha = A/B)}(1),
\end{gather*}
see \cite{GN, DZ1,Knu} for more details and applications.
\end{itemize}

\begin{con}\label{conj4.9} \quad
\begin{itemize}\itemsep=0pt
\item Polynomials ${\cal {KN}}_{w}^{(\beta,\alpha,\gamma)}(X)$ have
nonnegative integer coefficients
\begin{gather*}
{\cal {KN}}_{w}^{(\beta,\alpha,\gamma)}(X) \in \N[\beta,\alpha, \gamma ] [X_n].
\end{gather*}
\item Polynomials ${\cal {KN}}_{w}^{(\beta,\alpha,\gamma )}(x_1=1,\,
\forall\, i)$ have nonnegative integer coefficients
\begin{gather*}%\label{equation4.17}
{\cal {KN}}_{w}^{(\beta,\alpha,\gamma)}(x_i=1,\forall i) \in
\N[\beta,\alpha, \gamma ].
\end{gather*}
\item Double polynomials
\begin{gather*}
{\cal{KN}}_{w}^{(\beta = 0,\alpha, \gamma)}(X,Y) = T_{w}^{\beta=0,\alpha,\gamma}(x) \prod_{i+j \le n+1 \atop i \ge 1, \, j \ge 1}(x_i+y_j)
\end{gather*}
are well defined and have nonnegative integer coefficients\footnote{Note that the assumption $\beta=0$ is necessary.}.

\item Consider permutation $w=[n,1,2,\dots,n-1] \in \mathbb{S}_n$.
Clearly $w= s_{n-1} s_{n-2} \cdots s_{2} s_{1}$.
The number ${\cal{KN}}_{w}^{(\beta=1,\alpha=1,\gamma=0)}(1)$ is equal to
the number of Schr\"{o}der paths of semilength~$(n-1)$ in which the
$(2,0)$-steps come in $3$ colors and with no peaks at level~$1$,
see {\rm \cite[$A162326$]{SL}} for further properties of these numbers.
\end{itemize}
 \end{con}

It is well-known, see, e.g., \cite[$A126216$]{SL}, that the polynomial
${\cal{KN}}_{w}^{(\beta,\alpha=0)}(1)$ counts the number of {\it dissections} of a convex $(n+1)$-gon according the number of diagonals involved, where as
the polynomial ${\cal{KN}}_{w}^{(\beta,\alpha)}(1)$ (up to a normalization) is
equal to the {\it bidegree} of certain algebraic varieties introduced and
studied by A.~Knutson~\cite{Knu}.

A few comments in order.
\begin{enumerate}\itemsep=0pt
\item[$(a)$] One can consider more general family of polynomials
${\cal{KN}}_{w}^{(a,b,c,d)}(X_n)$ by the use of the divided dif\/ference
operators
$
T_{i}^{a,b,c,d}:= -b+((b+d) x_i + c x_{i+1}+1+d (b+c) x_i x_{i+1}) \partial_{i,i+1}^{x}
$
 instead of that $T_{i}^{\beta,\alpha,\gamma}$. However the
 polynomials ${\cal{KN}}_{w}^{(a,b,c,d)}(1) \in \Z[a,b,c,d]$ may have negative coef\/f\/icients in general. {\it Conjecturally}, to ensure the positivity of
polynomials ${\cal{KN}}_{w}^{(a,b,c,d)}(X_n)$, it is necessary take $d:=a+c+r$. In this case we state conjecture
\begin{gather*}
 {\cal{KN}}_{w}^{(a,b,c,a+c+r)}(X_n) \in \N[a,b,c,r].%\label{equation4.19}
\end{gather*}
 We state more general Conjecture~\ref{conjecture1.0} in introduction. In the present paper we
treat only the case $r=0$, since a combinatorial
meaning of polynomials ${\cal{KN}}_{w}^{(a,b,c,a+c+r)}(1)$ in the case
 $r \not= 0$ is missed for the author.

\item[$(b)$] If $\gamma \not= 0$, the polynomials ${\cal {KN}}_{w}^{(\beta,\alpha,\gamma)}(X_n) \in \Z [\alpha, \beta,\gamma] [X_n]$ may have negative coef\/f\/icients in general.
\end{enumerate}

\begin{Theorem}\label{theorem4.10} Let $T$ be a semistandard tableau and $\alpha(T)$ be its
 bottom code, see Definition~{\rm \ref{def2.27}}. Then
\begin{gather*}
{\cal K}_{T}(X)=K[\alpha(T)](X), \qquad
{\cal {KG}}_{T}(X)=\operatorname{KG}[\alpha(T)](X).
\end{gather*}
\end{Theorem}
 Let $\alpha=(\alpha_1 \le \alpha_2 \le \cdots \le \alpha_r)$ be a
composition, def\/ine partition $\alpha^{+}=(\alpha_r \ge \cdots \ge \alpha_1)$.
\begin{pr}\label{prop4.11} If $\alpha=(\alpha_1 \le \alpha_2 \le \cdots \le \alpha_r)$ is
a composition and $n \ge r$, then
\begin{gather*}
 K[\alpha](X_n)=s_{\alpha^{+}}(X_r).
\end{gather*}
\end{pr}
For example, $K[0,1,2,\dots,n-1]= \prod\limits_{1 \le i < j \le n}(x_i+x_j)$. Note
that
${\widehat K}[0,1,2,\dots,n-1]= \prod\limits_{i=2}^{n}x_i^{i-1}$.
\begin{pr}\label{4.12} If $\alpha=(\alpha_1 \le \alpha_2 \le \cdots \le \alpha_r)$ is
a composition and $n \ge r$, then
\begin{gather*}
\operatorname{KG}[\alpha](X_n)={\cal G}[{\alpha^{+}}](X_r).
\end{gather*}
\end{pr}

For example, $\operatorname{KG}[0,1,2,\dots,n-1]=
\prod\limits_{1 \le i < j \le n}(x_i+x_j+x_ix_j)$.
Note that
\begin{gather*}
{\widehat {\operatorname{KG}}}[0,1,2,\dots,n-1]=
\prod_{i=2}^{n}x_i^{i-1} \prod_{i=1}^{n-1}(1+x_i)^{n-i}.
\end{gather*}

%\begin{comm}\label{comm4.13} ${}$ {\rm

\begin{de}\label{def4.14} Def\/ine degenerate af\/f\/ine $2D$ nil-Coxeter algebra
${\cal{ANC}}_n^{(2)}$ to be an associative algebra over $\Q$ generated by the
set of elements $\{\{u_{i,j} \}_{1 \le i <j \le n},\, x_1,\ldots,
x_n \}$ subject to the set of relations
\begin{itemize}\itemsep=0pt
\item $x_i x_j = x_j x_i$ for all $i \not= j$, $x_i u_{j,k}=u_{j,k}
 x_i$, if $i \not= j, k$,

\item $u_{i,j} u_{k,l} = u_{k,l} u_{i,j}$, if $i$, $j$, $k$, $l$ are pairwise
distinct,

\item ($2D$-Coxeter relations) $u_{i,j} u_{j,k} u_{i,j} =u_{j,k} u_{i,j} u_{j,k}$, if $1 \le i < j < k \le n$,

\item $x_i u_{i,j}=u_{i,j} x_j + 1$, $x_j u_{i,j}=u_{i,j} x_i - 1$.
\end{itemize}
\end{de}

Now for a set of parameters\footnote{By def\/inition, a \textit{parameter} is assumed to be belongs to the
center of the algebra in question.}
 $ A:=(a,b,c,h,e)$ def\/ine elements
\begin{gather*}
T_{ij}:= a+ (b x_i+ c x_j +h+e x_i x_j) u_{i,j}, \qquad i < j.
\end{gather*}
Throughout the present paper we set $T_i:= T_{i,i+1}$.
\begin{lem}\label{lem4.15} \quad
\begin{enumerate}\itemsep=0pt
\item[$(1)$] $T_{i,j}^2 =(2a+b-c) T_{i,j} -a(a+b-c)$,
if $a=0$, then $T_{ij}^2 = (b-c) T_{ij}$.

\item[$(2)$] $2D$-Coxeter relations % Relations
$T_{i,j} T_{j,k} T_{i,j} = T_{j,k} T_{i,j} T_{j,k}$
are valid, if and only if the following relation among parameters
$a$, $b$, $c$, $e$, $h$ holds\footnote{The relation~\eqref{equation4.1} between parameters $a$, $b$, $c$, $e$, $h$ def\/ines a
{\it rational} four-dimensional hypersurface. Its open chart $ \{e h \not= 0
 \}$ contains, for example, the following set (cf.~\cite{L3}): $\{a=p_1
p_4-p_2 p_3, \, b= p_2 p_3, \, c= p_1 p_4, \, e=p_1 p_3$, $h= p_2 p_4 \}$, where $(p_1,p_2,p_3,p_4)$ are arbitrary parameters. However the points $(-b,a+b+c,c,1,(a+c)(b+c)$, $(a,b,c) \in \N^3\}$ do not belong to this set.}
\begin{gather}\label{equation4.19}
(a+b)(a-c)+h e =0.
\end{gather}
\item[$(3)$] Yang--Baxter relations %Relations
$T_{i,j} T_{i,k} T_{j.k}= T_{j,k} T_{i,k} T_{i,j}$
are valid if and only if $b=c=e=0$, i.e., $T_{ij}=a+d u_{ij}$.
\item[$(4)$] $T_{ij}^2=1$ if and only if $a= \pm 1$, $c=b \pm 2$, $he= (b\pm 1)^2$.
\item[$(5)$] Assume that parameters $a$, $b$, $c$, $h$, $e$ satisfy the conditions~\eqref{equation4.19} and
 that $b c+1=h e$. Then
\begin{gather*}
T_{ij} x_i T_{ij} = (h e - b c) x_j+ (h+(a+b)(x_i+x_j)+e x_i x_j) T_{ij}.
\end{gather*}
\end{enumerate}
Some special cases
\begin{itemize}\itemsep=0pt
\item $($representation of affine modified Hecke algebra~{\rm \cite{TO})}
 if $A=(a,-a,c,h,0)$, then $T_{ij} x_{i} T_{ij}$ $ = a c x_{j} + h T_{ij}$, $i < j$,
\item if $A=(-a,a+b+c,c,1,(a+c)(b+c)$, then
\begin{gather*}
T_{ij} x_{i} T_{ij}= a b x_{j}+ (1 +(b+c)(x_i+x_j)+(a+c)(b+c) x_{i} x_{j}) T_{ij}.
\end{gather*}
\end{itemize}
\begin{enumerate}\itemsep=0pt
\item[$(6)$] $($Quantum Yang--Baxter relations, or baxterization of Hecke's
algebra generators.$)$ Assume that parameters $a$, $b$, $c$, $h$, $e$
satisfy the conditions~\eqref{equation4.19} and that $\beta:=2 a+b-c \not= 0$. Then
$($cf.~{\rm \cite{IK, LLT}} and the literature quoted therein$)$
the elements $R_{ij}(u,v):= 1+ \frac{\lambda-\mu}
{\beta \mu} T_{ij}$ satisfy the twisted quantum Yang--Baxter relations
\begin{gather*}
 R_{ij}(\lambda_i,\mu_j) R_{jk}(\lambda_i,\nu_k) R_{ij}(\mu_j,\nu_k)= R_{jk}(\mu_j,\nu_k) R_{ij}(\lambda_i,\nu_k) R_{jk}(\lambda_i,\mu_j), \qquad i < j < k,
\end{gather*}
where $\{\lambda_i,\mu_i,\nu_i \}_{1 \le i \le n}$ are parameters.
\end{enumerate}
\end{lem}

\begin{cor}\label{cor4.16}
 If $(a+b)(a-c)+h e=0$, then for any permutation
$w \in {\mathbb{S}}_n$ the element
\begin{gather*}
T_w:=T_{i_{1}} \cdots T_{i_{l}} \in {\cal{ANC}}_n^{(2)},
\end{gather*}
 where $w=s_{i_{1}} \cdots s_{i_{l}}$ is any reduced decomposition of~$w$, is well-defined.
\end{cor}

\begin{ex}\label{exam4.17} \quad
\begin{itemize}\itemsep=0pt
\item Each of the set of elements
\begin{gather*}
s_i^{( h)}=1+(x_{i+1}-x_i + h) u_{i,i+1}\qquad \text{and} \\
 t_i^{(h)}=-1+(x_i-x_{i+1}+ h(1+x_i)(1+x_{i+1})) u_{ij}, \qquad i=1,\ldots,n-1,
\end{gather*}
by itself generate the symmetric group ${\mathbb{S}}_n$.

\item If one adds the af\/f\/ine elements $s_{0}^{(h)} :=\pi s_{n-1}^{(h)} \pi^{-1}$ and $t_{0}^{(h)}:=\pi t_{n-1}^{(h)} \pi^{-1}$, then each of the set of elements $\big\{s_{j}^{(h)},\, j \in \Z/n \Z \big\}$ and $\big\{t_{j}^{(h)},\, j
 \in \Z/n\Z\big\}$ by itself generate the af\/f\/ine symmetric group
${\mathbb{S}}_n^{\text{af\/f}}$, see \eqref{comm4.21pi} for a def\/inition of the transformation~$\pi$.

\item It seems an interesting problem to classify all rational,
trigonometric and elliptic divided dif\/ference operators satisfying the Coxeter
relations. A general divided dif\/ference ope\-ra\-tor with {\it polynomial
coefficients} had been constructed in~\cite{LSZ}, see also Lemma~\ref{lem4.15}, relation~\eqref{equation4.1}. One can construct a family of {\it rational} representations of the
symmetric group (as well as its af\/f\/ine extension) by ``iterating'' the
transformations $s_j^{(h)}$, $j \in \Z/n\Z$. For example, take parameters
$a$ and $b$, def\/ine \textit{secondary} divided dif\/ference operator
\begin{gather*}
\partial_{xy}^{[a,b]}:= -1 +(b+y-x) \partial_{xy}^{[a]},
\end{gather*}
 where
 \begin{gather*}
 \partial_{xy}^{[a]}:= \frac{1- {{\overline{s}}_{xy}^{(a)}}}{a-x+y},\qquad
{\overline{s}}_{xy}^{(a)}:=-1+(a+x-y) \partial_{xy}.
\end{gather*}
Observe that the set of operators $ \big\{s_{i}^{[a,b]}:= s_{x_{i},x_{i+1}}^{[a,b]},\,
 i \in \Z/n\Z \big\}$ gives rise to a rational representation of the af\/f\/ine
symmetric group ${\mathbb{S}}_n^{\text{af\/f}}$ on the f\/ield of rational functions
$\Z[a,b](X_n)$.
In the special case $a:= A$, $b:= A/h$, $h:=1-\beta/2$ the operators
$s_{i}^{[a,b]}$ coincide with operators $\Theta_i$, $i \in \Z/n\Z$ have been
introduced in~\cite[equation~(4.17)]{KnZ}.
\end{itemize}
\end{ex}

%\begin{comm}\label{comm4.18} {\rm
Let $A=(a,b,c,h,e)$ be a sequence of integers satisfying the conditions~\eqref{equation4.1}. Denote by $\partial_i^{A}$ the divided dif\/ference operator
\begin{gather*}
\partial_{i}^{A}= a+(b x_i+c x_{i+1}+h+e x_i x_{i+1}) \partial_i, \qquad i=1,\ldots,n-1.
\end{gather*}
It follows from Lemma~\ref{lem4.15} that the operators $\big\{\partial_{i}^{A}\big\}_{1 \le i
\le n}$ satisfy the Coxeter relations
\begin{gather*}
\partial_{i}^{A} \partial_{i+1}^{A} \partial_{i}^{A}=\partial_{i+1}^{A} \partial_{i}^{A} \partial_{i+1}^{A}, \qquad i=1,\ldots,n-1.
\end{gather*}

\begin{de}\label{def4.19}\quad
\begin{enumerate}\itemsep=0pt
\item[(1)] Let $w \in \mathbb{S}_n$ be a permutation. Def\/ine the generalized
Schubert polynomial corresponding to permutation $w$ as follows
\begin{gather*}
 \mathfrak{S}_w^{A}(X_n) = \partial_{w^{-1} w_{0}}^{A} x^{\delta_n}, \qquad
x^{\delta_n}:= x_1^{n-1} x_2^{n-2} \cdots x_{n-1},
\end{gather*}
and $w_0$ denotes the longest element in the symmetric group $\mathbb{S}_{n}$.

\item[(2)] Let $\alpha$ be a composition with at most $n$ parts, denote by
$w_{\alpha} \in \mathbb{S}_n$ the permutation such that $w_{\alpha}(\alpha)=
\alpha^{+}$. Let us recall that $\alpha^{+}$ denotes a unique partition
corresponding to composition~$\alpha$.
\end{enumerate}
\end{de}

\begin{lem}\label{lem4.20} Let $w \in \mathbb{S}_n$ be a permutation.{\samepage
\begin{itemize}\itemsep=0pt
\item If $A=(0,0,0,1,0)$, then $\mathfrak{S}_{w}^{A}(X_n)$ is equal to
the Schubert polynomial $\mathfrak{S}_{w}(X_n)$.

\item If $A=(-\beta,\beta,0,1,0)$, then $\mathfrak{S}_{w}^{A}(X_n)$ is
equal to the $\beta$-Grothendieck polynomial $\mathfrak{G}_w^{({\beta})}(X_n)$
 introduced in~{\rm \cite{FK1}}.

\item If $A=(0,\beta,0,1,0)$ then $\mathfrak{S}_{w}^{A}(X_n)$ is equal
to the dual $\beta$-Grothendieck polynomial ${\cal{H}}_{w}^{(\beta)}(X_n)$,
studied in depth for $\beta=-1$ and in the basis $\{x_i:=\exp(\xi_{l})\}$ in~{\rm \cite{L33}}.

\item If $A=(-1,2,0,1,1)$, then $\mathfrak{S}_{w}^{A}(X_n)$ is equal to
the Di Francesco--Zinn-Justin polynomials introduced in~{\rm \cite{DZ1}}.

\item If $A=(1,-1,1, h,0)$, then $\mathfrak{S}_{w}^{A}(X_n)$ is equal to
the $h$-Schubert polynomials.
\end{itemize}
In all cases listed above the polynomials $\mathfrak{S}_{w}^{A}(X_n)$ have
non-negative integer coefficients.}
\end{lem}

Def\/ine the generalized key or Demazure polynomial corresponding to a
composition~$\alpha$ as follows
\begin{gather*}
 K_{\alpha}^{A}(X_n)= \partial_{w_{\alpha}}^{A} x^{\alpha^{+}}.
\end{gather*}
\begin{itemize}\itemsep=0pt
\item If $A=(1,0,1,0,0)$, then $K_{\alpha}^{A}(X_n)$ is equal to key (or
Demazure) polynomial corresponding to~$\alpha$.

\item If $A=(0,0,1,0,0)$, then $K_{\alpha}^{A}(X_n)$ is equal to the
reduced key polynomial.

\item If $A=(1{,}0{,}1{,}0{,}\beta)$, then $K_{\alpha}^{A}(X_n)$ is equal to the
key $\beta$-Grothendieck polynomial $\operatorname{KG}_{\alpha}^{(\beta)}\!(X_n)$.

\item If $A=(0,0,1,0,\beta)$, then $K_{\alpha}^{A}(X_n)$ is equal to
the reduced key $\beta$-Grothendieck polynomials.
\end{itemize}
In all cases listed above the polynomials $\mathfrak{S}_{w}^{A}(X_n)$ have
 non-negative integer coef\/f\/icients.
\begin{itemize}\itemsep=0pt
\item If $A=(-1,q^{-1},-1,0,0)$ and $\lambda$ is a partition, then (up
to a scalar factor) polynomial $K_{\lambda}^{A}(X_n)$ can be identify with a
certain {\it Whittaker function} (of type $A$), see~\cite[Theorem~A]{BBL}.
Note that operator $T_{i}^{A}:= -1 +(q^{-1} x_i-x_{i+1}) \partial_i$, $1 \le i
\le n-1$, satisfy the Coxeter and Hecke relations, namely $(T_{i}^{A})^{2}=
(q^{-1}-1) T_{i}^{A}+q^{-1}$. In~\cite{BBL} the operator $T_{i}^{A}$ has been
denoted by~${\mathfrak{T}}_i$.

\item Let $w \in {\mathbb{S}}_n$ be a permutation and $\mathfrak{m}=(i_1,\ldots,i_{\ell})$ be a reduced word for $w$, i.e., $w=s_{i_{1}} \cdots s_{i_{\ell}}$ and $\ell(w)= \ell$. Denote by $Z_{\mathfrak{m}}$ the {\it Bott--Samelson} nonsingular variety corresponding to the reduced word~$\mathfrak{m}$.
 It is well-known that the Bott--Samelson variety $Z_{\mathfrak{m}}$ is
birationally isomorphic to the Schubert variety $X_{w}$ associated with
permutation $w$, i.e., the Bott--Samelson variety $Z_{\mathfrak{m}}$ is a~desingularization of the Schubert variety~$X_{w}$. Following~\cite{BBL} def\/ine
the Bott--Samelson polynomials $Z_{\mathfrak{m}}(x,\lambda,v)$ as follows
\begin{gather*}
Z_{\mathfrak{m}}(x,\lambda,v) = \big(1+T_{i_{1}}^{A}\big) \cdots \big(1+T_{i_{\ell}}^{A}\big) x^{\lambda},
\end{gather*}
where $A=(-v,-1,1,0)$. Note that $\big(1+T_{i}^{A}\big)^2= (1+v)\big(1+T_{i}^{A}\big)$, and the
 divided dif\/ference operators $1+T_{i}^{A}=1+v+(x_{i+1}-v x_{i}) \partial_i$
do {\it not} satisfy the Coxeter relations.

\item If $A=(- \beta, \beta + \alpha,0,1,\beta \alpha)$, then
$\s_{w}^{A}(X_n)$ constitutes a common generalization of the $\beta$-Grothendieck and the Di Francesco--Zinn-Justin polynomials.

\item If $A=(t,-1,t,1,0)$, then the divided dif\/ference operators
\begin{gather*}%\label{equation4.20}
T_i^{A}:= t+(-x_i+t x_{i+1} +1) \partial_i, \qquad 1 \le i \le n-1,
\end{gather*}
their \textit{baxterizations} and the \textit{raising operator}
\begin{gather*}%\label{equation4.21}
 \phi:= (x_n-1) \pi,
\end{gather*}
where $\pi$ denotes the {\it $q^{-1}$-shift operator}, namely $\pi(x_1,\ldots,x_n)=(x_n/q,x_1,\ldots,x_{n-1})$,
 can be used to \textit{generate} the
\textit{interpolation Macdonald polynomials} as well as the
\textit{nonsymmetric Macdonald polynomials}, see~\cite{LRW} for details.
\end{itemize}

 In similar fashion, relying on the operator $\phi$, operators
\begin{gather*}
T_i^{\beta,\alpha,\gamma,h}:=
 -\beta+((\alpha+\beta+\gamma) x_i+\gamma x_{i+1}+h\\
 \hphantom{T_i^{\beta,\alpha,\gamma,h}:=}{}
 +h^{-1} (\alpha+\gamma)(\beta+\gamma) x_i x_{i+1}) \partial_{i}, \qquad 1 \le i \le n-1,
\end{gather*}
and their \textit{baxterization},
one (A.K.) can introduce polynomials $M_{\delta}^{\beta,\alpha,\gamma,q}(X_n)$, where
$\delta$ is a \textit{composition}. These polynomials are \textit{common
generalizations} of the interpolation Macdonald polynomials
$M_{\delta}(X_n;q,t)$ (the case $\beta=-t$, $\alpha=-1$, $\gamma =t$), as well as
the Schubert, $\beta$-Grothendieck and its \textit{dual}, Demazure and
Di Francesco--Zinn-Justin polynomials, and \textit{conjecturally} their
{\it affine} analogues/versions. Details will appear elsewhere.
%\end{comm}

%\begin{comm}\label{comm4.21}
{\bf Double af\/f\/ine nilCoxeter algebra.}
Let $t$, $q$, $a$, $b$, $c$, $h$, $d$ be
parameters.
\begin{de}\label{def4.22} Def\/ine double af\/f\/ine nil-Coxeter algebra $\operatorname{DANC}_{n}$ to be (unital)
associative algebra over $\Q\big(q^{\pm 1},t^{\pm 1}\big)$ with the set of generators $ \big\{e_1,\ldots,e_{n-1}, x_{1},\ldots,x_n, \pi^{\pm 1} \big\}$ subject to relations
\begin{itemize}\itemsep=0pt
\item (nilCoxeter relations)
\begin{gather*}
e_i e_j =e_j e_i,\quad \text{if} \quad |i-j| \ge 2, \qquad e_i^2=0, \quad \forall\, i, \qquad e_i e_j e_i =e_j e_i e_j, \quad \text{if}\quad |i-j|=1;
\end{gather*}

\item (crossing relations)
\begin{gather*}%\label{equatiuon4.22}
 x_i e_k =e_k x_i, \quad \text{if} \quad k \not=i, i+1, \qquad x_i e_i - e_{i} x_{i+1} =1, \qquad
e_i x_i - x_{i+1} e_i =1;
\end{gather*}

\item (af\/f\/ine crossing relations)
\begin{gather*}
 \pi x_i =x_{i+1} \pi, \quad \text{if}\quad i < n,\qquad \pi x_{n} = q^{-1} x_1 \pi, \\
\pi e_{i} = e_{i+1} \pi, \quad \text{if}\quad i < n-1, \qquad \pi^2 e_{n-1} = q e_{1} \pi^2.
\end{gather*}
 \end{itemize}
\end{de}

Now let us introduce elements $e_{0}:= \pi e_{n-1} \pi^{-1}$ and
\begin{gather*}%\label{equation4.23}
T_{0}:= T_{0}^{a,b,c,h,d} =
\pi T_{n-1} \pi^{-1} = a+\big(b x_{n}+ q^{-1} c x_1 + h +
q^{-1} d x_1 x_n\big) e_{0}.
\end{gather*}
 It is easy to see that $\pi e_{0} =q e_{1} \pi$,
\begin{gather*}
\pi T_{0}^{a,b,c,h,d} = T_{1}^{a,b,c,q h,q^{-1} d}
e_{1} \pi = T_{1}^{a,b,c,h,d} +\big((1-q) h+\big(1-q^{-1}\big) d x_1 x_2\big) e_1.
\end{gather*}
Now let us assume that $a = t$, $b=-t$, $d=e=0$, $c=1$. Then,
\begin{gather*}
T_i= t+(x_{i+1} -t x_{i}) e_i , \quad i=1,\ldots,n-1 , \qquad T_0= t+\big(q^{-1} x_1 -t x_{n}\big) e_{0},
\\
 T_i^2 =(t-1) T +t , \quad 0 \le i < n, \qquad T_i x_i T_i =t x_{i+1}, \quad 1 \le i <
n, \qquad T_0 x_{n} T_{0}= t q^{-1} x_1,
\\
T_0 T_1 T_0=T_1 T_0 T_1 , \qquad T_{n-1} T_0 T_{n-1} = T_0 T_{n-1} T_0, \qquad T_0 T_i =T_i T_0, \qquad \text{if} \quad 2 \le i < n-1.
\end{gather*}

The operators $T_{i}:= T_{i}^{t,-t,1,0,0}$, $0 \le i \le n-1$ have been used
in~\cite{LRW} to give an ``elementary'' construction of nonsymmetric Macdonald polynomials. Indeed, one can realize the operator~$\pi$ as
follows:
\begin{gather}\label{comm4.21pi}
\pi(f)= f(x_{n} /q, x_1,x_2,\ldots,x_{n-1}), \qquad \text{so that} \quad
\pi^{-1}(f)=(x_2,\ldots,x_{n},q x_1),
\end{gather}
and introduce the raising operator~\cite{LRW} to be
\begin{gather*}
 \phi(f(X_n)) = (x_n-1) \pi (f(X_n)).
\end{gather*}
It is easily seen that $\phi T_i = T_{i+1} \phi$, $i=0,\dots,n-2$,
and $\phi^2 T_{n-1}= T_1 \phi^2$. It has been established in~\cite{LRW}
how to use the operators $\phi$, $T_1,\ldots,T_{n-1}$ to to give formulas for
the \textit{interpolation Macdonald polynomials}. Using operators $\phi,
T_{i}^{(a,b,c,h,d)}, i=1,\ldots,n-1$ instead of $\phi$, $T_1,\ldots,T_{n-1}$, $1 \le i \le n-1$, one get a $4$-parameter generalization of the interpolation
Macdonald polynomials, as well as the nonsymmetric Macdonald polynomials.

It follows from the nilCoxeter relations listed above, that the
 Dunkl--Cherednik elements, cf.~\cite{CH},
\begin{gather*}%\label{equation4.24}
Y_i:= \left( \prod_{a=i-1}^{1} T_{a}^{-1} \right) \pi \left( \prod_{a=n-1}^{i+1} T_{a} \right), \qquad i=1,\ldots,n,
\end{gather*}
where $T_i=T_{i}^{t,-t,1,0,0}$, generate a commutative subalgebra in the
double af\/f\/ine nilCoxeter algebra $\operatorname{DANC}_{n}$. Note that the algebra $DANC_n$
contains lot of other interesting commutative subalgebras, see, e.g.,~\cite{IK}.

It seems interesting to give an interpretation of polynomials generated by
the set of opera\-tors~$T_{i}^{t,-t,1,h,e}$, $i=0,\dots,n-1$ in a~way similar to
that given in~\cite{LRW}. We expect that these polynomials provide an~af\/f\/ine
version of polynomials ${\cal{KN}}_{w}^{(-t,-1,1,1,0)}(X)$, $w \in
{\mathbb{S}}_n\subset {\mathbb{S}}_n^{\text{af\/f}}$, see Remark~\ref{rem4.7}.

Note that for any af\/f\/ine permutation $ v \in {\mathbb{S}}_{n}^{\text{af\/f}}$, the
operator
\begin{gather*}
 T_{v}^{(a,b,c,h,d)}=T_{i_{1}}^{(a,b,c,h,d)} \cdots
T_{i_{\ell}}^{(a,b,c,h,d)},
\end{gather*}
 where $v=s_{i_{1}} \cdots s_{i_{\ell}}$ is any reduced decomposition
of $v$, is {\it well-defined} up to the sign~$\pm 1$. It seems an
interesting {\it problem} to investigate properties of polynomials
$L_{v}[\alpha](X_n)$, where $v \in {\mathbb{S}}_n^{\text{af\/f}}$ and $\alpha \in
\Z_{\ge 0}^{n}$, and f\/ind its algebro-geometric interpretations.

\section{Cauchy kernel}\label{section5}

Let $u_1,u_2,\dots,u_{n-1}$ be a set of generators of the free algebra
${\cal F}_{n-1}$, which are assumed also to be commute with the all variables
$ {\mathfrak{P}}_n:=\{p_{i,j}, \, 2 \le i+j \le n+1,\, i \ge 1, j \ge 1 \}$.
\begin{de}\label{def5.1} The Cauchy kernel ${\cal C}({\mathfrak{P}}_n,U)$ is def\/ined to be
as the ordered product
\begin{gather}\label{equation5.1}
{\cal C}({\mathfrak{P}}_n,U)= \prod_{i=1}^{n-1}
 \left \{{\prod_{j=n-1}^{i}(1+ p_{i,j-i+1} u_{j})} \right\}.
\end{gather}
\end{de}

For example,
\begin{gather*}
 {\cal C}({\mathfrak{P}}_4,U)=(1+p_{1,3} u_3)(1+p_{1,2} u_2)
(1+p_{1,1} u_1)(1+p_{2,2} u_3)(1+p_{2,1} u_2)(1+p_{3,1} u_3).
\end{gather*}

In the case $\{p_{ij}=x_i, \, \forall \, j \}$ we will write ${\cal{C}}_{n}(X,U)$
instead of ${\cal C}({\mathfrak{P}}_n,U)$.
\begin{lem}\label{lem5.2}
\begin{gather}\label{equation5.2}
 {\cal C}({\mathfrak{P}}_n,U) = \sum_{({\bf a},{\bf b}) \in {\cal S}_n}
\prod_{j=1}^{p} p_{\{{a_{j},b_{j} \}}} w({\bf a},{\bf b}),
\end{gather}
where ${\bf a} =(a_1,\ldots,a_p)$, ${\bf b}=(b_1,\ldots,b_p)$,
$w({\bf a},{\bf b})= \prod\limits_{j=1}^{p} u_{a_{j}+b_{j} -1}$, and the sum in
\eqref{equation5.2} runs over the set
\begin{gather*}
{\cal S}_n := \big\{({\bf a},{\bf b}) \in \N^p \times \N^p \,|\, {\bf a}=
(a_1 \le a_2 \le \cdots \le a_p),\\
\hphantom{{\cal S}_n := \{({\bf a},{\bf b}) \in \N^p \times \N^p \,|\,}{} a_i+b_i \le n, \ \text{and if} \ a_i=a_{i+1}
\Longrightarrow b_i > b_{i+1} \big\}.
\end{gather*}
\end{lem}

We denote by ${\cal S}_n^{(0)}$ the set $\{({\bf a},{\bf b}) \in {\cal S}_n
\,|\, w({\bf a},{\bf b}) \ \text{is a tableau word}\}$.

The number of terms in the r.h.s.\ of~\eqref{equation5.1} is equal to
$2^{{n \choose 2}}$, and therefore is equal to the number $\# |\operatorname{STY}(\delta_n,
\le n)|$ of semistandard Young tableaux of the staircase shape
$\delta_n:=(n-1$, $n-2,\ldots,2,1)$ f\/illed by the numbers from the set
$\{1,2,\ldots,n \}$. It is also easily seen that the all terms appearing in
the r.h.s.\ of~\eqref{equation5.2} are dif\/ferent, and thus $\# | {\cal S}_n |= \# |\operatorname{STY}(\delta_n,
\le n)|$.

We are interested in the decompositions of the Cauchy kernel
${\cal C}({\mathfrak{P}}_n,U)$ in the algebras ${\cal P}_{n}$,
${\cal NP}_{n}$, ${\cal IP}_{n}$, ${\cal NC}_{n}$ and ${\cal IC}_{n}$.

\subsection[Plactic algebra ${\cal P}_n$]{Plactic algebra $\boldsymbol{{\cal P}_n}$}\label{section5.1}

Let $\lambda$ be a partition and $\alpha$ be a composition of the same size.
Denote by ${\widetilde {\operatorname{STY}}}(\lambda,\alpha)$ the set of semistandard Young
tableaux~$T$ of the shape $\lambda$ and content $\alpha$ which must satisfy
the following conditions:
 for each $k=1,2,\dots$, the all numbers~$k$ are located in the
f\/irst~$k$ columns of the tableau~$T$. In other words, the all entries~$T(i,j)$ of a semistandard tableau $T \in {\widetilde {\operatorname{STY}}}(\lambda,\alpha)$
have to satisfy the following conditions: $T_{i,j} \ge j$.

For a given (semi-standard) Young tableau $T$ let us denote by~$R_i(T)$ the
set of numbers placed in the $i$-th row of~$T$, and denote by
${\widetilde {\operatorname{STY}}}_{0}(\lambda,\alpha)$ the subset of the set
${\widetilde {\operatorname{STY}}}(\lambda,\alpha)$ involving only tableaux $T$ which
satisfy the following constrains
\begin{gather*}
R_1(T) \supset R_2(T) \supset R_3(T) \supset \cdots.
\end{gather*}

To continue, let us denote by ${\cal A}_n$ (respectively by
${\cal A}^{(0)}_n$) the union of the sets
${\widetilde {\operatorname{STY}}}(\lambda, \alpha)$ (resp.\ that of
${\widetilde {\operatorname{STY}}}_{0}(\lambda, \alpha)$)
for all partitions $\lambda$ such that $\lambda_i \le n-i$ for $i=1,2,\dots,
n-1$, and all compositions $\alpha$, $l(\alpha) \le n-1$. Finally, denote by
${\cal A}_n(\lambda)$ (resp.~${\cal A}^{(0)}_n(\lambda)$) the subset of
${\cal A}_n$ (resp.~${\cal A}^{(0)}_n(\lambda)$) consisting of all tableaux
of the shape~$\lambda$.

\begin{lem}\label{lem5.3} \quad
\begin{itemize}\itemsep=0pt
\item $|{\cal A}_n(\delta_n)|=1$,
$|{\cal A}_n(\delta_{n-1})|= (n-1)!$, $|{\cal A}_n((n-1))|= C_{n-1}$ the
$(n-1)$-th Catalan number. More generally,
\begin{gather*}
\big|{\cal A}_n\big(\big(1^{k}\big)\big)\big|= {n -1 \choose k}, \qquad |{\cal A}_n((k))|= \frac{n-k}{n} {n+k-1 \choose k} = \dim {\mathbb{S}}_{n+k-1}^{(n - 1,k)},
\end{gather*}
 $k=1 \dots,n-1$, cf.~{\rm \cite[$A009766$]{SL}}; here $\dim {\mathbb{S}}_{n}^{\lambda}$ stands for the dimension of the irreducible representation of the
symmetric group ${\mathbb{S}}_{n}$ corresponding to a~partition~$\lambda \vdash n$.
\item Let $k \ge \ell \ge 2$, $n \ge k+2$, then
\begin{gather*}
|{\cal A}_n((k,\ell))| = \frac{(n-k)(n-\ell+1)(k-\ell+1)(n^2-n-\ell(k+1))}
{\ell ! (k+1) k (n+k)} {n+k \choose k-1} \prod_{i=1}^{\ell-2}(n+i),
\end{gather*}
$ k=1, \ldots,n$. The case $k= \ell$ has been studied in~{\rm \cite{Gou}} where one
can find a combinatorial interpretation of the numbers $|{\cal A}_n((k,k))|$
for all positive integers~$n$; see also~{\rm \cite[$A005701$, $A033276$]{SL}} for
more details concerning the cases $k=2$ and $k=3$. Note that in the case $k=\ell$ one has $n^2-n-k(k+1)= (n+k)(n-k-1)$, and the above formula can be rewritten as follows $( k \ge 2)$
\begin{gather*}
|{\cal A}_n((k,k))| = \frac{(n-k-1)(n-k)(n-k-1)}{(k+1) k^2 (k-1)}~{n+k-2 \choose k-2}~{n+k-1 \choose k-1}.
\end{gather*}
\item Boundary case: the number $|{\cal{A}}_{N}((n^k))|$ for $N=n+k$,\footnote{So far as we know, the third equality has been proved for the f\/irst
time in~\cite{CV}. The both sides of the third identity have a big variety of
combinatorial interpretations such as the number of $k$-tuples of noncrossing
Dyck paths; that of $k$-tringulations of a convex $(n+k+1)$-gon; that of semistandard Young tableaux with entries from the set $\{1,\ldots,n\}$ having only
columns of an even length and bounded by height~$2 k$~\cite{CV}; that of
pipe dreams ({\it or} compatible sequences) associated with the Richardson
permutation $1^{k} \times w_{0}^{(n)} \in {\mathbb{S}}_{n+k}$, etc.,
see, e.g.,~\cite[$A078920$]{SL} and the literature quoted therein. It seems an interesting task to read of\/f an alternating sign matrix of size~$(n+k) \times (n+k)$ from a given $k$-triangulation of a convex $n+k+1)$-gon.}
\begin{gather*}
\big|{\cal{A}}_{k+n}\big(\big(n^k\big)\big)\big| =
\big|{\cal{A}}_{k+n}\big(\big((n+1)^k\big)\big)| = \det|\operatorname{Cat}_{n+i+j-1}|_{1 \le i, j \le k}=
\!\!\prod_{1 \le i \le j \le n} \!\!\frac{i+j+2 k}{i+j}.
\end{gather*}

\item More generally $($A.K.$)$,
\begin{gather*}
\big|{\cal{A}}_{N}\big(\big(n^k\big)\big)| = \prod_{1 \le i \le k,\, 1 \le j \le n \atop j-i \le n-k }\frac{N-i-j+1}{i+j-1} {\prod_{1 \le i \le k,\, 1 \le j \le n \atop j-i > n-k} \frac{N+i+j-1}{i+j-1}}.
\end{gather*}
Moreover, the number $|{\cal{A}}_{N}((n^k))|$ also is equal to the number of
$k$-tuples of noncrossing Dyck paths staring from the point $(0,0)$ and ending at the point $(N, N-n-k)$.\footnote{So far as we know, for the case $k=2$ an equivalent formula for the number of pairs of noncrossing Dyck paths connecting the points~$(0,0)$ and $(N,N-n-2)$, has been obtained for the f\/irst time in~\cite{Gou}.}

\item There exists a bijection
$\rho_n \colon {\cal A}_n \longrightarrow \operatorname{ASM}(n)$ such that the image
$\operatorname{Im} \big({\cal A}^{(0)}_n\big)$ contains the set of $n \times n$ permutation matrices.

\item The number of row strict $($as well as column strict$)$
diagrams $\lambda \subset \delta_{n+1}$ is equal to~$2^n$.
\end{itemize}
\end{lem}

Recall that a row-strict diagram\footnote{Known also as a strict partition.}
$\lambda$ is on such that $\lambda_i - \lambda_{i+1} \ge 1$, $\forall\, i$;
$\delta_{n}:= (n-1$, $n-2,\ldots,2,1)$.

\begin{ex}\label{exam5.4} Take $n= 5$ so that $\operatorname{ASM}(5)=429$ and $\operatorname{Cat}(5)=42$. One has
\begin{gather*}
\big|{\cal A}^{(0)}_5\big|= \sum_{\lambda \subset \delta_4:=(4,3,2,1) \atop \lambda_i-\lambda_{i+1} \ge 1\, \forall\, i} \big|{\cal{A}}_{5}^{(0)}(\lambda)\big| =
\big|{\cal A}^{(0)}_5(\varnothing)\big|+\big|{\cal A}^{(0)}_5((1))\big|+
\big|{\cal A}^{(0)}_5((2))\big|+\big|{\cal A}^{(0)}_5((3))\big|\\
\hphantom{\big|{\cal A}^{(0)}_5\big|=}{}
+\big|{\cal A}^{(0)}_5((2,1))\big|+
\big|{\cal A}^{(0)}_5((4))\big|+\big|{\cal A}^{(0)}_5((3,1))\big|+\big|{\cal A}^{(0)}_5((3,2))\big|+
\big|{\cal A}^{(0)}_5((4,1))\big|\\
\hphantom{\big|{\cal A}^{(0)}_5\big|=}{}
+\big|{\cal A}^{(0)}_5((4,2))\big|+\big|{\cal A}^{(0)}_5((3,2,1))\big|+\big|{\cal A}^{(0)}_5((4,3))\big|+\big|{\cal A}^{(0)}_5((4,2,1))\big|\\
\hphantom{\big|{\cal A}^{(0)}_5\big|=}{}
+
\big|{\cal A}^{(0)}_5((4,3,1))\big|+\big|{\cal A}^{(0)}_5((4,3,2))\big|+
\big|{\cal A}^{(0)}_5((4,3,2,1))\big| \\
\hphantom{\big|{\cal A}^{(0)}_5\big|}{}
=
1+4+9+14+6+14+16+4+21+14+4+1+9+2+1+1 =121 \\
\hphantom{\big|{\cal A}^{(0)}_5\big|=}{}\text{(sum of $16$ terms)},\\
\sum_{k=0}^{4}|{\cal A}_{5}((k))| =1 +4+ 9+ 14 +14 = 42.
\end{gather*}
\end{ex}

We expect that the image $\rho_n \big( \bigcup_{k=0}^{n-1} {\cal A}_{n}((k))\big)$ coincides with the set of $n \times n$ permutation matrices corresponding to either $321$-avoiding (or 132-avoiding) permutations.

Now we are going to def\/ine a statistic $n(T)$ on the set ${\cal A}_n$.

\begin{de}\label{def5.5} Let $\lambda$ be a partition, $\alpha$ be a composition of the
 same size. For each tableau
$T \in {\widetilde {\operatorname{STY}}}(\lambda,\alpha) \subset {\cal A}_n(\lambda)$ def\/ine
\begin{gather*}
n(T)=\alpha_n= \# | \{(i,j) \in \lambda | T(i,j)=n \}|.
\end{gather*}
Clearly, $n(T) \le \lambda_1$.

Def\/ine polynomials
\begin{gather*}
{\cal A}_{\lambda}(t):=
\sum_{T \in {\cal A}_n(\lambda)}t^{\lambda_1-n(T)}.
\end{gather*}
\end{de}

It is instructive to display the numbers
$\{{\cal A}_n(\lambda), \lambda \subset \delta_n \}$ as a vector of the
length equals to the $n$-th Catalan number. For example,
\begin{gather*}
\begin{split}
& {\cal A}_4(\varnothing{,}(1),(2),(1,1),(3),(2,1),(1,1,1),(3,1),(2,2),(2,1,1),
(3,2),(3,1,1),(2,2,1),(3,2,1))\!\\
& \qquad{} =
(1,3,5,3,5,6,1,6,3,2,3,2,1,1).
\end{split}
\end{gather*}

It is easy to see that the above data, as well as the corresponding data for
 $n=5$, coincide with
the list of ref\/ined totally symmetric self-complementary plane partitions that
f\/it in the box $2n \times 2n \times 2n$ ($\operatorname{TSSCPP}(n)$ for short)
listed for $n=1,2,3,4,5$ in \cite[Appendix~D]{DZ1}.

In fact we have
\begin{Theorem}\label{theorem5.6} The sequence $\{{\cal A}_n(\lambda),
\lambda \subset \delta_n \}$ coincides with the set of refined
$\operatorname{TSSCPP}(n)$ numbers
as defined in~{\rm \cite{DZ1}}. More precisely,
\begin{itemize}\itemsep=0pt
\item $|{\cal A}_n(\lambda;N)|=
\det \big|{N-i \choose \lambda'_{j}-j+i} \big|_{1 \le i,j \le
\ell(\lambda)}$,

\item we have
\begin{gather*}
{\cal A}_{\lambda}(N;t):=
\det \left|{N-i-1 \choose \lambda'_{j}-j+i-1}+t {N-i-1 \choose
\lambda'_{j}-j+i} \right|_{1 \le i,j \le \ell(\lambda)},
\end{gather*}

\item let $\lambda$ be a partition, $|\lambda|=n$, consider the column
multi-Schur polynomial and $t$-deformation thereof
\begin{gather*} s_{\lambda}^{*}(X_N):= \det |
e_{\lambda'_{j} -j+i}(X_{N-i}) |_{1 \le i,j \le \ell(\lambda)}, \qquad \text{cf.
{\rm \cite[Chapter~III]{Ma}, \cite{Wa}}}, \quad \text{and}
\\
s_{\lambda}^{*}(X_N;t):= \det |x_{N-i} e_{\lambda'_{j}-j+i-1}(X_{N-i-1})+t e_{\lambda'_{j}-j+i}(X_{N-i-1}) |_{1 \le i,j \le \ell(\lambda)},
\end{gather*}
then, assuming that $\lambda \subset \delta_n$ and $N \ge \lambda_1+\lambda'_{\ell(\lambda)}$, the
polynomial $s_{\lambda}^{*}(X_N)$ has nonnegative $($integer$)$
coefficients,

\item polynomial ${\cal A}_{\lambda}(t)$ is equal to a $t$-analog of
refined $\operatorname{TSSCPP}(n)$ numbers
$P_n(\lambda'_{n-1}+1,\dots$, $\lambda'_{n-i}+i,\dots,\lambda'_1+n-1 | t)$
introduced by means of recurrence relations in {\rm
\cite[relation~(3.5)]{DZ1}},

\item one has
\begin{gather*}
 s_{((n^{k}))}^{*}(X_{n+k}) = M_{n,k}(X_{n+k}) \s_{1^{k} \times
w_{0}^{(n)}}
\big(x_{1}^{-1},\ldots,x_{k+n}^{-1}\big),
\end{gather*}
where $\s_{{1^{k} \times w_{0}^{(n)}}}(X_{n+k})$ denotes the Schubert
polynomial corresponding to the permutation
\begin{gather*}
 1^k \times w_{0}^{(n)}= [1,2,\ldots,k, n+k,n+k-1,\ldots,k+1],
\end{gather*}
and $ M_{n,k}(X_{n+k}) =
\prod\limits_{a=1}^{k}x_{a}^{-1} \prod\limits_{a=1}^{n+k}~x_{a}^{\min(n+k-a+1,n)}$.
\end{itemize}
\end{Theorem}

In particular, $ \sum\limits_{\lambda \subset \delta_n} {\cal A}_{\lambda}(t) =
\sum\limits_{1 \le j \le n-1} A_{n,j} t^{j-1}$, where $A_{n,j}$ stands for the
number of alternating sign matrices ($\operatorname{ASM}_n$ for short) of size $n \times n$
with a~1 on top of the $j$-th column.
\begin{cor}[\cite{L1, LS1}]\label{cor5.7} The number of different tableau subwords in the word
\begin{gather*}
w_0:= \prod _{j=1}^{n-1} \left\{\prod_{a=n-1}^{j} a \right\}
\end{gather*}
is equal to the number of alternating sign matrices of size $n \times n$, i.e.,
\begin{gather*}
|{\cal A}_n| = |\operatorname{TSSCPP}(n)|=|\operatorname{ASM}_n|.
\end{gather*}
\end{cor}

It is well-known~\cite{B} that
\begin{gather*}
A_{n,j}={n+j-2 \choose j-1} {(2n-j-1)! \over (n-j)!} \prod_{i=0}^{n-2}
{(3i+1)! \over (n+i)!},
\end{gather*}
and the total number $A_n$ of $\operatorname{ASM}$ of size $n \times n$ is equal to
\begin{gather*}
A_n \equiv A_{n+1,1}=\sum_{j=1}^{n} A_{n,j}=
\prod_{i=0}^{n-1}{(3i+1)! \over (n+i)!}.
\end{gather*}

\begin{Theorem}[the case $\lambda =\delta_n:=(n-1,n-2,\ldots,2,1)$]\label{theorem 5.A}\quad
\begin{itemize}\itemsep=0pt
\item One has
\begin{gather*}
{\cal A}_{\delta_{n}}(n+1;t) = \prod_{j=2}^{n}~(1+ j t).
\end{gather*}
\item Gandhi--Dumont polynomials $($see {\rm \cite{Du}} and {\rm \cite[$A036970$]{SL})},
\begin{gather*}
{\cal A}_{\delta_{n}}(n+2;t) = \sum_{k=2}^{n} B_{n,k} t^k,
\qquad
B_{n,k}= \sum_{\{k_{j}\}} \prod_{j=2}^{n-1} {2 j- k_{j-1} \choose 2 j-
k_{j}},
\end{gather*}
where the sum runs over set of sequences $ \{1 \le k_{1} < k_2 < \cdots <
k_{n-2} < k_{n-1}= 2 n -k \}$.

\item In particular,
\begin{gather*}
{\cal A}_{\delta_{n}}(n+2;0) = G_{2n}, \qquad {\cal A}_{\delta_{n}}(n+2;1)=
G_{2n+2},
\end{gather*}
where $G_{2n} = 2 (2^{2n} -1) B_{2n}$ and $B_{2n}$ denotes the unsigned
Genocchi numbers\footnote{Recall that the unsigned Genocchi numbers are defined through
the
generting function
\begin{gather*}
\frac{2 t}{e^t +1} =\sum_{n \ge 1} \frac{G_{2n}}{(2n) !}(- 1)^n t^{2n},
\end{gather*}
see, e.g., \cite{Du}, or \url{https://en.wikipedia.org/wiki/Genocchi_number}.},
and $B_{2n}$ denotes the Bernoulli number, see, e.g.,
{\rm \cite[$A027642$]{SL}}.

\item ${\cal A}_{\delta_{n}}(n+2; -1) = (-1)^n$.

\item Let ${\cal A}_{\delta_{n}}(N;t,q)$ denote the principal
specialization $x_i:=q^{i-1}$, $i \ge 1$, of the polynomial
$s_{\delta_{n}}^{*}
(X_{N};t)$, and write ${\cal A}_{\delta_{n}}(n+2;t,q)= \sum\limits_{k=2}^{n-1}
B_{n,k}(q) t^k$. Then $B_{n,k}(q) \in \N[q]$.
\end{itemize}
\end{Theorem}

For example, ${\cal A}_{\delta_{6}}(8;t)=
(2073,8146,12840,10248,4200,720)_{t}$, $2073 = G_{12}$, ${\cal
A}_{\delta_{6}}(8;1)= 38227 =G_{14}$, cf.~\cite[$A036970$]{SL};
${\cal A}_{(2,1)}(5;t,q)= q {4 \brack 2}_{q} +t (q+q^2)^3 + t^2 q^5
{3 \brack 1}_{q}$. The last example shows that the polynomials ${\cal
A}_{(\delta_n)}(n+2;t,q)$ give rise to a~$q$-de\-formation of the
polynomials $p_{n+2}(t;\delta_n)$ associated with ref\/ined {\rm
${\rm TSSCPP}(n)$} introduced in~\cite{DZ1} which appeared to be coincide $($A.K.$)$
with the Gandhi polynomials introduced, e.g., in~\cite{Du}. However, the
polynomials ${\cal A}_{(\delta_n)}(n+2;t,q)$ are
dif\/ferent from a $q$-deformation of Gandhi polynomials defined in~\cite{HZ}.

It is easy to see that ${\cal{A}}_{\delta_{n}}(X_{n+2};t)=
{\cal{A}}_{\delta_{n}}^{(1)}(X_{n};t)+ x_{n+1}
{\cal{A}}_{\delta_{n}}^{(2)}(X_{n};t)$ for some polynomials depending on
the variables $\{x_1,\ldots,x_n\}$ with nonnegative integer coef\/f\/icients.

\begin{Theorem}[the Genocchi numbers of the second kind, \protect{\cite[$A005439$]{SL}}]\label{theorem 5.B} \quad
\begin{itemize}\itemsep=0pt

\item ${\cal{A}}_{\delta_n}^{(1)}(t;\, x_i=1,\, \forall \,i \in [1,n])$ is a
polynomial of the following form
\begin{gather*}
{\cal{A}}_{\delta_n}^{(1)}(t;\, x_i=1,\, \forall \, i \in [1,n])= t\/
G_{n-2}^{(2)}+\cdots+ (n-1) ! t^{n-2}, \\
 {\cal{A}}_{\delta_n}^{(1)}(t=1;\, x_i=1, \, \forall\, i \in [1,n])=
G_{n-2}^{(2)},
\end{gather*}
where $G_{n}^{(2)}$ stands for the $n$-th Genocchi number of the second
kind.
It is well known that $G_{n}^{(2)}= 2^{n-1} G_{n}^{(m)}$, where
$G_n^{(m)}$
denotes the so-called $n$-th median Genocchi number~{\rm \cite[$A000366$]{SL}}.

\item ${\cal{A}}_{\delta_{n}}^{(2)}(t=0;\,x_i=1,\, \forall\, i \in [1,n])=
G_{n-1}$,
where as before, $G_{n}$ denotes the $n$-th Genocchi number $($of the
first kind$)$ {\rm \cite[$A036970$]{SL}}.
\end{itemize}
\end{Theorem}

\begin{ex}\label{example5.10}\quad
\begin{enumerate}\itemsep=0pt

\item[(1)] Take $\lambda= \delta_3$ and $n=5$, then
\begin{gather*}
{\cal{A}}_{\delta_{3}}(X_4;t)=t(x_1+x_2)\big(x_3^2+t (x_1 x_2 +x_1 x_3+x_2 x_3)\big)\\
\hphantom{{\cal{A}}_{\delta_{3}}(X_4;t)=}{}
+x_4(x_1 x_2+x_1 x_3+x_3 x_3 +t(x_1+x_1)(x_1+x_2+x_3)).
\end{gather*}
Therefore, ${\cal{A}}_{\delta_{3}}(x_i=1, \,\forall \,i \in [1,3];t) = 2 t+6
t^2 +x_4( 3+6 t)$.

\item[(2)] Take $\lambda= \delta_4$ and $n=6$, then
\begin{gather*}
{\cal{A}}_{\delta_{4}}(x_i=1,\, \forall\, i \in [1,4];t)= 8 t\big(1+3 t+3
t^2\big)+ x_5 \big( 17+ 46 t +36 t^2\big).
\end{gather*}

\item[(3)] Take $\lambda= \delta_5$ and $n=7$, then
\begin{gather*}
{\cal{A}}_{\delta_{5}}(x_i=1,\, \forall \, i \in [1,5];t)= t(56,192,240,120)_{t}+
5 x_6 (31,100,114,48)_{t}.
\end{gather*}
\end{enumerate}
\end{ex}

Therefore, the polynomials ${\cal{A}}_{\delta_{n}}(X_{n+1};t)$ and
${\cal{A}}_{\delta_{n}}^{(1)}(X_{n};t)$ def\/ine multi-parameter
deformations of the Genocchi numbers of the f\/irst and the second types
correspondingly. It is an
 interesting task to relate these polynomials with those have been
studied in~\cite{HZ}, if so.

\begin{prb} Give combinatorial interpretations of polynomials ${\cal
A}_{p \delta_{n}}\!(N{;}t{,}q)$ and ${\cal A}_{(n^k)}\!(N{;}t{,}q)$ for all $N \ge p n -1$.
\end{prb}

Let as before $\operatorname{STY}(\delta_n \le n):={\cal ST}_n$ denotes
the set of all semistandard Young tableaux of the staircase shape
$\delta_n=(n-1,n-2,\ldots,2,1)$ f\/illed by the numbers from the set
$\{1,\ldots,n \}$. Denote by ${\cal ST}_n^{(0)}$ the subset of ``anti-diagonally'' increasing tableaux, i.e.,
\begin{gather*}
 {\cal ST}_n^{(0)} = \{T \in \operatorname{STY}(\delta_n, \le n) \,|\, T_{i,j}
\ge T_{i-1,j+1} \, \text{for all}\, 2 \le i \le n-1, \,1 \le j \le n-2 \}.
\end{gather*}
One (A.K.) can construct bijections
\begin{gather*}
\iota_n \colon \ {\cal S}_n \sim {\cal ST}_n, \qquad \zeta_n \colon \ {\cal A}_n \sim {\cal ST}_n^{(0)}
\end{gather*}
such that $\operatorname{Im}(\iota_n)=\operatorname{Im}(\zeta_n)$.

\begin{pr}\label{prop5.9}
\begin{gather*}
\sum_{\lambda=(\lambda_1,\ldots,\lambda_n) \atop \rho_n \ge \lambda}
K_{\rho_{n}, \lambda}
\begin{pmatrix}
& n \\
m_0(\lambda), & m_1(\lambda), & \ldots, & m_{n}(\lambda)
\end{pmatrix} = 2^{{n \choose 2}},
\\
\sum_{\lambda=(\lambda_1,\ldots,\lambda_n) \atop \rho_n \ge \lambda}
\begin{pmatrix}
& n \\
m_0(\lambda), & m_1(\lambda), & \ldots, & m_n(\lambda)
\end{pmatrix} = {\cal{F}}_n,
\end{gather*}
where ${\cal {F}}_n$ denotes the number of forests of trees on $n$ labeled
nodes;
$K_{\rho_{n},\lambda}$ denotes the Kostka number, i.e., the number of
semistandard Young tableaux of the shape $\rho_{n}:= (n-1,n-2,$ $\ldots,1)$ and
content/weight~$\lambda$;
for any partition $\lambda =(\lambda_1 \ge \lambda_2 \ge \cdots \ge \lambda_n \ge 0)$ we set \mbox{$m_i(\lambda) = \{j \,|\, \lambda_j =i\}$}.
\end{pr}

Let $\alpha$ be a composition, we denote by~$\alpha^{+}$ the partition
obtained from $\alpha$ by reordering of its parts. For example, if
$\alpha= (0,2,0,3,1,0)$ then $\alpha^{+}= (3,2,1)$. Note that $\ell(\alpha)=6$,
 but~$\ell(\alpha^{+})=3$.

Now let $\alpha$ be a composition such that $\rho_{n} \ge \alpha^{+}$,
 $\ell(\alpha) \le n$, that is $\alpha_j = 0$, if $j > \ell(\alpha)$,
$|\alpha| = {n \choose 2}$ and
\begin{gather*}
\sum_{k \le j} {(\rho_n)}_k \ge \sum_{k \le j} (\alpha^{+})_k, \qquad \forall \, j.
\end{gather*}
 There is a unique semistandard Young tableau $T_n(\alpha)$ of shape $\rho_n$
and content~$\alpha$ which corresponds to the maximal conf\/iguration of type
$(\rho_n;\alpha)$ and has all quantum numbers (riggings) equal to zero. It
follows from Proposition~\ref{prop5.9} that $\# \{\alpha \,|\, \ell(\alpha) \le n, \, \rho_n
\ge \alpha^{+} \} = {\cal {F}}_n$. Therefore there is a~natural embedding of the set of forests on $n$ labeled nodes to the set of
semistandard Young tableaux of shape $\rho_n$ f\/illed by the numbers from the
set $[1,\ldots,n]$. We denote by ${\cal {FT}}_n \subset
\operatorname{STY}( \rho_n, \le n)$ the subset $ \{T_n(\alpha) \,|\, \rho_n \ge \alpha^{+}, \, \ell(\alpha) \le n \}$. Note that the set ${\cal{K}}_n:= \{\alpha \,|\, \ell(\alpha)=
 n,\, (\alpha)^{+} = \rho_n \}$ contains $n !$ compositions, and under the
rigged conf\/iguration bijection the elements of the set ${\cal{K}}_n$
correspond to the {\it key} tableaux~\cite{LS3} of shape~$\rho_n$. See also~\cite{Av} for connections of the Lascoux--Sch\"{u}tzenberger \textit{keys} and~$\operatorname{ASM}$.

Let us say a few words about the Kostka numbers $K_{\rho_{n}, \alpha}$. First
of all, it's clear that if $\alpha =(\alpha_1,\alpha_2,\dots)$ is a~composition such that $\alpha_1=n-1$, then $K_{\rho_n,\alpha}=
K_{\rho_{n-1},\alpha[1]}$, where we set $\alpha[1]:=(\alpha_2,\ldots)$.

Now assume that $n=2k+1$ is an odd integer, and consider partitions
$\nu_n:=(k^n)$ and $\mu_n:= ((k+1)^{k},k^{k})$. Then
\begin{gather*}
 K_{\rho_n,\nu_n} = \operatorname{Coef\/f}_{(x_{1} x_{2} \cdots x_{n})^k} \left( \prod_{1 \le i < j \le n} (x_i+x_j) \right),
 \qquad {2k \choose k} K_{\rho_n,\mu_n} =K_{\rho_n,\nu_n}.
 \end{gather*}
It is well-known that the number $K_{\rho_n,\nu_n}$ is equal to number of
labeled regular tournaments with $n:=2k+1$ nodes, see, e.g.,~\cite[$A007079$]{SL}.

In the case when $n=2k$ is an even number, one can show that
\begin{gather*}
 K_{\rho_n,\nu_n} = K_{\rho_{n-1},{\nu_{n-1}}}, \qquad K_{\rho_n,\mu_n} =
K_{\rho_{n+1},{\mu_{n+1}}}.
\end{gather*}
Note that the rigged conf\/iguration bijection gives rise to an embedding of
the set of labeled regular tournaments with $n:=2k+1$ nodes
\begin{gather*}
\text{to the set $\operatorname{STY}(\rho_n, \le n)$, if $n$ is an odd integer, and} \\
 \text{to the set $\operatorname{STY}(\rho_{n-1}, \le n-1)$, if $n$ is even integer}.
\end{gather*}

\begin{Theorem}\label{theorem5.10} \quad
\begin{enumerate}\itemsep=0pt
\item[$(1)$] In the plactic algebra ${\cal P}_{n}$ the Cauchy kernel
has the following decomposition
\begin{gather*}%\label{equation5.3}
{\cal C}_n(\mathfrak{P},U) = \sum_{T \in {\cal A}_n}
{\cal K}_{T}(\mathfrak{P}) u_{w(T)}.
\end{gather*}

\item[$(2)$] Let $T \in {\cal A}_n$, and $\alpha(T)$ be its bottom code. Then
\begin{gather*}
 {\cal K}_{T}(\mathfrak{P}) -\prod_{(i,j) \in T} p_{\{i,T(i,j)-j+1\}} \ge 0,
\end{gather*}
and equality holds if and only if the bottom code~$\alpha(T)$ is a partition.
\end{enumerate}
\end{Theorem}

Note that the number of {\it different shapes} among the tableaux in the set
${\cal A}_n$ is equal to the Catalan number
$C_n:={1 \over {n+1}} {2n \choose n}$.

\begin{prb}\label{prob5.11} Construct a bijection between the set ${\cal A}_n$ and the set
of alternating sign matrices $\operatorname{ASM}_n$.
\end{prb}

\begin{ex}\label{exam5.12} For $n=4$ one has
\begin{gather*}
{\cal C}_{4}(X,U)=
K[0]+
K[1]u_1+
K[01]u_2+
K[001]u_3+
K[11](u_{12}+u_{22})+
K[2](u_{21}+u_{31})\\
\hphantom{{\cal C}_{4}(X,U)= }{}
+
K[101]u_{13}+
K[02]u_{32}+
K[011](u_{23}+u_{33})+
K[3]u_{321}+
K[12](u_{312}+u_{322})\\
\hphantom{{\cal C}_{4}(X,U)= }{}
+
K[21]u_{212}+
K[111](u_{123}+u_{133}+u_{233}+u_{223}+u_{333})+
K[021]u_{323}\\
\hphantom{{\cal C}_{4}(X,U)= }{}
+
K[201](u_{313}+u_{213})+
K[31]u_{3212}+
K[301]u_{3213}\\
\hphantom{{\cal C}_{4}(X,U)= }{}
+
K[22](u_{3132}+u_{2132}+u_{3232})+
K[121](u_{3123}+u_{3233}+u_{3223})\\
\hphantom{{\cal C}_{4}(X,U)= }{}
+
K[211](u_{2123}+u_{2133}+u_{3133})+
K[32]u_{32132}+
K[311](u_{32123}+u_{32133})\\
\hphantom{{\cal C}_{4}(X,U)= }{}
+
K[221](u_{21323}+u_{31323}+u_{32323})+
K[321]u_{321323}.
\end{gather*}
\end{ex}
Let $w \in \mathbb{S}_n$ be a permutation with the Lehmer code $\alpha(w)$.
\begin{de}\label{def5.13} Def\/ine the {\it plactic polynomial}
${\cal {PL}}_{w}(U)$ to be
\begin{gather*}%\label{equation5.4}
 {\cal {PL}}_{w}(U) = \biggl \{\sum_{T \in {\cal A}_n,\, \alpha(T)=
\alpha(w)} u_{w(T)} \biggr \}.
\end{gather*}
\end{de}

\begin{comm}\label{com5.14} It is easily seen from a def\/inition of the Cauchy kernel
that
\begin{gather*}%\label{equation5.5}
 {\cal{C}}_{n}(X,U)= \sum _{\alpha \subset \delta_{n}} K[\alpha](X)
{\cal {PL}}_{w_{0} w_{\alpha}^{-1}} (U),
\end{gather*}
where $w_{\alpha}$ denotes a unique permutation in $\mathbb{S}_{n}$ with the
Lehmer code equals $\alpha$; $K[\alpha](X)$ denotes the key polynomial corresponding to composition $\alpha \subset \delta_n$. The polynomials ${\cal {PL}}_{w_{0} w_{\alpha}^{-1}}$ can be treated as a~plactic version of noncommutative Schur and
Schubert polynomials introduced and studied in \cite{FG,Kn,AL,LS3,Le}.

Now let $X=\{x_1,\ldots,x_n \}$ be a set of mutually commuting variables, and
\begin{gather*}
 I_{0}^{(n)}:= \big\{\underbrace{n-1,n-2,\ldots,2,1}_{n-1},
\ldots, \underbrace{n-1,,n-2,\ldots,k+1,k}_{n-k},\ldots, \underbrace{n-1,n-2}_{2},n-1 \big\}
\end{gather*} be
lexicographically maximal reduced expression for the longest element $w_{0}
\in \mathbb{S}_n$. Let $I$ be a~tableau subword of the set $I_{0}:=I_{0}^{(n)}$. One can show (A.K.) that under the specialization
\begin{gather*}
 u_i= \begin{cases}x_i,&\text{if} \ \ i \in I_{0} {\setminus} I, \\
 1,& \text{if} \ \ i \in I
\end{cases}
\end{gather*}
the polynomial ${\cal {PL}}_{w_{0} w_{\alpha}^{-1}} (U)$ turns into the
Schubert polynomial $\s_{w_{\alpha}}(X)$. In a similar fashion, consider
the decomposition
\begin{gather*}%\label{equation5.6}
{\cal{C}}_{n}(X,U)= \sum _{\alpha \subset \delta_{n}} \operatorname{KG}[\alpha](X;-\beta)
{\cal {PL}}_{w_{0} w_{\alpha}^{-1}} (U; \beta).
\end{gather*}
One can show (A.K.) that under the same specialization as has been listed above, the polynomial ${\cal {PL}}_{w_{0} w_{\alpha}^{-1}} (U; \beta)$ turns into
the $\beta$-Grothendieck polynomial ${\cal{G}}_{w_{\alpha}}^{\beta}(X)$.
\end{comm}

\begin{de}\label{def5.15} Def\/ine algebra ${\cal{PC}}_n$ to be the quotient of the plactic algebra ${\cal{P}}_n$ by the two-sided ideal $J_n$ by the set of
monomials
\begin{gather*}%\label{equation5.7}
\{u_{i_{1}} u_{i_{2}} \cdots u_{i_{n}} \}, \qquad 1 \le i_{1} \le i_{2} \le
\cdots \le i_n \le n, \qquad \#\{a \vert i_a=j\} \le j, \qquad \forall\, j=1,\ldots,n.
\end{gather*}
\end{de}
\begin{Theorem}\label{theorem5.16} {\samepage \quad
\begin{itemize}\itemsep=0pt
\item The algebra ${\cal{PC}}_n$ has dimension equals to $\operatorname{ASM}(n)$,

\item $\operatorname{Hilb}({\cal{PC}}_n,q)= \sum\limits_{\lambda \in \delta_{n-1}} |{\cal{A}}_{n}(\lambda)| q^{|\lambda|}$,

\item $\operatorname{Hilb}(({\cal{PC}}_{n+1})^{ab},q)= \sum\limits_{k=0}^{n} \frac{n-k+1}{n+1} {n+k \choose n} q^k$, cf.~{\rm \cite[$A009766$]{SL}}.
\end{itemize}}
\end{Theorem}

\begin{de}\label{def5.17} Denote by ${\cal{PC}}_n^{\sharp}$ the quotient of the algebra
${\cal{PC}}_n$ by the two-sided ideal ge\-ne\-rated
by the elements $\{u_i u_j -u_j u_i,\, |i-j| \ge 2 \}$.
\end{de}

\begin{pr}\label{prop5.18}
Dimension $\dim {\cal{PC}}_n^{\sharp}$ of the algebra ${\cal{PC}}_n^{\sharp}$
 is equal to the number of Dyck paths whose ascent lengths are exactly
$\{1,2,\ldots,n+1 \}$.
\end{pr}
See \cite[$A107876$, $A107877$]{SL} where the f\/irst few of these numbers are
displayed.
\begin{ex}\label{exam5.19}
\begin{gather*}
\operatorname{Hilb}\big({\cal{PC}}_5^{\sharp},t\big) =(1,4,12,27,48,56,54,38,20,7,1)_{t} ,\qquad
\dim {\cal{PC}}_5^{\sharp} = 268,\\
\operatorname{Hilb}\big({\cal{PC}}_6^{\sharp},t\big) =
(1,5,18,50,116,221,321,398,414,368,275,175,89,35,9,1)_{t},\\
\dim {\cal{PC}}_6^{\sharp} = 2496 , \qquad
\dim {\cal{PC}}_7^{\sharp} = 28612.
\end{gather*}
\end{ex}

\begin{ex}\label{exam5.20}
\begin{gather*}
\operatorname{Hilb}({\cal{PC}}_3,q)=(1,2,3,1)_{q} , \qquad
\operatorname{Hilb}({\cal{PC}}_4,q)=(1,3,8,12,11,6,1)_{q}, \\
\operatorname{Hilb}({\cal{PC}}_5,q)=(1,4,15,35,69,91,98,70,35,10,1)_{q}, \\
\operatorname{Hilb}({\cal{PC}}_6,q)=(1,5,24,74,204,435,783,1144,1379,1346,1037,628,275,85,15,1)_{q}, \\
\operatorname{Hilb}({\cal{PC}}_7,q)= (1, 6, 35, 133, 461, 1281, 3196, 6686,
 12472, 19804, 27811, 33271, 34685,\\
 \hphantom{\operatorname{Hilb}({\cal{PC}}_7,q)=(}{} 30527, 22864, 14124, 7126,2828,840,
175, 21, 1)_{q}.
\end{gather*}
\end{ex}

\begin{prb}\label{prob5.21} Denote by ${\mathfrak{A}}_n$ the algebra generated by the
curvature of $2$-forms of the tautolo\-gi\-cal Hermitian linear bundles
 $\xi_{i}$, $1 \le i \le n$, over the flag variety~${\cal{F}}l_{n}$~{\rm \cite{SS}}.
It is well-known~{\rm \cite{PSS}} that the Hilbert polynomial of the algebra~${\mathfrak{A}}_n$ is equal to
\begin{gather*}
 \operatorname{Hilb}({\mathfrak{A}}_{n},t) = \sum_{F \in {\cal{F}}_{n}} t^{\operatorname{inv}(F)} =
\sum_{F \in {\cal{F}}_{n}} t^{\operatorname{maj}(F)},
\end{gather*}
where the sum runs over the set ${\cal{F}}_{n}$ of forests $F$ on the $n$
labeled vertices, and $\operatorname{inv}(F)$ $($resp.\ $\operatorname{maj}(F))$ denotes the inversion index $($resp.\ the major index$)$ of a forest~$F$.\footnote{For the readers convenience we recall def\/initions of statistics
$\operatorname{inv}(F)$ and $\operatorname{maj}(F)$. Given a forest $F$ on $n$ labeled vertices, one can
construct a tree~$T$ by adding a new vertex (root) connected with the maximal
vertices in the connected components of~$F$.

The inversion index $\operatorname{inv}(F)$ is
equal to the number of pairs $(i,j)$ such that $1 \le i < j \le n$, and the
vertex labeled by~$j$ lies on the shortest path in $T$ from the vertex
labeled by~$i$ to the root.

The major index $\operatorname{maj}(F)$ is equal to $\sum\limits_{x \in \operatorname{Des}(F)} h(x)$; here for
any vertex $x \in F$, $h(x)$ is the size of the subtree rooted at~$x$;
the descent set $\operatorname{Des}(F)$ of $F$ consists of the vertices $x \in F$ which have
the labeling strictly greater than the labeling of its child.}

Clearly that
\begin{gather*}
\dim({\mathfrak{A}}_n)_{{n \choose 2}} =\dim({\cal{PC}}_n)_{{n \choose 2}}
= \dim(H^{\star}({\cal{F}}l_{n},\Q))_{{n \choose 2}} = 1.
\end{gather*}
For example,
\begin{gather*}
\operatorname{Hilb}({\cal{PC}}_6,t)=(1,5,24,74,204,435,783,1144,1379,1346,1037,628,275,85,15,1)_{t}, \\
\operatorname{Hilb}({\mathfrak{A}}_6,t)=(1,5,15,35,70,126,204,300,405,490,511,424,245,85,15,1)_{t}, \\
\operatorname{Hilb}(H^{\star}({\cal{F}}l_{n},\Q),t) = (1,5,14,29,49,71,90,101,101,90,71,49,29,14,5,1)_{t}.
\end{gather*}
 We \textit{expect} that $\dim({\cal{PC}}_n)_{{n\choose 2}-1} =
{n \choose 2}$ and $\dim({\cal{PC}}_n)_{{n\choose 2}-2} = {\frac{3 n+5}{4}} {n+2 \choose 3} =s(n+2,2)$, where $s(n,k)$ denotes the Stirling number of the
first kind, see, e.g., {\rm \cite[$A000914$]{SL}}.
\end{prb}

\begin{prb}\quad
\begin{enumerate}\itemsep=0pt
\item[$(1)$] Is it true that $\operatorname{Hilb}({\cal{PC}}_n,t) - \operatorname{Hilb}({\mathfrak{A}}_n,t)
\in \N [t]$?
If so, as we \textit{expect}, does there exist an embedding of sets
 $ \iota\colon {\cal{F}}(n) \hookrightarrow {\cal{A}}_n$ such that $\operatorname{inv}(F) =
n(\iota(F))$ for all $F \in {\cal{F}}_n$?
See Section~{\rm \ref{section5.1}}, Definition~{\rm \ref{def5.5}}, for definitions of the
set ${\cal{A}}_n$ and statistics $n(T)$, $T \in {\cal{A}}_n$.

\item[$(2)$] \textit{Define} a ``natural'' bijection $ \kappa\colon \operatorname{STY}(\delta_n, \le n) \longleftrightarrow 2^{\delta_{n}} $ such that the set $\kappa(\operatorname{MT}(n))$
admits a ``nice'' combinatorial description.
\end{enumerate}
\end{prb}

Here $\operatorname{MT}(n)$ denotes the set of (increasing) monotone triangles, namely, a
subset of the set $ \operatorname{STY}(\delta_n, \le n)$ consisting of tableaux
$\{T= (t_{i,j})\, |\, i+j \le n+1,\, i \ge 1, j \ge 1 \}$ such that
$t_{i,j} \ge t_{i-1,j+1}$, $2 \le i \le n$, $1 \le j < n$, cf.~\cite{ST1};
$\delta_n=(n-1,n-2,\ldots,2,1)$;
 $ \operatorname{STY}(\delta_n, \le n)$ denotes the set of semistandard Young tableaux of
shape $\delta_n$ with entries bounded by $n$;
$2^{\delta_{n}}$ stands for the set of all subsets of boxes of the staircase
 diagram~$\delta_{n}$.
 It is well-known that $\#|\operatorname{STY}(\delta_n, \le n)| = 2^{\delta_{n}} =
 2^{{n \choose 2}}$.

\begin{comm}\label{com5.22} One can ask a natural question:
when do noncommutative elementary polynomials $e_1(\mathbb{A}),\dots,
e_n(\mathbb{A})$ form a \textit{$q$-commuting family}, i.e.,
$e_i(\mathbb{A}) e_j(\mathbb{A})= q e_j(\mathbb{A}) e_{i}(\mathbb{A})$, $1 \le i < j \le n$?

Clearly in the case of two variables one needs to necessitate the
following relations
\begin{gather*}
e_i e_j e_i+e_j e_j e_i=q e_j e_i e_i+q e_j e_i e_j,\qquad i < j.
\end{gather*}
Having in mind to construct a $q$-deformation of the plactic
algebra~${\cal{P}}_n$ such that the wanted $q$-commutativity conditions are
fulf\/illed, one would be forced to add the following relations
\begin{gather*}
q e_j e_i e_j=e_j e_j e_i \qquad \text{and}\qquad q e_j e_i e_i=e_i e_i e_j e_i, \qquad i < j.
\end{gather*}
It is easily seen that these two relations are compatible if\/f $q^2=1$.
Indeed,
\begin{gather*}
 e_j \underline{e_j e_i e_j} = q e_j \underline{e_i e_j e_i} = q^2 e_j e_j
e_i e_i \quad \Longrightarrow \quad q^2=1.
\end{gather*}
In the case $q=1$ one comes to the Knuth relations $({\rm PL1})$ and $({\rm PL2})$.
In the case $q= -1$ one
comes to the ``odd'' analogue of the Knuth relations, or ``odd'' plactic
relations (${\rm OPL}_n$), i.e., $({\rm OPL}_n):$
\begin{gather*}
 u_j u_i u_k= - u_j u_k u_i, \qquad \text{if}\quad i < j \le k \le n, \qquad \text{and}\\
 u_i u_k u_j = - u_k u_i u_j, \qquad \text{if} \quad i \le j < k \le n.
\end{gather*}
\end{comm}

\begin{pr}[A.K.] \label{prop5.23} Assume that the elements $\{u_1,\ldots,u_{n-1} \}$ satisfy the odd plactic relations $({\rm OPL}_n)$. Then the noncommutative elementary
polynomials $e_1(U),\ldots,e_n(U)$ are mutually anticommute.
\end{pr}

More generally, let ${\bf {\cal{Q}}_n}:= \{q_{ij}\}_{1 \le i < j \le n-1}$ be a
set of parameters.
 Def\/ine \textit{generalized plactic algebra} ${\cal{QP}}_n$ to be (unital)
associative algebra over the ring $\Z[\{q_{ij}^{\pm 1} \}_{1 \le i < j \le
n-1}]$ generated by elements $u_1, \ldots, u_{n-1} $ subject to the set of
 relations
\begin{gather}
 q_{ik} u_{j} u_{i} u_{k} = u_j u_{k} u_i, \qquad \text{if}\quad i < j \le k, \qquad \text{and}\nonumber\\
q_{ik} u_i u_k u_j =u_k u_i u_j, \qquad \text{if}\quad i \le j < k.\label{equation5.8}
\end{gather}
\begin{pr}\label{5.24} Assume that $q_{ij}:= q_j$, $\forall\, 1 \le i < j$ be a set of invertible parameters.
Then the reduced generalized plactic algebra ${\cal{QPC}}_n$ is a~free $\Z[q_2^{\pm 1},\ldots,q_{n-1}^{\pm 1}]$-module of rank equals to the
number of alternating sign matrices $\operatorname{ASM}(n)$.
Moreover,
\begin{gather*}
\operatorname{Hilb}({\cal{QPC}}_n,t) = \operatorname{Hilb}({\cal{PC}}_n,t), \operatorname{Hilb}({\cal{QP}}_n,t)=
\operatorname{Hilb}({\cal{P}}_n,t).
\end{gather*}
\end{pr}

Recall that \textit{reduced generalized plactic algebra} ${\cal{QPC}}_n$ is the quotient of the generalized plactic algebra by the two-sided ideal~$J_n$
introduced in Def\/inition~\ref{def5.15}.

\begin{ex}\label{exam5.25} \quad\samepage
\begin{enumerate}\itemsep=0pt
\item[(A)] Super plactic monoid \cite{LNS, LT}. Assume that the
set of generators $U:=\{u_1,\ldots,u_{n-1} \}$ is divided on two non-crossing
subsets, say~$Y$ and~$Z$, $Y \cup Z = U, Y \cap Z =\varnothing$. To each
element $u \in U$ let us assign the weight $wt(u)$ as follows: $wt(u) =0$
if $u \in Y$, and $wt(u)=1$ if $u \in Z$. Finally, def\/ine parameters of the
generalized plactic algebra ${\cal{QP}}_n$ to be $q_{ij}= (-1)^{wt(u_{i}) wt(u_{j})}$. As a result we led to conclude that the generalized plactic
algebra ${\cal{QP}}_n$ in question coincides with the super plactic algebra
${\cal{PS}}(V)$ introduced in~\cite{LT}. We will denote this algebra by
${\cal{SP}}_{k,l}$, where $k=|Y|$, $l=|Z|$. We refer the reader to papers~\cite{LT} and~\cite{LNS} for more details about connection of the super plactic
algebra and super Young tableaux, and super analogue of the Robinson--Schensted--Knuth correspondence. We are planning to report on some properties of the
Cauchy kernel in the (reduced) super plactic algebra elsewhere.

\item[(B)] $q$-analogue of the plactic algebra.
Now let $q \not= 0, \pm 1$ be a parameter, and assume that $q_{ij}=q$, $\forall\, 1 \le i < j \le n-1$. This case has been treated recently in \cite{Li}. We
expect that the generalized Knuth relations~\eqref{equation5.8} are related
with \textit{quantum version} of the tropical/geometric RSK-correspondence (work in progress), and, as expected, with a $q$-weighted version of the
Robinson--Schensted algorithm, presented in~\cite{OP}. Another interesting
\textit{problem} is to understand a meaning of ${\cal{Q}}$-plactic
polynomials coming from the decomposition of the (plactic) Cauchy kernels
${\cal{C}}_n$ and ${\cal{F}}_n$ in the reduced generalized plactic algebra
 ${\cal{QPC}}_n$ (work in progress).

\item[(C)] Quantum pseudoplactic algebra $\operatorname{PPL}_{n}^{(q)}$ \cite{KT}.
By def\/inition, the quantum pseudoplactic algebra $\operatorname{PPL}_{n}{(q)}$ is an
associative algebra, generated, say over $\mathbb{Q}$, by the set of elements $\{e_1,\ldots,e_{n-1}\}$ subject to the set of def\/ining relations
\begin{gather*}
(a) \quad (1+q) e_i e_j e_i -q e_i^2 e_j - e_j e_i^2 =0, \qquad\! (1+q) e_j e_i e_j -e_j^2 e_j - q e_i e_j^2 =0, \qquad\! i < j,\\
(b) \quad (e_j,(e_i,e_k)):=e_j e_i e_k - e_j e_k e_i - e_i e_k e_j + e_k e_i e_j = 0, \qquad i < j < k.
\end{gather*}
\end{enumerate}
\end{ex}

\textit{Note} that if $q=1$, then the relations $(a)$ can be written in the
form $(e_i,(e_i,e_j))=0$ and $(e_2,(e_j,(e_i,e_j))=0$ correspondingly.
Therefore, $\operatorname{PPL}_{3}^{(q=1)}$ is the universal enveloping algebra over $\mathbb{Q}$ of the Lie algebra ${\mathfrak{sl}}^{+}_{3}$. The quotient of the algebra $\operatorname{PPL}_{n}^{q=1}$ by the two-sided ideal generated by the elements $(e_i,(e_j,e_k))$, $i$, $j$, $k$ are distinct, is isomorphic to the algebra from Remark~\ref{rem2.3}.

Def\/ine noncommutative $q$-elementary polynomials $\Lambda_{k}(q;X_n)$, cf.~\cite{KT}, as follows
\begin{gather}\label{equation5.9}
\Lambda_{k}(X;q):= \sum_{n \ge i_1 >i_2 > \cdots > i_k \ge 1} (x_{i_{1}},(x_{i_{2}},( \ldots,(x_{i_{k-1}},x_{i_{k}})_{q}) \cdots )_{q})_{q}.
\end{gather}

 \begin{pr}[\cite{KT}] \label{prop5.26} The noncommutative $q$-elementary
polynomials
$ \{\!\Lambda_{k}(E_{n{-}1}{;}q)\}_{1 {\le} k {\le} n{-}1}\! \}$
are pairwise commute in the algebra $\operatorname{PPL}_{n}^{(q)}$.
\end{pr}

\subsection[Nilplactic algebra ${\cal {NP}}_n$]{Nilplactic algebra $\boldsymbol{{\cal {NP}}_n}$}\label{section5.2}

Let $\lambda$ be a partition and $\alpha$ be a composition of the same size.
Denote by ${\widehat {\operatorname{STY}}}(\lambda,\alpha)$ the set of columns and rows
strict Young tableaux $T$ of the shape $\lambda$ and content $\alpha$ such that
the corresponding tableau word~$w(T)$ is reduced, i.e., $l(w(T))=|T|$.

Denote by ${\cal B}_n$ the union of the sets ${\widehat {\operatorname{STY}}}(\lambda,\alpha)$
for all partitions $\lambda$ such that $\lambda_i \le n-i$ for $i=1,2,\dots,
n-1$, and all compositions~$\alpha$, $\alpha \subset \delta_n$.

For example, $|{\cal B}_n|=1,2,6,25,139,1008,\dots $, for
$n=1,2,3,4,5,6,\dots$.

{\samepage \begin{Theorem}\label{theorem5.27} \quad
\begin{enumerate}\itemsep=0pt
\item[$(1)$] In the nilplactic algebra ${\cal {NP}}_{n}$ the Cauchy kernel
has the following decomposition
\begin{gather*}%\label{equation5.10}
{\cal C}_n(\mathfrak{P},U) = \sum_{T \in {\cal B}_n} {\cal K}_{T}(\mathfrak{P})
 u_{w(T)}.
\end{gather*}
\item[$(2)$] Let $T \in {\cal B}_n$ be a tableau, and assume that its bottom code is a~partition. Then
\begin{gather*}%\label{equation5.11}
 {\cal K}_{T}(\mathfrak{P})= \prod_{(i,j) \in T} p_{\{i,T(i,j)-j+1 \}}.
\end{gather*}
\end{enumerate}
\end{Theorem}}

\begin{ex}\label{exam5.28} For $n=4$ one has
\begin{gather*}
{\cal C}_{4}(X,U)= K[0]+K[1]u_1+K[01]u_2+K[001]u_3+K[11]u_{12}+K[2](u_{21}+u_{31})\\
\hphantom{{\cal C}_{4}(X,U)=}{}
+K[101]u_{13}+
K[02]u_{32}+K[011]u_{23}+K[3]u_{321}+K[12]u_{312}+K[21]u_{212}\\
\hphantom{{\cal C}_{4}(X,U)=}{}
+K[111]u_{123}+K[021]u_{323}+K[201]u_{213}+K[31]u_{3212}+
K[301]u_{3213}\\
\hphantom{{\cal C}_{4}(X,U)=}{}
+K[22]u_{2132}+K[121]u_{3123}+K[211]u_{2123}+K[32]u_{32132}+
K[311]u_{32123}\\
\hphantom{{\cal C}_{4}(X,U)=}{}
+K[221]u_{21323}+K[321]u_{321323}.
\end{gather*}
\end{ex}

\subsection[Idplactic algebra ${\cal {IP}}_n$]{Idplactic algebra $\boldsymbol{{\cal {IP}}_n}$}\label{section5.3}

Let $\lambda$ be a partition and $\alpha$ be a composition of the same size.
Denote by ${\overline{\operatorname{STY}}}(\lambda,\alpha)$ the set of columns and rows
strict Young tableaux $T$ of the shape $\lambda$ and content $\alpha$ such that
$l(w(T))=\operatorname{rl}(w(T))$, i.e., the tableau word\footnote{See page~\pageref{footnote9} for the def\/inition of {\it tableau word}.}
$w(T)$ is a unique tableau word of
minimal length in the idplactic class of~$w(T)$, cf.\ Example~\ref{exam2.16}.

Denote by ${\cal D}_n$ the union of the sets
${\overline{\operatorname{STY}}}(\lambda,\alpha)$
for all partitions $\lambda$ such that $\lambda_i \le n-i$ for $i=1,2,\dots,n-1$, and all compositions $\alpha$, $l(\alpha) \le n-1$.

 For example, $\# |{\cal D}_n|=1,2,6,26,154,1197,\dots $, for $n=1,2,3,4,5,6,\dots$.

\begin{Theorem}\label{theorem5.29} \quad
\begin{enumerate}\itemsep=0pt
\item[$(1)$]
 In the idplactic algebra ${\cal {IP}}_{n}$ the Cauchy kernel
has the following decomposition
\begin{gather*}%\label{equation5.12}
{\cal C}_n(X,Y,U) =
\sum_{T \in {\cal D}_n} {\cal {\operatorname{KG}}}_{T}(X,Y)
 u_{w(T)}.
\end{gather*}
\item[$(2)$] Let $T \in {\cal D}_n$ be a tableau, and assume that its bottom code is a
partition. Then
\begin{gather*}%\label{equation5.13}
{\cal {\operatorname{KG}}}_{T}(X,Y)= {\cal K}_{T}(X,Y)= \prod_{(i,j) \in T} (x_i+y_{T(i,j)-j+1}).
\end{gather*}
\end{enumerate}
\end{Theorem}

\begin{ex}\label{exam5.30} For $n=4$ one has
\begin{gather*}
{\cal C}_{4}(X,U)=
\operatorname{KG}[0]+\operatorname{KG}[1]u_1+\operatorname{KG}[01]u_2+\operatorname{KG}[001]u_3+\operatorname{KG}[11]u_{12}
+\operatorname{KG}[2](u_{21}+u_{31})\\
\hphantom{{\cal C}_{4}(X,U)=}{}
+
\operatorname{KG}[101]u_{13}+
\operatorname{KG}[02]u_{32}+\operatorname{KG}[011]u_{23}+\operatorname{KG}[3]u_{321}+\operatorname{KG}[12]u_{312}\\
\hphantom{{\cal C}_{4}(X,U)=}{}
+\operatorname{KG}[21]u_{212}+
\operatorname{KG}[111]u_{123}+\operatorname{KG}[021]u_{323}+\operatorname{KG}[201](u_{313}+u_{213})\\
\hphantom{{\cal C}_{4}(X,U)=}{}
+K[31]u_{3212}+
\operatorname{KG}[301]u_{3213}+\operatorname{KG}[22]u_{2132}+\operatorname{KG}[121]u_{3123}\\
\hphantom{{\cal C}_{4}(X,U)=}{}
+\operatorname{KG}[211]u_{2123}+
\operatorname{KG}[32]u_{32132}+
\operatorname{KG}[311]u_{32123}+\operatorname{KG}[221]u_{21323}\\
\hphantom{{\cal C}_{4}(X,U)=}{}
+\operatorname{KG}[321]u_{321323}.
\end{gather*}
\end{ex}

\begin{Theorem}\label{theorem5.31} For each composition~$\alpha$ the key Grothendieck
polynomial $\operatorname{KG}[\alpha](X)$ is a linear combination of key
polynomials $K[\beta](X)$ with nonnegative integer coefficients.
\end{Theorem}

\subsection[NilCoxeter algebra ${\cal {NC}}_n$]{NilCoxeter algebra $\boldsymbol{{\cal {NC}}_n}$}\label{section5.4}

\begin{Theorem}\label{theorem5.32}
 In the nilCoxeter algebra ${\cal {NC}}_{n}$ the Cauchy kernel
has the following decomposition
\begin{gather*}%\label{equation5.14}
{\cal C}_n(X,Y,U) =
\sum_{w \in {\mathbb S}_n} {\s}_w(X,Y) u_{w}.
\end{gather*}
\end{Theorem}

Let $w \in {\mathbb S}_n$ be a permutation, denote by $R(w)$ the set of all its
reduced decompositions. Since the nilCoxeter algebra ${\cal {NC}}_n$ is the
quotient of the nilplactic algebra ${\cal {NP}}_n$, the set $R(w)$ is the
union of nilplactic classes of some tableau words $w(T_i)$:
$R(w)= \bigcup C(T_i)$.
Moreover, $R(w)$ consists of only one nilplactic class if and only if $w$ is
a {\it vexillary} permutation. In general case we see that the set of
compatible sequences $CR(w)$ for permutation~$w$ is the union of sets~$C(T_i)$.

\begin{cor}\label{cor5.33} Let $w \in {\mathbb S}_n$ be a permutation of length~$l$, then
\begin{enumerate}\itemsep=0pt
\item[$(1)$] ${\s}_{w}(X,Y) = \sum\limits_{{\bf b} \in CR(w)} x_{b_1} \cdots x_{b_{l}}$.
\item[$(2)$] Double Schubert polynomial ${\s}_{w}(X,Y)$ is a linear combination of
double key polynomials ${\cal K}_{T}(X,Y)$, $T \in {\cal B}_n, w=w(T)$,
 with nonnegative integer coefficients.
 \end{enumerate}
\end{cor}

\subsection[IdCoxeter algebras ${\cal {IC}}_n^{\pm}$]{IdCoxeter algebras $\boldsymbol{{\cal {IC}}_n^{\pm}}$}\label{section5.5}

\begin{Theorem}\label{theorem5.34}
 In the IdCoxeter algebra ${\cal {IC}}_{n}^{+}$ with $\beta=1$,
the Cauchy kernel has the following decomposition
\begin{gather*}%\label{equation2.15}
{\cal C}_n(X,Y,U) =
\sum_{w \in {\mathbb S}_n} {\cal G}_w(X,Y) u_{w}.
\end{gather*}
\end{Theorem}

\begin{Theorem}\label{theorem5.35}
 In the IdCoxeter algebra ${\cal {IC}}_{n}^{-}$ with $\beta=-1$,
one has the following decomposition
\begin{gather*}%\label{equation5.16}
\prod_{i=1}^{n-1} \left\{{\prod_{j=n-1}^{i} ((1+x_{i})
(1+y_{j-i+1})+(x_{i}+y_{j-i+1}) u_{j})} \right\} =
\sum_{w \in {\mathbb S}_n} {\cal H}_w(X,Y) u_{w}.
\end{gather*}
\end{Theorem}

A few remarks in order.
\begin{enumerate}\itemsep=0pt
\item[$(a)$] The (dual) Cauchy identity \eqref{equation5.9} is still valid in the idplactic
algebra with constrain $u_{i}^2 = - \beta u_i$, $i=1,\ldots, n-1$.

\item[$(b)$] The left hand side of the identity \eqref{equation5.9} can be written in the
following form
\begin{gather*}%\label{equation5.17}
\prod_{1 \le i,j \le n \atop i+j \le n} (x_i+y_j)
\prod_{i=1}^{n-1} \left \{{\prod_{j=n-1}^{i}} \frac{1}{1-(x_{i}+ y_{j-i+1}+ \beta x_{i} y_{j-i+1}) u_{j}} \right\}.
\end{gather*}
Indeed, $(1+ \beta x +x u_{i})(1-x u_{i}) = 1+ \beta x$, since $u_{i}^2= - \beta u_{i}$.
\end{enumerate}

Let $w \in {\mathbb S}_n$ be a permutation, denote by $\operatorname{IR}(w)$ the
set of all decompositions in the idCoxeter algebra ${\cal IC}_n$ of the
element $u_w$ as the product of the generators $u_{i}$, $1 \le i \le n-1$,
 of the algebra~${\cal IC}_n$.
Since the idCoxeter algebra ${\cal {IC}}_n$ is the
quotient of the idplactic algebra ${\cal {IP}}_n$, the set $\operatorname{IR}(w)$
is the union of idplactic classes of some tableau words~$w(T_i)$:
$\operatorname{IR}(w)= \bigcup \operatorname{IR}(T_i)$.
Moreover, the set of
compatible sequences $\operatorname{IC}(w)$ for permutation $w$ is the union
of sets $\operatorname{IC}(T_i)$.

\begin{cor}\label{cor5.36} Let $w \in {\mathbb S}_n$ be a permutation of length~$l$, then
\begin{enumerate}\itemsep=0pt
\item[$(1)$] ${\cal G}_{w}(X,Y) =
\sum\limits_{{\bf b} \in IC(w)}
 \prod\limits_{i=1}^{l}(x_{b_i}+y_{a_i-b_i+1})$.
\item[$(2)$] Double Grothendieck polynomial ${\cal G}_{w}(X,Y)$ is a linear
combination of double key Grothen\-dieck polynomials ${\cal {\operatorname{KG}}}_{T}(X,Y)$,
$ T \in {\cal B}_n$, $w=w(T)$, with nonnegative integer coefficients.
\end{enumerate}
\end{cor}

\section[${\cal F}$-kernel and symmetric plane partitions]{$\boldsymbol{{\cal F}}$-kernel and symmetric plane partitions}\label{section6}

Let us f\/ix natural number $n$ and $k$, and a partition $\lambda
\subset (n^k)$. Clearly the
number of such partitions is equal to ${n +k \choose n}$; note that in the
case $n=k$ the number ${2 n \choose n}$ is equal to the {\it Catalan number
of type~$B_n$}.

Denote by ${\cal B}_{n,k}(\lambda)$ the set of semistandard Young tableaux of
shape $\lambda \subset (n^k)$ f\/illed by the numbers from the set $\{1,2,\ldots,n \}$. For a
tableau $T \in {\cal B}_{n,k}$ set as before,
\begin{gather*}
n(T):= \operatorname{Card} \{(i,j) \in \lambda
\mid T(i,j)=n \},
\end{gather*}
 and def\/ine polynomial
\begin{gather}\label{equation6.1}
{\cal B}_{n,k}(\lambda)(q):= \sum_{T \in {\cal B}_{n,k}(\lambda)}
q^{\lambda_{1}-n(T)}.
\end{gather}
Denote by ${\cal B}_{n,k} := \bigcup_{\lambda \subset (n^{k})}
{\cal B}_{n,k}(\lambda)$.

\begin{lem}[\cite{G,KGV}]\label{lem6.1} The number of elements in the set
${\cal B}_{n,k}$ is equal to
\begin{gather*}
 \# | {\cal B}_{n,k} | = \prod_{1 \le i \le j \le k} {i+j+n-1 \over i+j-1} = \prod_{0 \le 2 a \le k -1} \frac{{n+2 k- 2a -1 \choose n}}
{{n+2 a\choose n}}.
\end{gather*}
\end{lem}

See also \cite[$A073165$]{SL} for other combinatorial interpretations of the
numbers $ \# | {\cal B}_{n,k} |$. For example, the number
$ \# | {\cal B}_{n,k} |$ is equal to the number of symmetric plane partitions
 that f\/it inside the box $n \times k \times k$. Note that $B_{n,k}= T(n+k,k)$, where the triangle of positive integers $\{T(n+k,k)\}$ can be found in \cite[$A102539$]{SL}.

\begin{pr}\label{prop6.2} One has
\begin{gather*}
\bullet \quad \# |{\cal B}_{n,n} | := \operatorname{SPP}(n+1)= \operatorname{TSPP}(n+1) \times \operatorname{ASM}(n),\\
\hphantom{\bullet} \quad \# |{\cal B}_{n,n+1} | = \operatorname{TSPP}(n+1) \times \operatorname{ASM}(n+1),
\end{gather*}
where $\operatorname{TSPP}(n)$ denotes the number of totally symmetric plane partitions f\/it
inside the $n \times n \times n$-box, see, e.g., {\rm \cite[$A005157$]{SL}}, whereas
$\operatorname{ASM}(n)=\operatorname{TSSCPP}(2n)$ denotes the of
$n \times n$ alternating sign matrices, and $\operatorname{TSSCPP}(2n)$ denotes the number of
totally symmetric self-complimentary plane partitions fit inside the
$2n \times 2n \times 2n$-box.
\begin{gather*}
\bullet \quad \# | {\cal B}_{n+2,n} | = \# | {\cal B}_{n,n+1} |.
\end{gather*}
\end{pr}

Note that in the case $n=k$ the number ${\cal{B}}_{n} := {\cal{B}}_{n,n}$ is
equal to the number of symmetric plane portions f\/itting inside the
$n \times n \times n$-box, see~\cite[$A049505$]{SL}. Let us point out that in
general it may happen that the number $\# |{\cal B}_{n,n+2} |$ is not
divisible by any $\operatorname{ASM}(m)$, $m \ge 3$. For example, ${\cal{B}}_{3,5}= 4224=
2^5 \times 3 \times 11$. On the other hand, it's possible that the number
$\# |{\cal B}_{n,n+2} |$ is divisible by $\operatorname{ASM}(n+1)$, but does not divisible by
$\operatorname{ASM}(n+2)$. For example, ${\cal{B}}_{4,6}=306735 = 715 \times 429$, but $
306735\, {\nmid} \, 7436 =\operatorname{ASM}(6)$.

\begin{exer}\label{exer6.3} \quad
\begin{enumerate}\itemsep=0pt
\item[$(a)$] Show that $B_{n+4,n}$ is divisible by
\begin{gather*}
 \begin{cases}
\operatorname{TSPP}(n+2), & \text{if} \ \ n \equiv 1 \ (\operatorname{mod}~2), \ n \ge 3,\\
\operatorname{ASM}(n+2), & \text{if} \ \ n \equiv 2 \ (\operatorname{mod}~8),\\
\operatorname{ASM}(n+1) \ \text{and} \ \operatorname{ASM}(n+2), & \text{if} \ \ n \equiv 4 \ (\operatorname{mod}~8), \ n \neq 4; \ B_{8,4}=\operatorname{ASM}(5)^2,\\
\operatorname{ASM}(n+1) \operatorname{and} \ \operatorname{ASM}(n+2), & \text{if} \ \ n \equiv 6 \ (\operatorname{mod}~8),\\
\operatorname{ASM}(n+1), & \text{if} \ \ n \equiv 0 \ (\operatorname{mod}~8), \ n \ge 1.
\end{cases}
\end{gather*}
\item[$(b)$] Show that $B_{n,n+4}$ is divisible by
\begin{gather*}
\begin{cases}
\operatorname{ASM}(n+1), & \text{if} \ \ n \equiv 0 \ (\operatorname{mod}~2), \\
\operatorname{TSPP}(n+1), & \text{if} \ \ n \equiv 1 \ (\operatorname{mod}~2).
\end{cases}
\end{gather*}
\end{enumerate}
\end{exer}

In all cases listed in Exercise~\ref{exer6.3}, it is an open {\it problem} to give
combinatorial interpretations of the corresponding ratios.

\begin{prb}\label{prob6.4} Let $a$ is equal to either $0$ or $1$. Construct
bijection between the set $\operatorname{SPP}(n,n+a,$ $n+a)$ of symmetric plane partitions
fitting inside the box $n \times n+a \times n+a$ and the set
of pairs $(P,M)$ where $P$ is the totally symmetric plane partitions
fitting inside the box $n \times n \times n$ and $M$ is an alternating sign
matrix of size $n+a \times n+a$.
\end{prb}

\begin{ex}\label{exam6.5} Take $n=3$. One has $\# |{\cal{B}}_{3}|= 112 = 16 \times 7$. The number of partitions $\lambda \subset (3^3)$ is
equal to 20, namely, the following partitions
\begin{gather*}
\big\{\varnothing,(1),(2),(1,1),(3),(2,1),\big(1^3\big),(3,1),(2,2),\big(2,1^2\big),(3,2),\big(3,1^2\big),\big(2^2,1\big),\big(3^2\big), \\
(3,2,1),\big(2^3\big),\big(3^2,1\big),\big(3,2^2\big),\big(3^2,2\big),\big(3^3\big) \big\},
\end{gather*}
 and
\begin{gather*}
\begin{split}
& {\cal{B}}_3(q):= \sum_{\lambda \subset (3^3)}
 \# | {\cal{B}}_3(\lambda)| q^{|\lambda|} = (1,3,9,19,24,24,19,9,3,1)\\
& \hphantom{{\cal{B}}_3(q):= \sum_{\lambda \subset (3^3)}
 \# | {\cal{B}}_3(\lambda)| q^{|\lambda|}}{} = (1+q)^3(1+q^2)\big(1+5 q^2 +q^4\big).
 \end{split}
\end{gather*}
 Note, however, that
\begin{gather*}
\sum_{\lambda \subset (4^4)} \# | {\cal B}_4(\lambda)| q^{|\lambda|} =
(1,4,16,44,116,204,336,420,490,420,336,204,116,44,16,4,1)
\end{gather*}
 is an irreducible polynomial, but its value at $q=1$ is equal to $2772=66
\times 42$.
\end{ex}

Let ${\bf p}=(p_{i,j})_{1 \le i \le n, 1 \le j \le k}$ be a
$n \times k$ matrix of variables.
\begin{de}\label{def6.6} Def\/ine the kernel ${\cal F}_{n,k}({\bf p}, U)$ as follows
\begin{gather*}%\label{equation6.2}
 {\cal F}_{n,k}({\bf p}, U) = \prod_{i=1}^{k-1} \prod_{j=
n-1}^{1} (1+ p_{i, {\overline {j-i+1}}^{(n)}} u_j),
\end{gather*}
where for a f\/ixed $n \in \N$ and an integer $a \in \Z$, we set
\begin{gather*}
 \overline{a}=\overline{a}^{(n)}: = \begin{cases}
 a, & \text{if} \ \ a \ge 1,\\
 n+a-1, & \text{if} \ \ a \le 0.
 \end{cases}
\end{gather*}
\end{de}

For example,
\begin{gather*}
{\cal F}_{3}({\bf p},U)= (1+p_{1,2} u_2)(1+p_{1,1} u_1)(1+p_{2,1} u_2)(1+p_{2,2} u_1).
\end{gather*}

In the plactic algebra ${\cal{FP}}_{3,3}$ one has
\begin{gather*}
{\cal F}_{3,3}({\bf p},U) = 1
 + (p_{1,1}+p_{2,2}) u_1
+ (p_{1,2}+p_{2,1}) u_2
+p_{1,1} p_{2,1} u_{11}
+p_{1,1} p_{2,1} u_{12} \\
\hphantom{{\cal F}_{3,3}({\bf p},U) =}{}
+ (p_{1,2} p_{1,1}+p_{1,2} p_{2,2} +p_{2,1} p_{2,2}) u_{21}
+p_{1,2} p_{2,1} u_{22} \\
\hphantom{{\cal F}_{3,3}({\bf p},U) =}{}
+(p_{1,1} p_{1,2} p_{2,2}+ p_{1,2} p_{2,2} p_{2,1}) u_{212}
+(p_{1,1} p_{1,2} p_{2,2}+p_{1,1} p_{2,2} p_{2,1}) u_{211} \\
\hphantom{{\cal F}_{3,3}({\bf p},U) =}{}
+p_{1,1} p_{1,2} p_{2,1} p_{2,2} u_{2121}.
\end{gather*}

\begin{de}\label{def6.7} Def\/ine algebra ${\cal{PF}}_{n,k}$ to be the quotient of the plactic
algebra ${\cal{P}}_n$ by the two-sided ideal $I_n$ generated by the set of
monomials
\begin{gather*}
\{u_{i_{1}} u_{i_{2}} \cdots u_{i_{k}} \}, \qquad 1 \le i_{1} \le i_{2} \le
\cdots \le i_k \le n -1.
\end{gather*}
\end{de}

\begin{Theorem}\label{theorem6.8} \quad
\begin{gather*}
\operatorname{Hilb}({\cal{PF}}_{n,k},q)= {\cal{B}}_{n -1 ,k-1}(q),
\end{gather*}
In particular,
\begin{itemize}\itemsep=0pt
\item The algebra ${\cal{PF}}_{n,n}$ has dimension equals to the number
of symmetric plane partitions $\operatorname{SPP}(n - 1)$,
\begin{gather*}%\label{equation6.3}
 \operatorname{Hilb}({\cal{PF}}_{n,k},q) = q^{\frac{k n}{2}} {\mathfrak{so}}_{(\frac{k}{2})^n } \big(\underbrace{q^{\pm 1}, \dots, q^{\pm 1}}_{k},1\big),
\end{gather*}
where ${\mathfrak{so}}_{(\frac{k}{2})^n} \big(\underbrace{q^{\pm 1}, \dots, q^{\pm 1}}_{k},1\big)$ denotes the specialization $x_{2j}=q$, $x_{2j-1}=q^{-1}$, $1 \le j \le k$, of the character ${\mathfrak{so}}_{\lambda} \big(x_1, x_1^{-1}\ldots, x_{k},x_{k}^{-1},1\big)$ of the odd orthogonal Lie algebra
${\mathfrak{so}}(2k+1)$ corresponding to the highest weight $\lambda = \big(\underbrace{\tfrac{k}{2},\ldots,\tfrac{k}{2}}_{n}\big)$.

\item $\deg_{q} \operatorname{Hilb}({\cal{PF}}_{n,k},q) =(n-1)(k-1)$, and
$\dim ({\cal{PF}}_{n,k})_{(n-1)(k-1)}=1$.

\item The Hilbert polynomial $\operatorname{Hilb}({\cal{PF}}_{n,k},q)$ is symmetric
 and unimodal polynomial in the va\-riab\-le~$q$.

\item $\operatorname{Hilb}(({\cal{PF}}_{n,k})^{ab},q) =\sum\limits_{j=0}^{k-1} {n+j-2 \choose n-2} q^{j}$, $\dim {(\cal{PF}}_{n,k})^{ab}= {n+k-2 \choose k-1}$.
\end{itemize}
\end{Theorem}

The key step in proofs of Lemma~\ref{lem6.1} and Theorem~\ref{theorem6.8} is based on the following
 identity
\begin{gather*}%\label{equation6.4}
 \sum_{\lambda \subset (n^{k})} s_{\lambda} (x_1,\ldots,x_k) =
(x_1 \cdots x_k)^{n/2} {\mathfrak{so}}_{(\frac{k}{2})^n}\big(x_{1},x_{1}^{-1}, \ldots,x_{k},x_{k}^{-1},1\big),
\end{gather*}
see, e.g., \cite[Chapter~I, Section~5, Example~19]{Ma}, \cite{KGV} and the literature quoted
therein.
\begin{prb}\label{prob6.9}
Let $\Gamma:= {\Gamma_{n,m}}^{k,\ell} = (n^{k},m^{\ell})$, $n \ge m$ be a
``fat hook''. Find generalizations of the identity~\eqref{equation6.1} and those listed
in~{\rm \cite[p.~71]{Kin}}, to the case of fat hooks, namely to find ``nice''
expressions for the following sums
\begin{gather*}
\sum_{\lambda \subset \Gamma} s_{\lambda} (X_{k+\ell}),\qquad
\sum_{\lambda \subset \Gamma} s_{\lambda} (X_{k+\ell}) s_{\lambda}(Y_{k+\ell}).
\end{gather*}
Find ``bosonic'' type formulas for these sum at the limit
$n \longrightarrow \infty$, $\ell \longrightarrow \infty$, $m$, $k$ are fixed.
\end{prb}

\begin{ex}\label{exam6.10}
\begin{gather*}
\operatorname{Hilb}({\cal{PF}}_{2,3},q)= (1,3,9,9,9,3,1)_{q}, \qquad \dim ({\cal{PF}}_{2,4}) = 35 = 5 \times 7, \\
\dim ({\cal{PF}}_{2,5}) = 126 =3 \times 42, \qquad \dim {\cal{PF}}_{2,n}= {2 n+1 \choose n}= (2 n+1) \operatorname{Cat}_{n} \\
(\text{see, e.g., \cite[$A001700$]{SL}}),\\
\operatorname{Hilb}({\cal{PF}}_{3,4},q)= (1,4,16,44,81,120,140,120,81,44,16,4,1)_{q}, \\
\dim ({\cal{PF}}_{3,4}) = 672 = 16 \times 42,\\
\operatorname{Hilb}({\cal{PF}}_5,q)=(1,4,16,44,116,204,336,420,490,420,336,204,116,44,16,4,1)_{q}, \\
\dim({\cal{PF}}_{5,5}) =2772=66 \times 42.
\end{gather*}
\end{ex}

\begin{pr}\label{prop6.11}
\begin{gather*}
\operatorname{Hilb} ({\cal{PF}}_{2,n},q) = \sum_{k=0}^{2 n} {n \choose \big[\frac{k}{2} \big]} {n \choose \big[\frac{k+1}{2} \big]} q^k;
\qquad \text{recall} \qquad \dim ({\cal{PF}}_{2,n}) = {2 n + 1 \choose n}.
\end{gather*}
Therefore, $\operatorname{Hilb}({\cal{PF}}_{2,n},q)$ is equal to the generating function
for the number of symmetric Dyck paths of semilength $2 n-1$ according to the
number of peaks, see~{\rm \cite[$A088855$]{SL}},
\begin{gather*}
 \dim ({\cal{PF}}_{3,n}) = 2^n \operatorname{Cat}_{n+1}, \qquad \text{if} \quad n \ge 1.
 \end{gather*}
\end{pr}

For example,
\begin{gather*}
\dim ({\cal{PF}}_{3,6}) =27456 =64 \times 429,\\
\operatorname{Hilb}({\cal{PF}}_{3,6},q) = (1,6,36,146,435,1056,2066,3276,4326,4760,4326,3276,2066,1056,\\
\hphantom{\operatorname{Hilb}({\cal{PF}}_{3,6},q) = (}{} 435,146,36,6,1).
\end{gather*}
 Several interesting interpretations of these numbers are given in
\cite[$A003645$]{SL}.

{\samepage \begin{Theorem}\label{theorem6.12} \quad
\begin{itemize}\itemsep=0pt
\item Symmetric plane partitions and Catalan numbers:
\begin{gather*}
\# |{\cal{B}}_{4,n}| = {\frac{1}{2}} \operatorname{Cat}_{n+1} \times \operatorname{Cat}_{n+2}.
\end{gather*}
\item Symmetric plane partitions and alternating sign matrices:
\begin{gather*}
\# |{\cal{B}}_{n+3,n}| = {\frac{1}{2}} \operatorname{TSPP}(n+1) \times \operatorname{ASM}(n+1) = \frac{1}{2} \# |{\cal{B}}_{n+1,n+1}|.
\end{gather*}
\item Plane partitions and alternating sign matrices invariant under
a half-turn:
\begin{gather*}
 \# |\operatorname{PP}(n)| = \operatorname{ASM}(n) \times \operatorname{ASMHT}(2n),
\end{gather*}
where $PP(n)$ denotes the number of plane partitions fitting inside an $n \times n \times n$ box, see, e.g., {\rm \cite{B,Ku,Ma}}, {\rm \cite[$A008793$]{SL}} and the literature quoted theirin; $\operatorname{ASMHT}(2n)$ denotes the number
of alternating sign $2n \times 2n$-matrices invariant under a half-turn, see,
e.g., {\rm \cite{B, Ku,Ok,ST1}}, {\rm \cite[$A005138$]{SL}}.
\item Plactic decomposition of the ${\cal{F}}_n$-kernel:
\begin{gather}\label{equation6.5}
 {\cal{F}}_{n,m}({\bf p},U) =\sum_{T} u_{T} U_{T}(\{p_{ij} \}),
\end{gather}
where summation runs over the set of semistandard Young tableaux $T$ of shape
$\lambda \subset (n)^m$ filled by the numbers from the set $\{1,\ldots, m\}$.
\item $U_{T}(\{p_{ij}=1,\, \forall \, i,j\}) =
\dim V_{\lambda'}^{{{\mathfrak{gl}}(m)}}$, where $\lambda$ denotes the
shape of a tableau~$T$, and~$\lambda'$ denotes the conjugate/transpose of a
partition $\lambda$.
\end{itemize}
\end{Theorem}}

\begin{exer}\label{exer6.13}
It is well-known~\cite{Gou} that that the number $\operatorname{Cat}_{n+1} \operatorname{Cat}_{n+2}$ counts the number~$S_{n+1}^{(4)}$ of standard Young tableaux having $2 n +1$ boxes and at most four rows.
Give a~bijective proof of the equality $\# |{\mathcal{B}}_{4,n}| = {\frac{1}{2}} S_{n+1}^{(4)}$.
\end{exer}

\appendix

\section{Appendix}\label{section7}

\subsection[Some explicit formulas for $n=4$ and compositions $\alpha$ such that $\alpha_i \le n-i$ for $i=1,2,\dots$]{Some explicit formulas for $\boldsymbol{n=4}$ and compositions $\boldsymbol{\alpha}$\\ such that $\boldsymbol{\alpha_i \le n-i}$ for $\boldsymbol{i=1,2,\dots}$}\label{section7.1}

$(1)$ Schubert and $(-\beta)$-Grothendieck polynomials
$\cal{G}^{-}[\alpha]:= \cal{G}^{-\beta}[\alpha]$ for $n=4$:
\begin{gather*}
{\s}_{1234}={\s}[0]=1 = {\cal{G}}^{-}[0], \qquad
{\s}_{2134}={\s}[1]=x_1 = {\cal{G}}^{-}[1] , \\
{\s}_{1324}={\s}[01]=x_1+x_2 = {\cal{G}}^{-}[01] + \beta {\cal{G}}^{-}[11], \\
{\s}_{1243}={\s}[001]=x_1+x_2+x_3= {\cal{G}}^{-}[001] + \beta
 {\cal{G}}^{-}[011] + \beta^2 {\cal{G}}^{-}[111] , \\
{\s}_{3124}={\s}[2]=x_1^2 = {\cal{G}}^{-}[2], \qquad
{\s}_{2314}={\s}[11]=x_1x_2 = {\cal{G}}^{-}[11], \\
{\s}_{2143}={\s}[101]=x_1^2+x_1x_2+x_1x_3 = {\cal{G}}^{-}[101] + \beta {\cal{G}}^{-}[201] + \beta^2 {\cal{G}}^{-}[111], \\
{\s}_{1342}={\s}[011]=x_1x_2+x_1x_3+x_2x_3 = {\cal{G}}^{-}[011] + 2 \beta {\cal{G}}^{-}[111], \\
{\s}_{1423}={\s}[02]=x_1^2+x_1x_2+x_2^2 = {\cal{G}}^{-}[02] + \beta {\cal{G}}^{-}[12] + \beta^2 {\cal{G}}^{-}[22], \\
{\s}_{4123}={\s}[3]=x_1^3 = {\cal{G}}^{-}[3], \qquad
{\s}_{3214}={\s}[21]=x_1^2x_2 = {\cal{G}}^{-}[21], \\
{\s}_{2341}={\s}[111]=x_1x_2x_3 = {\cal{G}}^{-}[111], \\
{\s}_{2413}={\s}[12]=x_1^2x_2+x_1x_2^2 = {\cal{G}}^{-}[12] + \beta {\cal{G}}^{-}[22], \\
{\s}_{1432}={\s}[021]=x_1^2x_2+x_1^2x_3+x_1x_2^2+x_2^2x_3+x_1x_2x_3 \\
\hphantom{{\s}_{1432}}{} = {\cal{G}}^{-}[021] + 2 \beta {\cal{G}}^{-}[121] + \beta {\cal{G}}^{-}[22] + \beta^2 {\cal{G}}^{-}[221], \\
{\s}_{3142}={\s}[201]=x_1^2x_2+x_1^2x_3 = {\cal{G}}^{-}[201] + \beta {\cal{G}}^{-}[211], \\
{\s}_{4213}={\s}[31]=x_1^3x_2 = {\cal{G}}^{-}[31], \qquad
{\s}_{3412}={\s}[22]=x_1^2x_2^2= {\cal{G}}^{-}[22] , \\
{\s}_{4132}={\s}[301]=x_1^3x_2+x_1^3x_3 = {\cal{G}}^{-}[301]+ \beta {\cal{G}}^{-}[311] , \\
{\s}_{3241}={\s}[211]=x_1^2x_2x_3 = {\cal{G}}^{-}[211], \\
{\s}_{2431}={\s}[121]=x_1^2x_2x_3+x_1x_2^2x_3 = {\cal{G}}^{-}[121] + \beta {\cal{G}}^{-}[221], \\
{\s}_{4312}={\s}[32]=x_1^3x_2^2 = {\cal{G}}^{-}[32], \qquad
{\s}_{4231}={\s}[311]=x_1^3x_2x_3 = {\cal{G}}^{-}[311], \\
{\s}_{3421}={\s}[221]=x_1^2x_2^2x_3 = {\cal{G}}^{-}[211], \qquad
{\s}_{4321}={\s}[321]=x_1^3x_2^2x_3 = {\cal{G}}^{-}[321].
\end{gather*}

\begin{Theorem}[cf.~\protect{\cite[Section 5.5]{LS7}}]\label{theorem7.1}
Each Schubert polynomial is a linear combination of $(- \beta)$-Grothendieck
polynomials with nonnegative coefficients from the ring~$\N[\beta]$.
\end{Theorem}

$(2)$ Key and reduced key polynomials:
\begin{gather*}
K[0]=1={\widehat K}[0], \qquad
K[1]=x_1={\widehat K}[1], \qquad
K[01]=x_1+x_2, \qquad {\widehat K}[01]=x_2, \\
K[001]=x_1+x_2+x_3, \qquad {\widehat K}[001]=x_3, \qquad
K[2]=x_1^2={\widehat K}[2], \\
K[11]=x_1x_2={\widehat K}[11], \qquad
K[101]=x_1x_2+x_1x_3, \qquad {\widehat K}[101]=x_1x_3, \\
K[02]=x_1^2+x_1x_2+x_2^2, \qquad {\widehat K}[02]=x_1x_2+x_2^2, \qquad
K[011]=x_1x_2+x_1x_3+x_2x_3,\\ {\widehat K}[011]=x_2x_3, \qquad
K[3]=x_1^3={\widehat K}[3], \qquad
K[21]=x_1^2x_2={\widehat K}[21], \\
K[111]=x_1x_2x_3={\widehat K}[111], \qquad
K[12]=x_1^2x_2+x_1x_2^2, \qquad {\widehat K}[12] =x_1x_2^2, \\
K[021]=x_1^2x_2+x_1^2x_3+x_1x_2^2+x_2^2x_3+x_1x_2x_3, \qquad
 {\widehat K}[021]=x_1x_2x_3+x_2^2x_3,\\
K[201]=x_1^2x_2+x_1^2x_3,\qquad {\widehat K}[201]=x_1^2x_3, \qquad
K[31]=x_1^3x_2={\widehat K}[31], \\
K[22]=x_1^2x_2^2={\widehat K}[22], \qquad
K[211]=x_1^2x_2x_3={\widehat K}[211], \qquad
K[301]=x_1^3x_2+x_1^3x_3,\\
 {\widehat K}[301]=x_1^3x_3, \qquad
K[121]=x_1^2x_2x_3+x_1x_2^2x_3,\qquad {\widehat K}[121]=x_1x_2^2x_3, \\
K[32]=x_1^3x_2^2={\widehat K}[32], \qquad
K[311]=x_1^3x_2x_3={\widehat K}[311], \qquad
K[221]=x_1^2x_2^2x_3={\widehat K}[221], \\
K[321]=x_1^3x_2^2x_3={\widehat K}[321].
\end{gather*}
Note that if $n=4$, then ${\s}[\alpha]=K[\alpha]$ for all
$\alpha \subset \delta_4$, except $\alpha =(101)$ in which
${\s}[101]= K[2]+K[101]$.

$(3)$ Grothendieck and dual Grothendieck polynomials for $\beta=1$:
\begin{gather*}
{\cal G}_{1234}={\cal G}[0]= 1 = \s[0] , \qquad
 {\cal{H}}[0]=(1+x_1)^3(1+x_2)^2(1+x_3), \\
{\cal G}_{2134}={\cal G}[1]=x_1 = \s[1], \qquad
 {\cal H}[1]=(1+x_1)^2(1+x_2)^2(1+x_3){\cal G}[1], \\
{\cal G}_{1324}={\cal G}[01]=x_1+x_2+x_1x_2 =\s[01]+\s[11] , \\
 {\cal H}[01]= (1+x_1)^2(1+x_2)(1+x_3){\cal G}[01], \\
{\cal G}_{1243}={\cal G}[001]=x_1+x_2+x_3+x_1x_2+x_1x_3+x_2x_3+
x_1x_2x_3 \\
\hphantom{{\cal G}_{1243}}{}
= \s[001]+\s[011]+\s[111], \\
 {\cal H}[001]=(1+x_1)^2(1+x_2){\cal G}[001], \\
{\cal G}_{3124}={\cal G}[2]=x_1^2 = \s[2] , \qquad
 {\cal H}[2]= (1+x_1)(1+x_2)^2(1+x_3){\cal G}[2], \\
{\cal G}_{2314}={\cal G}[11]=x_1x_2 = \s[11], \qquad
 {\cal H }[11]=
(1+x_1)^2(1+x_2)(1+x_3){\cal G}[11], \\
{\cal G}_{2143}={\cal G}[101]=x_1^2+x_1x_2+x_1x_3+x_1^2x_2+x_1^2x_3+
x_1x_2x_3+x_1^2x_2x_3 \\
\hphantom{{\cal G}_{2143}}{} = \s[101]+\s[201]+\s[111]+\s[211] , \\
 {\cal H}[101] = (1+x_1)(1+x_2){\cal G}[101], \\
{\cal G}_{1342}={\cal G}[011]=x_1x_2+x_1x_3+x_2x_3+2 x_1x_2x_3 = \s[011]+ 2
\s[111], \\
 {\cal H}[011]=(1+x_1)^2(1+x_2)(x_1x_2+x_1x_3+x_2x_3+x_1x_2x_3), \\
{\cal G}_{1423}={\cal G}[02]=x_1^2+x_1x_2+x_2^2 +x_1^2x_2+x_1x_2^2 = \s[02]+\s[12], \\
 {\cal H}[02]=(1+x_1)(1+x_3)
(x_1^2+x_1x_2+x_2^2+2 x_1^2x_2+2 x_1x_2^2+x_1^2x_2^2), \\
{\cal G}_{4123}={\cal G}[3]=x_1^3 = \s[3], \qquad
 {\cal H}[3]=
(1+x_1)^2(1+x_3){\cal G}[3], \\
{\cal G}_{3214}={\cal G}[21]=x_1^2x_2 = \s[21], \qquad
 {\cal H}[21]=
(1+x_1)(1+x_2)(1+x_3){\cal G}[21], \\
{\cal G}_{2341}={\cal G}[111]=x_1x_2x_3 = \s[111], \qquad
 {\cal H}[111]=
(1+x_1)^2(1+x_2){\cal G}[111], \\
{\cal G}_{2413}={\cal G}[12]=x_1^2x_2+x_1x_2^2+x_1^2x_2^2 = \s[12]+
\s[22],\\
 {\cal H}[12]=
(1+x_1)(1+x_3){\cal G}[12], \\
{\cal G}_{1432}={\cal G}[021]=
x_1^2x_2+x_1^2x_3+x_1x_2^2+x_2^2x_3+x_1x_2x_3+2 x_1x_2x_3(x_1+x_2)\\
\hphantom{{\cal G}_{1432}={\cal G}[021]=}{} +x_1^2x_2^2+
x_1^2x_2^2x_3 = \s[021]+ 2 \s[121]+ \s[22]+ \s[211], \\
 {\bf {\cal H}}[021]=(1+x_1){\cal G}[021], \\
{\cal G}_{3142}={\cal G}[201]=x_1^2x_2+x_1^2x_3+x_1^2x_2x_3 =\s[201]+
\s[211], \\
 {\cal H}[201]=(1+x_1)(1+x_2){\cal G}[201], \\
{\cal G}_{4213}={\cal G}[31]=x_1^3x_2 = \s[31], \qquad
 {\cal H}[31]=(1+x_2)(1+x_3){\cal G}[31], \\
{\cal G}_{3412}={\cal G}[22]=x_1^2x_2^2 = \s[22], \qquad
 {\cal H}[22]=(1+x_1)(1+x_3){\cal G}[22], \\
{\cal G}_{4132}={\cal G}[301]=x_1^3x_2+x_1^3x_3+x_1^3x_2x_3 = \s[301]+
\s[311], \\
 {\cal H}[301]=
(1+x_2){\cal G}[301], \\
{\cal G}_{3241}={\cal G}[211]=x_1^2x_2x_3 = \s[211], \qquad
 {\cal H}[211]=(1+x_2){\cal G}[211], \\
{\cal G}_{2431}={\cal G}[121]=x_1^2x_2x_3+x_1x_2^2x_3+x_1^2x_2^2x_3 = \s[121] + \s[221], \\
 {\cal H}[121]=(1+x_1)(1+x_2){\cal G}[121], \\
{\cal G}_{4312}={\cal G}[32]=x_1^3x_2^2 = \s[32], \qquad
 {\cal H}[32]= (1+x_3){\cal G}[32], \\
{\cal G}_{4231}={\cal G}[311]=x_1^3x_2x_3 = \s[311], \qquad
 {\cal H}[311]=
(1+x_2){\cal G}[311], \\
{\cal G}_{3421}={\cal G}[221]=x_1^2x_2^2x_3 = \s[221], \qquad
 {\cal H}[221]=
(1+x_1){\cal G}[221], \\
{\cal G}_{4321}={\cal G}[321]=x_1^3x_2^2x_3= \s[321] = {\cal H}[321].
\end{gather*}
Clearly that any $\beta$-Grothendieck polynomial is a linear combination of
Schubert polynomials with coef\/f\/icients from the ring~$ \N[\beta]$.

$(4)$ Key and reduced key Grothendieck polynomials:
\begin{gather*}
\operatorname{KG}[0] = 1 = {\widehat {\operatorname{KG}}}[0], \qquad
\operatorname{KG}[1] = x_1= {\widehat {\operatorname{KG}}}[1], \qquad
\operatorname{KG}[01] = x_1+x_2+x_1x_2,\\
 {\widehat {\operatorname{KG}}}[01] = x_2+x_1x_2, \qquad
\operatorname{KG}[001] = x_1+x_2+x_3+x_1x_2+x_1x_3+x_2x_3+x_1x_2x_3,\\
 {\widehat {\operatorname{KG}}}[001] = x_3+x_1x_3+x_2x_3+x_1x_2x_3, \qquad
\operatorname{KG}[2] = x_1^2 = {\widehat {\operatorname{KG}}}[2], \\
\operatorname{KG}[11] = x_1x_2 = {\widehat {\operatorname{KG}}}[11], \qquad
\operatorname{KG}[101] = x_1x_2+x_1x_3+x_1x_2x_3,\\
 {\widehat {\operatorname{KG}}}[101] =
x_1x_3+x_1x_2x_3, \qquad
\operatorname{KG}[02] = x_1^2+x_1x_2+x_2^2+x_1^2x_2+x_1x_2^2,\\
 {\widehat {\operatorname{KG}}}[02] = x_1x_2+x_2^2+x_1^2x_2+x_1x_2^2, \\
\operatorname{KG}[011] = x_1x_2+x_1x_3+x_2x_3+2 x_1x_2x_3, \qquad {\widehat {\operatorname{KG}}}[011] =
x_2x_3+x_1x_2x_3, \\
\operatorname{KG}[3] = x_1^3 = {\widehat {\operatorname{KG}}}[3], \qquad
\operatorname{KG}[21] = x_1^2x_2 = {\widehat {\operatorname{KG}}}[21], \qquad
\operatorname{KG}[111] = x_1x_2x_3 = {\widehat {\operatorname{KG}}}[111], \\
\operatorname{KG}[12] = x_1^2x_2+x_1x_2^2+x_1^2x_2^2, \qquad {\widehat {\operatorname{KG}}}[12] =
x_1x_2^2+x_1^2x_2^2, \\
\operatorname{KG}[201] = x_1^2x_2+x_1^2x_3+x_1^2x_2x_3,\qquad {\widehat {\operatorname{KG}}}[201] =
x_1^2x_3+x_1^2x_2x_3, \\
\operatorname{KG}[021] = x_1^2x_2+x_1^2x_3+x_1x_2^2+x_1x_2x_3+x_2^2x_3+2 x_1^2x_2x_3+
2 x_1x_2^2x_3+x_1^2x_2^2+x_1^2x_2^2x_3, \\
 {\widehat {\operatorname{KG}}}[021] = x_1x_2x_3+x_2^2x_3+x_1^2x_2x_3+2 x_1x_2^2x_3+
x_1^2x_2^2x_3, \\
\operatorname{KG}[31] = x_1^3x_2 = {\widehat {\operatorname{KG}}}[31], \qquad
\operatorname{KG}[22] = x_1^2x_2^2 = {\widehat {\operatorname{KG}}}[22], \\
\operatorname{KG}[211] = x_1^2x_2x_3 = {\widehat {\operatorname{KG}}}[211], \qquad
\operatorname{KG}[301] = x_1^3x_2+x_1^3x_3+x_1^3x_2x_3,\\
 {\widehat {\operatorname{KG}}}[301] =
x_1^3x_3+x_1^3x_2x_3, \qquad
\operatorname{KG}[121] = x_1^2x_2x_3+x_1x_2^2x_3+x_1^2x_2^2x_3, \\
{\widehat {\operatorname{KG}}}[121] =
x_1x_2^2x_3+x_1^2x_2^2x_3, \qquad
\operatorname{KG}[32] = x_1^3x_2^2 = {\widehat {\operatorname{KG}}}[32], \\
\operatorname{KG}[311] = x_1^3x_2x_3 = {\widehat {\operatorname{KG}}}[311], \qquad
\operatorname{KG}[221] = x_1^2x_2^2x_3 = {\widehat {\operatorname{KG}}}[221], \\
\operatorname{KG}[321] = x_1^3x_2^2x_3 = {\widehat {\operatorname{KG}}}[321].
\end{gather*}

$(5)$ $42$ (deformed) double key polynomials for $n=4$:
\begin{gather*}
{\cal K}_{\rm id}= 1, \qquad
{\cal K}_{1}= p_{1,1}, \qquad
{\cal K}_{2}= p_{1,2}+p_{2,1}, \qquad
{\cal K}_{3}= p_{1,3}+p_{2,2}+p_{3,1}, \\
{\cal K}_{12}= p_{1,1} p_{2,1}, \qquad
{\cal K}_{21}= p_{1,2} p_{1,1}, \qquad
{\cal K}_{23}= p_{1,2} p_{2,2}+p_{1,2} p_{3,1}+p_{2,1} p_{3,1}, \\
{\cal K}_{32}= p_{1,3} p_{1,2}+p_{1,3} p_{2,1}+p_{2,2} p_{2,1}, \qquad
{\cal K}_{13}= p_{1,1} p_{2,2}+p_{1,1} p_{3,1}, \qquad
{\cal K}_{31}= p_{1,3} p_{1,1}, \\
{\cal K}_{22}= p_{1,2} p_{2,1}, \qquad
{\cal K}_{33}= p_{1,3} p_{2,2}+p_{1,3} p_{3,1}+p_{2,2} p_{3,1}, \qquad
{\cal K}_{123}= p_{1,1} p_{2,1} p_{3,1}, \\
{\cal K}_{133}= p_{1,1} p_{2,2} p_{3,1}, \qquad
{\cal K}_{212}= p_{1,2} p_{1,1} p_{2,1}, \qquad
{\cal K}_{213}= p_{1,2} p_{1.1} p_{2,2} +p_{1,2} p_{1,1} p_{3,1}, \\
{\cal K}_{223}= p_{1,2} p_{2,1} p_{3,1}, \qquad
{\cal K}_{233}= p_{1,2} p_{2,2} p_{3,1}, \qquad
{\cal K}_{321}= p_{1,3} p_{1,2} p_{1,1}, \\
{\cal K}_{312}= p_{1,3} p_{1,1} p_{2,1}+\boldsymbol{q_{13}^{-1}} p_{1,1} p_{2,2} p_{2,1}, \qquad
{\cal K}_{313}= p_{1,3} p_{1,1} p_{2,2}+p_{1,3} p_{1,1} p_{3,1}, \\
{\cal K}_{322}= p_{1,3} p_{1,2} p_{2,1}+\boldsymbol{q_{23}^{-1}} p_{1,2} p_{2,2} p_{2,1}, \\
{\cal K}_{323}= p_{1,3} p_{1,2} p_{2,2}+p_{1,3} p_{1,2} p_{3,1}+p_{1,3} p_{2,1} p_{3,1}+p_{2,2} p_{2,1} p_{3,1}+\boldsymbol{q_{23}} p_{1,3} p_{2,2} p_{2,1} \\
{\cal K}_{333}= p_{1,3} p_{2,2} p_{3,1}, \qquad
{\cal K}_{2123}= p_{1,2} p_{1,1} p_{2,1} p_{3,1}, \qquad
{\cal K}_{2132}= p_{1,2} p_{1,1} p_{2,2} p_{2,1}, \\
{\cal K}_{2133}= p_{1,2} p_{1,1} p_{2,2} p_{3,1}, \qquad
{\cal K}_{3123}= p_{1,3} p_{1,1} p_{2,1} p_{3,1}+\boldsymbol{q_{13}^{-1}} p_{1,1} p_{2,2} p_{2,1} p_{3,1}, \\
{\cal K}_{3132}= p_{1,3} p_{1,1} p_{2,2} p_{2,1}, \qquad
{\cal K}_{3133}= p_{1,3} p_{1,2} p_{2,2} p_{3,1}, \qquad
{\cal K}_{3212}= p_{1,3} p_{1,2} p_{1,1} p_{2,1}, \\
{\cal K}_{3213}= p_{1,3} p_{1,2} p_{1,1} p_{2,2}+p_{1,3} p_{1,2} p_{1,1} p_{3,1}, \qquad
{\cal K}_{3223}= p_{1,3} p_{1,2} p_{2,1} p_{3,1}+\boldsymbol{q_{23}^{-1}} p_{1,2} p_{2,2} p_{2,1} p_{3,1}, \\
{\cal K}_{3232}= p_{1,3} p_{1,2} p_{2,2} p_{2,1}, \qquad
{\cal K}_{3233}= p_{1,3} p_{1,2} p_{2,2} p_{3,1}+\boldsymbol{q_{23}} p_{1,3} p_{2,2} p_{2,1} p_{3,1}, \\
{\cal K}_{21323}= p_{1,2} p_{1,1} p_{2,2} p_{2,1} p_{3,1}, \qquad
{\cal K}_{31323}= p_{1,3} p_{1,1} p_{2,2} p_{2,1} p_{3,1}, \\
{\cal K}_{32123}= p_{1,3} p_{1,2} p_{1,1} p_{2,1} p_{3,1}, \qquad
{\cal K}_{32132}= p_{1,3} p_{1,2} p_{1,1} p_{2,2} p_{2,1}, \\
{\cal K}_{32133}= p_{1,3} p_{1,2} p_{1,1} p_{2,2} p_{3,1}, \qquad
{\cal K}_{32323}= p_{1,3} p_{1,2} p_{2,2} p_{2,1} p_{3,1}, \\
{\cal K}_{321323}= p_{1,3} p_{1,2} p_{1,1} p_{2,2} p_{2,1} p_{3,1}.
\end{gather*}

\begin{Theorem}[cf.~\protect{\cite[the case $\beta=-1$]{L2}}]\label{theorem7.2}
Each double $\beta$-Grothendieck polynomials is a linear combination of double
 key polynomials with the coefficients from the ring~$\N[\beta]$.
\end{Theorem}

Let us remind that the total number of double key polynomials is equal to the
number of alternating sign matrices. We expect that the interrelations between
double key polynomials which follow from the structure of the plactic algebra
 ${\cal{PC}}_n$, see Section~\ref{section5.1}, can be identif\/ied with the graph
corresponding to the MacNeile completion of the poset associated with the
Bruhat order on the symmetric group ${\mathbb{S}}_n$, see Section~\ref{section7.2} for a~def\/inition of the MacNeile completion. It is an interesting problem to
describe interrelation graph associated with the (rectangular) key polynomials corresponding to the Cauchy kernel for the algebra~${\cal{PF}}_{n,m}$.

$(6)$ $26$ double key Grothendieck polynomials for $n=4$:
\begin{gather*}
{\cal {GK}}_{\rm id}=1, \qquad
{\cal {GK}}_{1}= p_{1,1} ={\cal K}_{1} \qquad
{\cal {GK}}_{2}= p_{1,2}+p_{2,1}+p_{1,2} p_{2,1}= {\cal K}_{2}+ {\cal K}_{22}, \\
{\cal {GK}}_{3}= p_{1,3}+p_{1,2}+p_{3,1}+ p_{1,3} p_{2,2}+p_{1,3} p_{3,1}+
p_{2,2} p_{3,1}+ p_{1,3} p_{2,2} p_{3,1} = {\cal K}_3+{\cal K}_{33}+{\cal K}_{333}, \\
{\cal {GK}}_{12}= p_{1,1} p_{2,1}={\cal K}_{12}, \qquad
{\cal {GK}}_{21}= p_{2,1} p_{1,1} = {\cal K}_{21}, \\
{\cal {GK}}_{13}= p_{1,1} p_{2,2}+p_{1,1} p_{3,1}+p_{1,1} p_{2,2} p_{3,1} =
{\cal K}_{13}+{\cal K}_{133}, \qquad
{\cal {GK}}_{31}= p_{3,1} p_{1,1} = {\cal K}_{31}, \\
{\cal {GK}}_{23}= p_{1,1} p_{2,2}+p_{1,2} p_{3,1}+p_{2,1} p_{3,1}+p_{1,2} p_{2,1} p_{3,1}+p_{1,2} p_{2,2} p_{3,1} = {\cal K}_{23}+{\cal K}_{223}+{\cal K}_{233}, \\
{\cal {GK}}_{32}= p_{1,3} p_{1,2}+p_{1,3} p_{2,1}+p_{2,2} p_{2,1}+p_{1,3} p_{1,2} p_{2,1}+p_{1,2} p_{2,1} p_{2,2}={\cal K}_{32}+{\cal K}_{322}, \\
{\cal {GK}}_{123}= p_{1,1} p_{2,1} p_{3,1}= {\cal K}_{123}, \qquad
{\cal {GK}}_{212}= p_{1,2} p_{1,1} p_{2,1} = {\cal K}_{212}, \\
{\cal {GK}}_{213}= p_{1,2} p_{1,1} p_{2,2}+ p_{2,1} p_{1,1} p_{3,1} +
p_{1,2} p_{1,1} p_{2,2} p_{3,1} = {\cal K}_{213}+{\cal K}_{2133}, \\
{\cal {GK}}_{312}= p_{1,3} p_{1,1} p_{2,1}+p_{1,1} p_{2,2} p_{2,1}+
p_{1,2} p_{1,1} p_{2,2} p_{2,1}={\cal K}_{312} +{\cal K}_{2132}, \\
{\cal {GK}}_{313}= p_{1,3} p_{1,1} p_{2,2}+p_{1,3} p_{1,1} p_{3,1}+
p_{1,3} p_{1,1} p_{2,2} p_{3,1} = {\cal K}_{313}+{\cal K}_{3133}, \\
{\cal {GK}}_{321}= p_{1,1} p_{1,2} p_{1,3} = {\cal K}_{123}, \\
{\cal {GK}}_{323}= p_{1,3} p_{1,2} p_{2,2}+p_{1,3} p_{1,2} p_{3,1}+
p_{1,3} p_{2,1} p_{3,1}+p_{2,2} p_{2,1} p_{3,1}
+p_{1,3} p_{2,2} p_{2,1} \\
\hphantom{{\cal {GK}}_{323}=}{}
+p_{1,3} p_{1,2} p_{2,1} p_{2,2}+p_{1,2} p_{2,1} p_{2,2} p_{3,1}+
p_{1,3} p_{1,2} p_{2,2} p_{3,1}+p_{1,3} p_{1,2} p_{2,1} p_{3,1} \\
\hphantom{{\cal {GK}}_{323}=}{}
+p_{1,2} p_{2,2} p_{2,1} p_{3,1}+p_{1,3} p_{1,2} p_{2,2} p_{2,1} p_{3,1} =
{\cal K}_{323}+{\cal K}_{3232}+{\cal K}_{3233}+{\cal K}_{3223}+{\cal K}_{32323}, \\
{\cal {GK}}_{2123}= p_{1,2} p_{1,1} p_{2,1} p_{3,1} = {\cal K}_{2123}, \qquad
{\cal {GK}}_{2132}= p_{1,2} p_{1,1} p_{2,2} p_{2,1} ={\cal K}_{2132}, \\
{\cal {GK}}_{3123}= p_{1,3} p_{1,1} p_{2,1} p_{3,1} p_{1,1} p_{2,2} p_{2,1} p_{3,1}+p_{1,3} p_{1,1} p_{2,2} p_{2,1} p_{3,1}={\cal K}_{3123}+
{\cal K}_{31323}, \\
{\cal {GK}}_{3212}= p_{1,3} p_{1,2} p_{1,1} p_{2,1} = {\cal K}_{3212}, \\
{\cal {GK}}_{3213}= p_{1,3} p_{1,2} p_{1,1} p_{2,2}+p_{1,3} p_{1,2} p_{1,1}
p_{3,1}+p_{1,3} p_{1,2} p_{1,1} p_{2,2} p_{3,1} = {\cal K}_{3213}+
{\cal K}_{32133}, \\
{\cal {GK}}_{21323}= p_{1,2} p_{1.1} p_{2,2} p_{2,1} p_{3,1}={\cal K}_{21323}, \qquad
{\cal {GK}}_{32123}= p_{1,3} p_{1,2} p_{1,1} p_{2,1} p_{31} ={\cal K}_{32123}, \\
{\cal {GK}}_{32132}= p_{1,3} p_{1,2} p_{1,1} p_{2,2} p_{2,1} ={\cal K}_{32132}, \qquad
{\cal {GK}}_{321323}= p_{3,1} p_{2,1} p_{1,1} p_{2,2} p_{2,1} p_{3,1} ={\cal K}_{321323}.
\end{gather*}

$(7)$ $14$ double local key polynomials for $n=4$:
\begin{gather*}
{\cal{LK}}_{\rm id}=1, \qquad
{\cal{LK}}_{1}={\cal{K}}_{1}, \qquad
{\cal{LK}}_{2}={\cal{K}}_{2}, \qquad
{\cal{LK}}_{3}={\cal{K}}_{3}, \qquad
{\cal{LK}}_{12}={\cal{K}}_{12}, \\
{\cal{LK}}_{21}={\cal{K}}_{21} + {\cal{K}}_{212}, \qquad
{\cal{LK}}_{13}={\cal{K}}_{13} + {\cal{K}}_{31}+ {\cal{K}}_{313}, \qquad
{\cal{LK}}_{23}={\cal{K}}_{23}, \\
{\cal{LK}}_{32}= {\cal{K}}_{32} + {\cal{K}}_{323}, \qquad
{\cal{LK}}_{123}={\cal{K}}_{123}, \qquad
{\cal{LK}}_{213}={\cal{K}}_{213} + {\cal{K}}_{2123}, \\
{\cal{LK}}_{312}={\cal{K}}_{312} + {\cal{K}}_{3123}, \qquad
{\cal{LK}}_{321}={\cal{K}}_{321} + {\cal{K}}_{3212}+ {\cal{K}}_{3213}+ {\cal{K}}_{32123}+ {\cal{K}}_{32132}+ {\cal{K}}_{321323}, \\
{\cal{LK}}_{2132}={\cal{K}}_{2132} + {\cal{K}}_{21323}.
\end{gather*}

$(8)$ $35$ $(2,3)$-key polynomials:
\begin{gather*}
 U_{\rm id} =1, \qquad
U_{1} = p_{11}+ p_{23}, \qquad
U_{2} = p_{12}+ p_{21}, \qquad
U_{3} = p_{13}+ p_{22}, \qquad
U_{11} = p_{11} p_{23}, \\
U_{12} = p_{11} p_{21}, \qquad
U_{13} = p_{11} p_{22}, \qquad
U_{21} = p_{12} p_{11} + p_{12} p_{23} + p_{21} p_{23}, \qquad
U_{23} = p_{12} p_{22}, \\
U_{22} = p_{12} p_{21}, \qquad
U_{31} = p_{13} p_{11} + p_{13} p_{23} + p_{22} p_{23}, \qquad
U_{32} = p_{13} p_{12} + p_{13} p_{21} + p_{22} p_{21}, \\
U_{33} = p_{13} p_{22}, \qquad
U_{211} = p_{12} p_{11} p_{23} + p_{11} p_{21} p_{23},\qquad
U_{212} = p_{12} p_{11} p_{21} + p_{12} p_{21} p_{23}, \\
U_{213} = p_{12} p_{11} p_{22} + p_{12} p_{22} p_{23},\qquad
U_{311} = p_{13} p_{11} p_{23} + p_{11} p_{22} p_{23},\\
U_{312} = p_{13} p_{11} p_{21} + p_{11} p_{22} p_{21},\qquad
U_{313} = p_{13} p_{11} p_{22} + p_{13} p_{22} p_{23},\\
U_{321} = p_{13} p_{12} p_{11} + p_{13} p_{12} p_{23} + p_{13} p_{21} p_{23}
+ p_{22} p_{21} p_{23},\qquad
U_{322} = p_{13} p_{12} p_{21} + p_{12} p_{22} p_{21},\\
U_{323} = p_{13} p_{12} p_{22} + p_{13} p_{22} p_{21},\qquad
U_{2121} = p_{12} p_{11} p_{21} p_{23},\qquad
U_{2131} = p_{12} p_{11} p_{22} p_{23},\\
U_{2132} = p_{12} p_{11} p_{22} p_{23}, \qquad
U_{3132} = p_{13} p_{11} p_{22} p_{23}, \qquad
U_{3131} = p_{13} p_{11} p_{22} p_{23},\\
U_{3232} = p_{13} p_{21} p_{22} p_{23}, \qquad
U_{3211} = p_{13} p_{12} p_{11} p_{23} + p_{13} p_{11} p_{21} p_{23}
+p_{11} p_{22} p_{21} p_{23},\\
U_{3212} = p_{13} p_{12} p_{11} p_{22} + p_{13} p_{12} p_{21} p_{23}
+p_{12} p_{21} p_{22} p_{23},\\
U_{3213} = p_{13} p_{12} p_{11} p_{21} + p_{13} p_{22} p_{21} p_{23} + p_{12} p_{22} p_{21} p_{23},\\
U_{32121} = p_{12} p_{11} p_{22} p_{21} p_{23} + p_{13} p_{12} p_{11} p_{21} p_{23},\\
U_{32131} = p_{13} p_{12} p_{11} p_{22} p_{23} + p_{13} p_{11} p_{22} p_{21} p_{23},\\
U_{32132} = p_{13} p_{12} p_{11} p_{22} p_{21} + p_{13} p_{12} p_{22} p_{21} p_{23},\qquad
U_{321321} = p_{13} p_{12} p_{11} p_{22} p_{21} p_{23}.
\end{gather*}

$(9)$ Polynomials ${\cal{KN}}_{w}: ={\cal{KN}}_{w}^{(\beta,\alpha)} (1)$ for $n =4$:
\begin{gather*}
{\cal{KN}}_{\rm id} = 1, \qquad
{\cal{KN}}_{1}={\cal{KN}}_{2}= {\cal{KN}}_{3} = \beta+1+\alpha \beta, \\
{\cal{KN}}_{12} = 1 + 2 \alpha + \alpha^2 + 3 \alpha \beta + 3 \alpha^2 \beta + \alpha \beta^2 + 2 \alpha^2 \beta^2,\quad (13),\\
{\cal{KN}}_{21} = 2 + 3 \alpha + \alpha^2 + \beta + 3\alpha \beta +
2 \alpha^2 \beta + \alpha^2 \beta^2, \quad (13),\\
{\cal{KN}}_{13} = 1 + 2 \alpha + \alpha^2 + 2 \alpha \beta + 2 \alpha^2 \beta + \alpha^2 \beta^2 =
(1 + \alpha + \alpha \beta)^2, \quad (9),\\
{\cal{KN}}_{23} = {\cal{KN}}_{12},\qquad
{\cal{KN}}_{32} = {\cal{KN}}_{21},\\
{\cal{KN}}_{132} = 2 + 5 \alpha + 4 \alpha^2 + \alpha^3 + \beta + 7 \alpha \beta + 10 \alpha^2 \beta + 4 \alpha^3 \beta + 2 \alpha\beta^2 + 7 \alpha^2 \beta^2 + 5 \alpha^3 \beta^2 + \alpha^2 \beta^3 \\
\hphantom{{\cal{KN}}_{132} =}{}
+ 2 \alpha^3 \beta^3 =
(1 + \alpha + \alpha \beta)\big(2 + 3 \alpha + \alpha^2 + \beta + 4 \alpha \beta
+ 3 \alpha^2 \beta + \alpha \beta^2 + 2 \alpha^2 \beta^2\big),\quad \!\! (51),\\
{\cal{KN}}_{121} = 1 + 3 \alpha + 3 \alpha^2 + \alpha^3 + 4 \alpha \beta + 7 \alpha^2 \beta + 3 \alpha^3 \beta + \alpha \beta^2 + 4 \alpha^2\beta^2 +
3 \alpha^3 \beta^2 + \alpha^3 \beta^3, \quad\!\!\! (31),\\
{\cal{KN}}_{321} = 5 + 10 \alpha + 6 \alpha^2 + \alpha^3 + 5 \beta +
14 \alpha \beta + 12 \alpha^2 \beta + 3 \alpha^3 \beta + \beta^2 +
4 \alpha \beta^2 + 6 \alpha^2 \beta^2 \\
\hphantom{{\cal{KN}}_{321} =}{}
+ 3 \alpha^3 \beta^2 + \alpha^3 \beta^3, \quad (71),\\
{\cal{KN}}_{232} = {\cal{KN}}_{121},\\
{\cal{KN}}_{123} = 1 + 3 \alpha + 3 \alpha^2 + \alpha^3 + 6 \alpha \beta +
12 \alpha^2 \beta + 6 \alpha^3 \beta + 4 \alpha \beta^2 + 14 \alpha^2 \beta^2 + 10 \alpha^3 \beta^2 + \alpha \beta^3 \\
\hphantom{{\cal{KN}}_{123} =}{}
+ 5 \alpha^2 \beta^3 + 5 \alpha^3 \beta^3 = \beta^3 \alpha^3 {\cal{KN}}_{321}^{(\beta^{-1},\alpha^{-1})}(1),\quad (71),\\
{\cal{KN}}_{213} = (\alpha \beta)^3 {\cal{KN}}_{132}^{(\alpha^{-1},\beta^{-1})}, \\
{\cal{KN}}_{3121} = 3 + 10 \alpha + 12 \alpha^2 + 6 \alpha^3 + \alpha^4 + 2 \beta + 16 \alpha \beta + 29 \alpha^2 \beta + 19 \alpha^3 \beta +
4 \alpha^4 \beta + 7 \alpha \beta^2\\
\hphantom{{\cal{KN}}_{3121} =}{}
+ 21 \alpha^2 \beta^2 +
20 \alpha^3 \beta^2 + 6 \alpha^4 \beta^2 + \alpha \beta^3 + 4 \alpha^2 \beta^3 + 7 \alpha^3 \beta^3 + 4 \alpha^4 \beta^3 + \alpha^4 \beta^4, \quad (173),\\
{\cal{KN}}_{2321} = (\alpha \beta)^4 {\cal{KN}}_{3121}^{(\alpha^{-1},\beta^{-1})},\\
{\cal{KN}}_{1213} =1 + 4 \alpha + 6 \alpha^2 + 4 \alpha^3 + \alpha^4 + 7 \alpha \beta + 20 \alpha^2 \beta + 19 \alpha^3 \beta + 6 \alpha^4 \beta + 4 \alpha \beta^2 \\
\hphantom{{\cal{KN}}_{1213} =}{}
+ 21 \alpha^2 \beta^2+ 29 \alpha^3 \beta^2 + 12 \alpha^4 \beta^2 + \alpha \beta^3 + 7 \alpha^2 \beta^3 + 16 \alpha^3 \beta^3 + 10 \alpha^4 \beta^3 \\
\hphantom{{\cal{KN}}_{1213} =}{}
+
 2 \alpha^3 \beta^4 + 3 \alpha^4 \beta^4, \quad (173),\\
{\cal{KN}}_{1232} = (\alpha \beta)^4 {\cal{KN}}_{1213}^{(\alpha^{-1},\beta^{-1})},\\
{\cal{KN}}_{2132} = 3 + 9 \alpha + 10 \alpha^2 + 5 \alpha^3 + \alpha^4 + 3 \beta + 16 \alpha \beta + 28 \alpha^2 \beta + 20 \alpha^3 \beta + 5 \alpha^4 \beta + \beta^2\\
 \hphantom{{\cal{KN}}_{2132} =}{}
+ 7 \alpha \beta^2 + 24 \alpha^2 \beta^2 + 28 \alpha^3 \beta^2 + 10 \alpha^4 \beta^2 + 7 \alpha^2 \beta^3 + 16 \alpha^3 \beta^3 + 9 \alpha^4 \beta^3 + \alpha^2 \beta^4\\
 \hphantom{{\cal{KN}}_{2132} =}{}
 + 3 \alpha^3 \beta^4 + 3 \alpha^4 \beta^4, \quad (209),\\
{\cal{KN}}_{21321} = 3 + 12 \alpha + 19 \alpha^2 + 15 \alpha^3 + 6 \alpha^4 + \alpha^5 + 3 \beta + 21 \alpha \beta + 49 \alpha^2 \beta + 52 \alpha^3 \beta
\\
\hphantom{{\cal{KN}}_{21321} =}{}
 + 26 \alpha^4 \beta+ 5 \alpha^5\beta + \beta^2 + 9 \alpha \beta^2 + 39 \alpha^2 \beta^2 + 64 \alpha^3 \beta^2 + 43 \alpha^4 \beta^2 + 10 \alpha^5 \beta^2
\\
\hphantom{{\cal{KN}}_{21321} =}{}
 + 10 \alpha^2 \beta^3 + 32 \alpha^3 \beta^3 + 32 \alpha^4 \beta^3 +
 10 \alpha^5 \beta^3 + \alpha^2 \beta^4 + 5 \alpha^3 \beta^4 + 9 \alpha^4 \beta^4 \\
\hphantom{{\cal{KN}}_{21321} =}{}
 + 5 \alpha^5 \beta^4 + \alpha^5 \beta^5,\quad (483),\\
 {\cal{KN}}_{12312} = 1 + 5 \alpha + 10 \alpha^2 + 10 \alpha^3 + 5 \alpha^4 + \alpha^5 + 9 \alpha \beta + 32 \alpha^2 \beta + 43 \alpha^3 \beta + 26 \alpha^4 \beta\\
 \hphantom{{\cal{KN}}_{12312} =}{}
+ 6 \alpha^5 \beta + 5 \alpha \beta^2 + 32 \alpha^2 \beta^2 + 64 \alpha^3 \beta^2 + 52 \alpha^4 \beta^2 + 15 \alpha^5 \beta^2 + \alpha \beta^3 \\
 \hphantom{{\cal{KN}}_{12312} =}{}
+ 10 \alpha^2 \beta^3 + 39 \alpha^3 \beta^3 + 49 \alpha^4 \beta^3 + 19 \alpha^5 \beta^3 + 9 \alpha^3 \beta^4 + 21 \alpha^4 \beta^4 + 12 \alpha^5 \beta^4 + \alpha^3 \beta^5\\
\hphantom{{\cal{KN}}_{12312} =}{}
 + 3 \alpha^4 \beta^5 + 3 \alpha^5 \beta^5, \quad (483),\\
{\cal{KN}}_{12321} = 2 + 9 \alpha + 16 \alpha^2 + 14 \alpha^3 + 6 \alpha^4 + \alpha^5 + \beta + 18 \alpha \beta + 54 \alpha^2 \beta + 64 \alpha^3 \beta + 33 \alpha^4 \beta\\
\hphantom{{\cal{KN}}_{12321} =}{}
 + 6 \alpha^5 \beta + 14 \alpha \beta^2 + 65 \alpha^2 \beta^2 + 101 \alpha^3 \beta^2 + 64 \alpha^4 \beta^2 + 14 \alpha^5 \beta^2 + 6 \alpha \beta^3 + 33 \alpha^2 \beta^3\\
\hphantom{{\cal{KN}}_{12321} =}{}
 + 65 \alpha^3 \beta^3 + 54 \alpha^4 \beta^3 + 16 \alpha^5 \beta^3 + \alpha \beta^4 + 6 \alpha^2 \beta^4 + 14 \alpha^3 \beta^4 + 18 \alpha^4 \beta^4 + 9 \alpha^5 \beta^4\\
\hphantom{{\cal{KN}}_{12321} =}{}
 + \alpha^4 \beta^5 + 2 \alpha^5 \beta^5, \quad
 (707),\\
{\cal{KN}}_{121321} = 1 + 6 \alpha + 15 \alpha^2 + 20 \alpha^3 + 15 \alpha^4 +6 \alpha^5 + \alpha^6 + 10 \alpha \beta + 45 \alpha^2 \beta + 81 \alpha^3 \beta+ 73 \alpha^4 \beta \\
\hphantom{{\cal{KN}}_{121321} =}{}
+ 33 \alpha^5 \beta + 6 \alpha^6 \beta + 5 \alpha \beta^2 +44 \alpha^2 \beta^2 + 116 \alpha^3 \beta^2 + 135 \alpha^4 \beta^2 + 73 \alpha^5 \beta^2 + 15 \alpha^6 \beta^2 \\
\hphantom{{\cal{KN}}_{121321} =}{}
+ \alpha \beta^3 + 15 \alpha^2 \beta^3 + 69 \alpha^3 \beta^3 + 116 \alpha^4 \beta^3 + 81 \alpha^5 \beta^3 + 20 \alpha^6 \beta^3 + \alpha^2 \beta^4 + 15 \alpha^3 \beta^4 \\
\hphantom{{\cal{KN}}_{121321} =}{}
+ 44 \alpha^4 \beta^4 + 45 \alpha^5 \beta^4 + 15 \alpha^6 \beta^4 + \alpha^3 \beta^5 + 5 \alpha^4 \beta^5 + 10 \alpha^5 \beta^5 + 6 \alpha^6 \beta^5 + \alpha^6 \beta^6\\
\hphantom{{\cal{KN}}_{121321}}{}
 =
\beta^6 \alpha^6 {\cal{KN}}_{121321}^{(\alpha^{-1},\beta^{-1})}(1), \quad (1145).
\end{gather*}

$(10)$ Polynomials ${\cal{KN}}_{w}^{(\beta,\alpha,\gamma)}:=
{\cal{KN}}_{w}^{(\beta,\alpha,\gamma)}(1)$ for $n=3$:
\begin{gather*}
{\cal{KN}}_{\rm id}^{(\beta,\alpha,\gamma)} =1, \\
{\cal{KN}}_{1}^{(\beta,\alpha,\gamma)} = {\cal{KN}}_{2}^{(\beta,\alpha,\gamma)}
= 1+(\beta+\gamma)(1+\alpha+\gamma),\quad (7), \\
{\cal{KN}}_{12}^{(\beta,\alpha,\gamma)} = 1 + 2 \alpha + \alpha^2 +
3 \alpha \beta + 3 \alpha^2 \beta + \alpha \beta^2 + 2 \alpha^2 \beta^2 +
5 \gamma + 8 \alpha \gamma + 3 \alpha^2 \gamma + 4 \beta \gamma\\
\hphantom{{\cal{KN}}_{12}^{(\beta,\alpha,\gamma)} =}{}
 +
11 \alpha \beta \gamma + 4 \alpha^2 \beta \gamma + \beta^2 \gamma +
4 \alpha \beta^2 \gamma + 9 \gamma^2 + 10 \alpha \gamma^2 +
2 \alpha^2 \gamma^2 + 8 \beta \gamma^2 + 8 \alpha \beta \gamma^2\\
\hphantom{{\cal{KN}}_{12}^{(\beta,\alpha,\gamma)} =}{}
 +
2 \beta^2 \gamma^2 + 7 \gamma^3 + 4 \alpha \gamma^3 + 4 \beta \gamma^3 +
2 \gamma^4, \quad (109),\\
{\cal{KN}}_{21}^{(\beta,\alpha,\gamma)} = 2 + 3 \alpha + \alpha^2 + \beta +
3 \alpha \beta + 2 \alpha^2 \beta + \alpha^2 \beta^2 + 7 \gamma +
8 \alpha \gamma + 2 \alpha^2 \gamma + 4 \beta \gamma + 7 \alpha \beta \gamma\\
\hphantom{{\cal{KN}}_{21}^{(\beta,\alpha,\gamma)} =}{}
 +
 2 \alpha^2 \beta \gamma + 2 \alpha \beta^2 \gamma + 9 \gamma^2 +
7 \alpha \gamma^2 + \alpha^2 \gamma^2 + 5 \beta \gamma^2 +
4 \alpha \beta \gamma^2 + \beta^2 \gamma^2 + 5 \gamma^3 \\
\hphantom{{\cal{KN}}_{21}^{(\beta,\alpha,\gamma)} =}{}
+ 2 \alpha \gamma^3 +
 2 \beta \gamma^3 + \gamma^4,\quad (82),\\
{\cal{KN}}_{121}^{(\beta,\alpha,\gamma)} = 1 + 3 \alpha + 3 \alpha^2 + \alpha^3 + 4 \alpha \beta + 7 \alpha^2 \beta +
 3 \alpha^3 \beta + \alpha \beta^2 + 4 \alpha^2 \beta^2 + 3 \alpha^3 \beta^2 +
 \alpha^3 \beta^3\\
\hphantom{{\cal{KN}}_{121}^{(\beta,\alpha,\gamma)} =}{}
 + 6 \gamma + 15 \alpha \gamma + 12 \alpha^2 \gamma +
 3 \alpha^3 \gamma + 3 \beta \gamma + 20 \alpha \beta \gamma +
 22 \alpha^2 \beta \gamma + 6 \alpha^3 \beta \gamma + 7 \alpha \beta^2 \gamma\\
\hphantom{{\cal{KN}}_{121}^{(\beta,\alpha,\gamma)} =}{}
 +
 12 \alpha^2 \beta^2 \gamma + 3 \alpha^3 \beta^2 \gamma + 3 \alpha^2 \beta^3 \gamma +
 15 \gamma^2 + 30 \alpha \gamma^2 + 18 \alpha^2 \gamma^2 + 3 \alpha^3 \gamma^2 +
 12 \beta \gamma^2 \\
\hphantom{{\cal{KN}}_{121}^{(\beta,\alpha,\gamma)} =}{}
 + 37 \alpha \beta \gamma^2 + 24 \alpha^2 \beta \gamma^2 +
 3 \alpha^3 \beta \gamma^2 + 3 \beta^2 \gamma^2 + 15 \alpha \beta^2 \gamma^2 +
 9 \alpha^2 \beta^2 \gamma^2 + 3 \alpha \beta^3 \gamma^2\\
\hphantom{{\cal{KN}}_{121}^{(\beta,\alpha,\gamma)} =}{}
 + 20 \gamma^3 +
 30 \alpha \gamma^3 + 12 \alpha^2 \gamma^3 + \alpha^3 \gamma^3 + 18 \beta \gamma^3 +
 30 \alpha \beta \gamma^3 + 9 \alpha^2 \beta \gamma^3 + 6 \beta^2 \gamma^3 \\
\hphantom{{\cal{KN}}_{121}^{(\beta,\alpha,\gamma)} =}{}
 +
 9 \alpha \beta^2 \gamma^3 + \beta^3 \gamma^3 + 15 \gamma^4 + 15 \alpha \gamma^4 +
 3 \alpha^2 \gamma^4 + 12 \beta \gamma^4 + 9 \alpha \beta \gamma^4 +
 3 \beta^2 \gamma^4 + 6 \gamma^5\\
\hphantom{{\cal{KN}}_{121}^{(\beta,\alpha,\gamma)} =}{}
 + 3 \alpha \gamma^5 + 3 \beta \gamma^5 + \gamma^6, \quad
 (521).
\end{gather*}

$(11)$ Few more examples:
\begin{gather*}
{\cal{KN}}_{4321}^{(\beta,\alpha,\gamma=0)}(1) = 14 + 35 \alpha + 30 \alpha^2 + 10 \alpha^3 + \alpha^4 + 21 \beta + 65 \alpha \beta + 70 \alpha^2 \beta +
30 \alpha^3 \beta + 4 \alpha^4 \beta\\
\hphantom{{\cal{KN}}_{4321}^{(\beta,\alpha,\gamma=0)}(1) =}{}
 + 9 \beta^2 + 35 \alpha \beta^2 +
50 \alpha^2 \beta^2 + 30 \alpha^3 \beta^2 + 6\alpha^4 \beta^2 + \beta^3 +
 5 \alpha \beta^3 + 10 \alpha^2 \beta^3 \\
\hphantom{{\cal{KN}}_{4321}^{(\beta,\alpha,\gamma=0)}(1) =}{}
 + 10 \alpha^3 \beta^3 +
4 \alpha^4 \beta^3 + \alpha^4 \beta^4,\\
{\cal{KN}}_{4321}^{(\beta=1,\alpha=1,\gamma)}(1)
= (441,1984,3754,3882,2385,885,192,22,1)_{\gamma},\\
{\cal{KN}}_{54321}^{(\beta=1,\alpha=1,\gamma)} = (1 + \gamma)(2955, 13297, 25678, 27822, 18553, 7852, 2094, 336, 29, 1)_{\gamma}.
\end{gather*}
Note that polynomial $L_n(\gamma):= {\cal{KN}}_{n,n-1,\ldots,2,1}^{(\beta=1,\alpha=1,\gamma)}(1)$ has degree $2 n$ and $L_n(\gamma = -1) =0$.
\begin{gather*}
h^3 {\cal{KN}}_{321}^{(\beta=1,\alpha=1,\gamma,h)} = (1,3,7,9,7,1)_{h}+ \gamma (6,18,32,32,18,6)_{h} + \gamma^2 (15,42,58,42,15)_{h} \\
\hphantom{h^3 {\cal{KN}}_{321}^{(\beta=1,\alpha=1,\gamma,h)} =}{}
+\gamma^3 (20,48,48,20)_{h} + \gamma^4 (15,27,15)_{h} + \gamma^5 (6,6)_h +
\gamma^6 h^6,\\
{\cal{KN}}_{121}^{(a=1,b=1,c,r)}(1)= 31 + 112 c + 168 c^2 + 124 c^3 + 44 c^4 + 6 c^5 \\
\hphantom{{\cal{KN}}_{121}^{(a=1,b=1,c,r)}(1)=}{}
+ \big(60 + 176 c + 195 c^2 + 93 c^3 + 16 c^4\big) r + \big(38 + 85 c +
61 c^2 + 14 c^3\big) r^2\\
\hphantom{{\cal{KN}}_{121}^{(a=1,b=1,c,r)}(1)=}{}
 + \big(8 + 12 c + 4 c^2\big) r^3.
\end{gather*}

\begin{prb}\label{prob7.3} Let $n \ge k \ge 0$ be integers, consider permutation $w_{n,k}:=
[k,k-1,\ldots, 1,$ $n,n-1,n-2,\ldots,k+1] \in \mathbb{S}_n$. Give
combinatorial interpretations of polynomials $L_{n,k}(\alpha,\beta,\gamma):=
{\cal{KN}}_{w_{m,k}}^{(\alpha,\beta,\gamma)}(1)$.
\end{prb}

\begin{con}\label{conj7.4} Set $d:= \gamma-1$.
\begin{itemize}\itemsep=0pt
\item For any permutation $w \in \mathbb{S}_n$, ${\cal{KN}}_{w}^{(\beta,\alpha=1,\gamma=d-1)}(1)$ is a polynomial in $\beta$ and $d$ with non-negative
coefficients.

\item The polynomial $L_n(d)$ has non-negative
coefficients, and polynomial $L_n(d)+d^n$ is symmetric and unimodal.

\item $L_{n,1}(\alpha=1,\beta, d) \in d \N[\beta,d]$.

\item $L_{n,1}(\alpha,\beta=0,d) \in d^{n-1} (\alpha + d) \N [\alpha,d]$,
$L_{n,1}(\alpha=1,\beta=0,d=1) =2 \operatorname{Sch}_{n+1}$, $L_{n,1}(\alpha=0,$ $\beta=0,d=2) =
2^n \operatorname{Sch}_{n+1}$ $($see {\rm \cite[$A156017$]{SL}} for a combinatorial interpretation of
 these numbers$)$,
where $\operatorname{Sch}_{n}$ denotes the
$n$-th Schr\"{o}der number, see, e.g., {\rm \cite[$A001003$]{SL}}.

\item $L_{n,1}(\alpha=0,\beta=t-1,\gamma) \in \N [t,\gamma]$,
$L_{n,1}(\alpha=0,\beta= -1,\gamma =1)$ is equal to the number of Dyck
$(n+1)$-paths $(${\rm \cite[$A000108$]{SL})} in which each up step $(U)$ not at
ground level is colored red $(R)$ or blue $(B)$, {\rm \cite[$A064062$]{SL}}.
\end{itemize}
\end{con}

Note that the number $2 \operatorname{Sch}_{n}$ known also as \textit{large Schr\"{o}der} number, see, e.g., \cite[$A006318$]{SL}.

For example,
\begin{gather*}
L_{3,1}(\alpha=1,\beta,d) = d \big(\beta^2 + 5 \beta d + 4 \beta^2 d + 5 d^2
+ 14 \beta d^2 + 6 \beta^2 d^2 + \beta^3 d^2 + 10 d^3 +
 12 \beta d^3 \\
 \hphantom{L_{3,1}(\alpha=1,\beta,d) =}{}
 + 3 \beta^2 d^3 + 6 d^4 + 3 \beta d^4 + d^5\big),\\
L_{7,1}(1,1,d) = d(1,27,260,1245,3375,5495,5494,3375,1245,260,27,1)_{d},\\
L_{7,1}(\alpha,\beta=0,d)= d^6 (\alpha + d) (1 + \alpha + d) \big(1 + 14 \alpha +
36 \alpha^2 + 14 \alpha^3 + \alpha^4 + 14 d + 72 \alpha d \\
 \hphantom{L_{7,1}(\alpha,\beta=0,d)=}{} +
 42 \alpha^2 d + 4 \alpha^3 d + 36 d^2+ 42 \alpha d^2 + 6 \alpha^2 d^2 +
14 d^3 + 4 \alpha d^3 + d^4\big), \\
L_{7,1}(\alpha=0,\beta=t-1,\gamma =1) = (14589,39446,39607,18068,3627,246,1)_{t}.
\end{gather*}

We expect a similar conjecture for polynomials $L_{n,k}(\alpha=1,\beta=1,\gamma)$, $k \ge 1$.

\subsection{MacMeille completion of a partially ordered set}\label{section7.2}

Let\footnote{For the reader convenience we review a def\/inition and basic
facts concerning the MacNeille completion of a~poset, see for example, notes
by E.~Turunen, available at
\url{http://math.tut.fi/~eturunen/AppliedLogics007/Mac1.pdf}.} $(\Sigma, \le)$ be a partially ordered set (poset for short) and
$X \subseteq \Sigma$. Def\/ine
\begin{itemize}\itemsep=0pt
\item The set of upper bounds for $X$, namely,
\begin{gather*}
X^{up} := \{z \in \Sigma \,\vert\, x \le z,\, \forall\, x \in X \}.
\end{gather*}
\item The set of lower bounds for $X$, namely,
\begin{gather*}
X^{\rm lo} : = \{z \in \Sigma \,\vert\, z \le x \, \forall\, x \in X \}.
\end{gather*}
\item A poset $({\cal{MN}}({\Sigma}), \le )$, namely,
\begin{gather*}
{\cal{MN}}({\Sigma}):= \{{\cal{MN}}({X}) \vert X \subseteq \Sigma \},
\end{gather*}
 where ${\cal{MN}}(X):= \big(X^{\rm up}\big)^{\rm lo}$.
Clearly, $ X \subseteq {\cal{MN}}(X)$ and ${\cal{MN}}({\cal{MN}}({X})) =
{\cal{MN}}(X)$.

\item A map $\kappa \colon \Sigma \longrightarrow {\cal{MN}}({\Sigma})$,
namely, $\kappa(X) = {\cal{MN}}(X)$, $X \subseteq \Sigma$.
\end{itemize}

{\samepage \begin{pr}\label{prop7.5} \quad
\begin{itemize}\itemsep=0pt
\item The map $\kappa$ is an embedding, that is for
$X,Y \subseteq \Sigma$,
\begin{gather*}
 X \le Y \qquad \text{if and only if} \quad \kappa(X) \subseteq \kappa(Y),
\end{gather*}
\item Poset $({\cal{MN}}({\Sigma}), \le )$ is a lattice,
called the MacNeille completion of poset $(\Sigma, \le)$.
\end{itemize}
\end{pr}}

\begin{pr}\label{prop7.6}
Let $(\Sigma, \le)$ be a poset\footnote{See \url{https://en.wikipedia.org/wiki/Graded_poset}.}. Then there is a poset
$(L, \le)$ and a map $\kappa\colon \Sigma \longrightarrow L$ such that
\begin{enumerate}\itemsep=0pt
\item[$(1)$] $\kappa$ is an embedding,

\item[$(2)$] $(L, \le)$ is a complete lattice\footnote{That is every subset of $L$ has a meet and join, see, e.g., \cite[p.~249]{ST}.},

\item[$(3)$] for each element $ a \in L$ one has
\begin{enumerate}\itemsep=0pt
\item[$(a)$] ${\cal{MN}}(\{x \in \Sigma \,\vert\, \kappa(x) \le a \}) = \{x \in \Sigma \,\vert\, \kappa(x) \le a \}$,

\item[$(b)$] $a = \bigvee\{\kappa(x) \,\vert\, x \in \Sigma, \,\kappa(x) \le a \}$.
\end{enumerate}
\end{enumerate}
Moreover, the pair $(\kappa, (L, \le))$ is defined uniquely up to an order
preserving isomorphism.
\end{pr}

Therefore, the lattice $(L, \le)$, is an order-isomorphic to the MacNeille
completion ${\cal{MN}}(\Sigma)$ of a poset $\Sigma$.
\begin{prb}\label{prob7.7}
Let $\Sigma$ be a $($finite$)$ graded poset\footnote{See, e.g., \cite[p.~244]{ST}, or
 \url{https://en.wikipedia.org/wiki/Graded_poset}.}, denote by
\begin{gather*}
r_{\Sigma}(t): = \sum_{a \in \Sigma} t^{r(a)},
\end{gather*}
the rank generating function of a poset $\Sigma$.
Here $r(a)$ denotes the rank/degree of an element $a \in \Sigma$.
Describe polynomial $r_{{\cal{MN}}(\Sigma)}(t)$.
\end{prb}

In the present paper we are interesting in properties of the MacNeille
completion of the Bruhat poset ${\cal{B}}_n= {\cal B}({\mathbb{S}}_n)$
corresponding to the symmetric group ${\mathbb{S}}_n$. Below we brief\/ly
describe a~const\-ruc\-tion of the MacNeille completion $L_n({\mathbb{S}}_n) :=
{\cal{MN}}_n({\cal{B}}_n)$ following \cite{LS1} and \cite[p.~552,~{\bf d}]{ST+}.

Let $w=(w_1 w_2 \ldots w_n) \in \mathbb{S}_n$, associate with $w$ a
semistandard Young tableaux $T(w)$ of the staircase shape $\delta_n=(n-1,n-2,
\ldots, 2,1)$ f\/illed by integer numbers from the set $[1,n]:= \{1,2,
\ldots,n \}$ as follows:
the $i$-th row of of $T(w)$, denoted by $R_i(w)$, consists of the numbers
$w_1,\ldots, w_{n-i+1}$ in increasing order. Clearly the tableaux $T(w)=
[T_{i,j}(w)]_{1 \le i <j \le n-1}$ obtained in such a manner, satisf\/ies
the so-called {\it monotonic} and {\it flag} conditions, namely,
\begin{enumerate}\itemsep=0pt
\item[(1)] (monotonic conditions) $T_{1,i} \ge T_{2,i-1} \ge \cdots \ge T_{i,1}$, $i=1,\ldots,n-1$,

\item[(2)] (f\/lag conditions) $R_1(w) \supset R_2(w) \supset \cdots \supset
R_{n-1}(w)$.
\end{enumerate}

Denote by $L(\mathbb{S}_n)$ the subset of the set of all Young tableaux $T \in
\operatorname{STY}(\delta_n \le n)$ consisting of that~$T$ which satisf\/ies the monotonicity
conditions~(1). The set~$L(\mathbb{S}_n)$ has the natural poset
structure denoted by ``$\ge$'', and def\/ined as follows:
 if
$T^{(1)}=[t_{ij}^{(1)}]_{1 \le i < j \le n-1}$ and $T^{(2)}=
[t_{ij}^{(2)}]_{1 \le i < j \le n-1}$ belong to the set $L(\mathbb{S}_n)$,
then by def\/inition
\begin{gather*}
 T^{(1)} \ge T^{(2)} \qquad \text{if and only if} \quad t_{ij}^{(1)} \ge t_{ij}^{(2)} \quad \text{for
all} \quad 1 \le i < j \le n-1.
\end{gather*}
It is clearly seen that the set $L({\mathbb{S}}_n)$ is closed under the
following operations
\begin{itemize}\itemsep=0pt
\item $\big($meet $T^{(1)} T^{(2)}\big)$ $\bigwedge \big(T^{(1)},T^{(2)}\big):=
T^{(1)} \bigwedge T^{(2)} = \big[\min\big(t_{i,j}^{(1)},t_{i,j}^{(2)}\big)\big]$,

\item $\big($join $T^{(1)} T^{(2)}\big)$ $\bigvee\big( (T^{(1)},T^{(2)}\big):= T^{(1)} \bigvee T^{(2)} = \big[\max\big(t_{i,j}^{(1)},t_{i,j}^{(2)}\big)\big]$.
\end{itemize}

\begin{Theorem}[\cite{LS1}]\label{theorem7.8}
The poset $L({\mathbb{S}}_n)$ is a complete distributive
lattice with number of vertices equals to the number $\operatorname{ASM}(n)$ that is the
number of alternating sigh matrices of size $n \times n$. Moreover, the
lattice $L({\mathbb{S}}_n)$ is order isomorphic to the MacNeille completion of
the Bruhat poset~${\cal{B}}_n$.
\end{Theorem}

Indeed it is not dif\/f\/icult to prove that the set of all monotonic triangles
obtained by applying repeatedly operation $\bigvee$ ($=$~meet) to the set
$\{T(w),\, w \in {\mathbb{S}}_n\}$ of triangles corresponding to all elements of
the symmetric group ${\mathbb{S}}_n$, coincides with the set of all monotonic triangles $L({\mathbb{S}_n})$. The natural map $\kappa\colon {\mathbb{S}}_n
\longrightarrow L({\mathbb{S}}_n)$ is obviously embedding, and all other
conditions of Proposition~\ref{prop7.6} are satisf\/ied. Therefore
$L({\mathbb{S}}_n) ={\cal{MN}}({\cal{B}}_n)$. The fact that the lattice
$L({\mathbb{S}}_n)$ is a~distributive one follows from the well-known identities
\begin{gather*}
\max(x,\min(y,z)) = \min\big(\max(x,y),\max(x,z)\big), \qquad x,y,z \in \big(\R_{\ge 0}\big)^3.
\end{gather*}
In the lattice $L({\mathbb{S}}_n$ this identity can be written in the
following forms
\begin{gather*}
 T^{(1)} \bigvee \big(T^{(2)} \bigwedge T^{(3)}\big) = \big(T^{(1)} \bigwedge T^{(2)}\big) \bigvee \big(T^{(1)} \bigwedge T^{(3)}\big), \\
 T^{(1)} \bigwedge \big(T^{(2)} \bigvee T^{(3)}\big) = \big(T^{(1)} \bigvee T^{(2)}\big) \bigwedge \big(T^{(1)} \bigvee T^{(3)}\big).
\end{gather*}
Finally the fact that the cardinality of the lattice $L({\mathbb{S}}_n)$ is
equal to the number $\operatorname{ASM}(n)$ had been proved by A.~Lascoux and M.-P.~Sch\"{u}tzenberger~\cite{LS1}.

If $T=[t_{ij}] \in L(\mathbb{S}_n)$, def\/ine {\it rank} of~$T$, denoted by
$r(T)$, as follows:
\begin{gather*}
r(T) = \sum_{1 \le i < j \le n-1} t_{ij} - {n \choose 3}.
\end{gather*}
It had been proved by C.~Ehresmann \cite{Ehresmann} that
$v \le w$ with respect to the Bruhat order in the symmetric group
$\mathbb{S}_n$ if and only if $T_{i,j}(v) \le T_{i,j}(w)$ for all $1 \le i < j \le n-1$.

It follows from an improved tableau criterion for Bruhat order on the
symmetric group \cite{BB} that\footnote{It has been proved in~\cite[Corollary~5]{BB}, that the Ehresmann
 criterion stated above is equivalent to either the criterion
$T_{i,j}^{(1)} \le T_{i,j}^{(2)}$ for all $j$ such that $w_j > w_{j+1}$
 and $1 \le i \le j$,
or that
$T_{i,j}^{(1)} \le T_{i,j}^{(2)}$ for all $j \in \{1,2,\ldots,n-1 \}
\backslash \{k \,\vert\, v_{k} > v_{k+1} \}$ and $1 \le i \le j$.}
the length $\ell(w)$ of a permutation $w \in \mathbb{S}_n$ can be
computed as follows
\begin{gather*}
\ell(w) = r(T(w)) - \sum_{(i,j) \in I(w)} (j-i - 1),
\end{gather*}
where $I(w):= \{(i,j) \,\vert\, 1 \le i < j \le n, \, w_i > w_j \}$ denotes the set of {\it inversions} of permutation~$w$; a~detailed proof can be found in~\cite{Ko}.

For example, consider permutation $w= [4,6,2,7,5,1,3]$. Then the code $c(w)$
of $w$ is equal to $c(w) = (3,4,1,3,2)$, and $w$ has the length $\ell(w) =
 13$. The corresponding Young tableau or monotonic triangle displayed below
\begin{gather*}
T(w) = \left[ \begin{matrix}
1 & 2 & 4 & 5 & 6 & 7\cr
2 & 4 & 5 & 6 & 7\cr
2 & 4 & 6 & 7 \cr
2 & 4 & 6\cr
4 & 6 \cr
4 \cr
\end{matrix} \right].
\end{gather*}
Wherefore, $r(T(w)) = |T(w)|- {7 \choose 3} =94-56 =38$. On the
other hand, the inversion set $I(w) = \{(1,3),(1,6),(1,7),(2,3),(2,5),(2,6),(2,7),(3,6),(4,5),(4,6),(4,7),(5,6),(5,7) \}$,
hence $\sum\limits_{(i,j) \in I(w)}(j-i-1) = 10+9+2+3+1 = 25$ and $\ell(w)= 38 -25 = 13$, as it should be.

 It is easily seen that the
polynomial $r_{{\cal{MN}}_n}(t)$ is symmetric and
$\deg (r_{{\cal{MN}}_n}(t)) = {n+1 \choose 3}$, For example,
\begin{gather*}
r({\cal{MN}}_3)=(1,2,1,2,1), \qquad r({\cal{MN}}_4)=(1,3,3,5,6,{\bf 6},6,5,3,3,1), \\
r({\cal{MN}}_5)=(1,4,6,10,16,20,27,34,37,40,{\bf 39},40,37,34,27,20,16,10,6,4,1),\\
r(\mathbb{S}_3 \subset {\cal{MN}}_3)=(1,2,0,2,1), \qquad r(\mathbb{S}_4
\subset {\cal{MN}}_4)=(1,3,1,4,2,{\bf 2},2,4,1,3,1)), \\
r(\mathbb{S}_5 \subset {\cal{MN}}_5)=(1,4,3,6,7,6,4,10,6,10,{\bf 6},10,6,10,
4,6,7,6,3,4,1).
\end{gather*}

\begin{con}\label{conj7.9}
The number $\operatorname{Coef\/f}_{[{n+1 \choose 3}/2]} r_{{\cal{MN}}_n}(t)$ is
a divisor of the number $\operatorname{ASM}(n)$.
\end{con}

\subsection*{Acknowledgements}

A bit of history. Originally these notes have been designed as a continuation
of \cite{FK1}. The main purpose was to extend the methods developed in~\cite{FK2} to obtain by the use of plactic algebra, a noncommutative
generating function for the key (or~Demazure) polynomials introduced by A.~Lascoux and M.-P.~Sch\"{u}tzenberger~\cite{LS2}. The results concerning the
polynomials introduced in Section~\ref{section4}, {\it except} the Hecke--Grothendieck polynomials, see Def\/inition~\ref{def4.6}, has been presented in my lecture-courses ``Schubert Calculus'' and have been delivered at the Graduate School
 of Mathematical Sciences, the University of Tokyo, November~1995~-- April~1996, and at the Graduate School of Mathematics, Nagoya University, October~1998~-- April~1999. I~want to thank
Professor M.~Noumi and Professor T.~Nakanishi who made these courses possible.
 Some early versions of the present notes are circulated around the world and now I was asked to put it for the wide audience. I~would like to thank
Professor M.~Ishikawa (Department of Mathematics, Faculty of Education,
University of the Ryukyus, Okinawa, Japan) and Professor S.~Okada (Graduate
School of Mathematics, Nagoya University, Nagoya, Japan) for valuable comments.
My special thanks to the referees for very careful reading of a preliminary version of the present paper and many valuable remarks, comments and
suggestions.

\LastPageEnding

\end{document}